\documentclass{article}
\usepackage[a4paper]{geometry}
\usepackage[utf8]{inputenc}
\usepackage[all]{xy}
\usepackage[initials]{amsrefs}
\usepackage[labelsep=period]{caption}
\usepackage{amsfonts,authblk,bm,doi,enumitem,fancyhdr,float,graphicx,amssymb,amsthm,mathrsfs,mleftright,relsize,sectsty,subcaption,tikz,tikz-3dplot,titlesec,wrapfig,xcolor}
\usepackage[fleqn]{amsmath}

\newtheorem{theorem}{Theorem}[section]
\newtheorem{proposition}[theorem]{Proposition}
\newtheorem{lemma}[theorem]{Lemma}
\newtheorem{corollary}[theorem]{Corollary}
\theoremstyle{definition}
\newtheorem*{definition}{Definition}

\numberwithin{equation}{section}
\numberwithin{figure}{section}

\newcommand{\SL}{\textnormal{\textbf{SL}}_{\mathbf{2}}}
\newcommand{\FR}{\textnormal{\textbf{FR}}_{\mathbf{n}}}
\newcommand{\FRS}{\textnormal{\textbf{FR}}_{\mathbf{n}}^{\mathbf{\ast}}}
\newcommand{\FRT}{\textnormal{\textbf{FR}}_{\mathbf{n}}^{\mathbf{\dagger}}}

\DeclareMathOperator{\cha}{\textnormal{char}}
\DeclareMathOperator{\diam}{\textnormal{diam}}

\DeclareMathOperator{\Ima}{\textnormal{Im}}

\makeatletter
\newcommand{\tpmod}[1]{{\@displayfalse\pmod{#1}}}
\makeatother

\definecolor{col1}{RGB}{48,65,93} 
\definecolor{col2}{RGB}{207,103,102} 
\definecolor{col3}{RGB}{3,20,36} 
\definecolor{col4}{RGB}{142,174,189} 
\definecolor{grey}{RGB}{210,210,210}
\definecolor{darkgrey}{RGB}{129,132,136}
\definecolor{col5}{RGB}{140, 0, 26}
\definecolor{col6}{RGB}{165, 0, 33} 
\definecolor{col7}{RGB}{27, 77, 62}
\renewcommand{\geq}{\geqslant}
\renewcommand{\leq}{\leqslant}

\usetikzlibrary{calc}
\usetikzlibrary{cd}
\usetikzlibrary{matrix}
\usetikzlibrary{decorations.markings}
\usetikzlibrary{arrows.meta}
\usetikzlibrary{arrows,shapes,positioning}
\tikzcdset{every label/.append style = {font = \normalsize}}
\tikzstyle arrowstyle=[scale=1]
\tikzstyle directed=[postaction={decorate,decoration={markings,
		mark=at position 0.5 with {\arrow[arrowstyle]{{latex}}}}}]
\tikzstyle reverse directed=[postaction={decorate,decoration={markings,
		mark=at position 0.5 with {\arrow[arrowstyle]{{latex reversed}}}}}]

\tikzset{pointer/.style 2 args={draw,fill,single arrow,
    single arrow tip angle=45,
    single arrow head extend=#1,
    single arrow head indent=0pt,
    inner sep=0pt,
    rotate=#2}}

\newcount\recurdepth
\def\fareyrecur[#1]#2#3#4#5#6#7{
  \recurdepth=#6
  \ifnum\the\recurdepth>1\relax
    \advance\recurdepth by-1\relax
    \edef\tempnum{\number\numexpr#2+#4\relax}
    \edef\tempden{\number\numexpr#3+#5\relax}
    \pgfmathparse{\tempnum/\tempden}\edef\temp{\pgfmathresult}

    \ifnum \the\recurdepth >3 
      \ifnum\tempnum>0
        \node[below=1pt,scale=1]at({(\temp)},0){$\frac{\tempnum}{\tempden}$};
      \else 
        \ifnum\the\recurdepth>4\relax 
          \edef\abstempnum{\number\numexpr -\tempnum\relax}
          \node[below=1pt,scale=1]at({(\temp)},0){$-\frac{\abstempnum}{\tempden}\phantom{-}$};
        \fi 
      \fi 
      \draw[#1] ({(\temp)},0) arc (180:0:{((#4/#5)-\temp)*0.5});
      \draw[#1] ({(\temp)},0) arc (0:180:{(\temp-(#2/#3))*0.5});
    \fi

    \ifnum \the\recurdepth <4 
      \draw[#1] ({(\temp)},0) arc (180:0:{((#4/#5)-\temp)*0.5});
      \draw[#1] ({(\temp)},0) arc (0:180:{(\temp-(#2/#3))*0.5});
    \fi

    \begingroup
      \edef\ttempup{\noexpand\fareyrecur[#1]{\tempnum}{\tempden}{#4}{#5}{\the\recurdepth}{#7}}
      \edef\ttempdown{\noexpand\fareyrecur[#1]{#2}{#3}{\tempnum}{\tempden}{\the\recurdepth}{#7}}
      \ttempup\ttempdown
    \endgroup
  \fi
}

\def\fareygraph[#1]#2#3#4{
  \draw[#1] ({#2},0) -- ({#3+1},0);
  \foreach \n in {#2,...,#3}{
    \draw[#1] ({\n},0.75) -- ({\n},0);
    \ifnum #4=0 
      \node[below=1pt,scale=1,black] at ({\n},0) {$\frac{\n}{1}$};	
    \fi
    \draw[#1] ({(\n)},0) arc (180:0:0.5);
    \fareyrecur[#1]{\n}{1}{\n+1}{1}{7}{#4}	
    \foreach \k in {8,...,25}{
      \draw[#1] ({\n},0) arc (0:180:{0.5/\k}); \draw[#1] ({\n-1/\k},0) arc(0:180:{0.5/((\k-1)*\k)});
      \draw[#1] ({\n},0) arc (180:0:{0.5/\k}); \draw[#1] ({\n+1/\k},0) arc(180:0:{0.5/((\k-1)*\k)});
    }
  }
  \edef\nplusone{\number\numexpr#3+1\relax}
  \draw[#1] ({\nplusone},0.75) -- ({\nplusone},0) node[below=4pt,scale=1] {$\frac{\nplusone}{1}$};	
}

\setlist[enumerate]{leftmargin=20pt,itemsep=0pt,topsep=0pt}
\setlist[enumerate,1]{label=\emph{(\roman*)}}
\setlist[itemize]{leftmargin=20pt,itemsep=0pt,topsep=0pt}

\makeatletter
\renewcommand\section{\@startsection {section}{1}{\z@}%
                                   {-3.5ex \@plus -1ex \@minus -.2ex}%
                                   {1.3ex \@plus.2ex}%
                                   {\normalfont\large}}
\makeatother

\AtEndDocument{%
	\par
	\medskip
	\begin{tabular}{@{}l@{}}%
		{School of Mathematics and Statistics, The Open University,}\\
		{Milton Keynes, MK7 6AA, United Kingdom}\\
		\textit{E-mail address}: \texttt{ian.short@open.ac.uk}
\end{tabular}}

\title{ \vspace{-6ex}\bf \large Frieze patterns and Farey complexes\footnotetext{
		
		
		Research supported by EPSRC grants EP/W002817/1 (IS and MvS) and EP/W524098/1 (AZ).
		
		There is no data associated with this article.}}

\author{\normalsize Ian Short, Matty Van Son, and Andrei Zabolotskii
}

\date{\vspace{-3ex}}

\begin{document}

\maketitle

\begin{abstract}
Frieze patterns have attracted significant attention recently, motivated by their relationship with cluster algebras. A longstanding open problem has been to provide a combinatorial model for frieze patterns over the ring of integers modulo $N$ akin to Conway and Coxeter's  celebrated model for positive integer frieze patterns. Here we solve this problem using the Farey complex of the ring of integers modulo $N$; in fact, using more general Farey complexes we provide combinatorial models for frieze patterns over any rings whatsoever.
    
Our strategy generalises that of the first author and of Morier-Genoud et al. for integers and that of Felikson et al. for Eisenstein integers. We also generalise results of Singerman and Strudwick on diameters of Farey graphs, we recover a theorem of Morier-Genoud on enumerating friezes over finite fields, and we classify those  frieze patterns modulo $N$ that lift to frieze patterns over the integers in terms of the topology of the corresponding Farey complexes.
\end{abstract}

\section{Introduction}\label{section1}

The purpose of this paper is to provide a combinatorial model for classifying $\text{SL}_2$-tilings and frieze patterns over any ring. In particular, we solve the outstanding problem of describing such a model for frieze patterns over the ring of integers modulo $N$ in a manner similar to that used by Conway and Coxeter for positive integer frieze patterns.

Frieze patterns (or, more concisely, friezes) are bi-infinite arrays of numbers of the type shown in Figure~\ref{fig11}, in which any diamond of four entries $a$, $b$, $c$, and $d$ satisfies the rule $ad-bc=1$. 

\begin{figure}[ht]
	\centering
	\begin{subfigure}[b]{0.7\textwidth}\centering	
	\centering
	\(
		\vcenter{\small
			\xymatrix @-0.9pc @!0 {
				& 0 && 0 && 0 && 0 && 0 && 0 && 0 && 0 && 0 && 0 & \\
				&& 1 && 1 && 1 && 1 && 1 && 1 && 1 && 1 && 1 && \\      
\raisebox{-5pt}{\ldots}\,  & 1 && 3 && 1 && 2 && 2 && 1 && 3 && 1 && 2 && 2 &   \,\raisebox{-5pt}{\ldots}  \\
				  && 2 && 2 && 1 && 3 && 1 && 2 && 2 && 1 && 3 &&   \\
				& 1 && 1 && 1 && 1 && 1 && 1 && 1 && 1 && 1 && 1 \\
				&& 0 && 0 && 0 && 0 && 0 && 0 && 0 && 0 && 0 &&    \\
			}
		}
	\)
	\end{subfigure}
	\begin{subfigure}[b]{0.2\textwidth}\centering	
\(\vcenter{\
	\xymatrix @-1.0pc @!0 {
		&b &\\
		a && d\\
		& c &
	}
}\)
	\end{subfigure}
	\caption{A frieze (left) and a diamond of four entries (right)}
	\label{fig11}
\end{figure}

This frieze  has integer entries; however, we will consider friezes with entries in any ring, where we assume, throughout, that all our rings are commutative and contain a multiplicative identity element~1.

To facilitate our formal definition of a frieze, it is helpful to label the entries as follows.
\[
\vcenter{
	\xymatrix @-0.1pc @!0 {
		&m_{-3,-3} && m_{-2,-2} && m_{-1,-1} && m_{0,0} && m_{1,1} && m_{2,2} && m_{3,3}& \\
		  && m_{-2,-3}	&& m_{-1,-2} && m_{0,-1} && m_{1,0} && m_{2,1} && m_{3,2}  && \\
		 \dotsb\quad& m_{-2,-4} && m_{-1,-3} && m_{0,-2} && m_{1,-1} && m_{2,0} && m_{3,1} && m_{4,2}&\dotsb \\
	&	&&  &&  && \vdots &&  && &&    	     &  \\ 
 	&	m_{n-5,-5} && m_{n-4,-4} && m_{n-3,-3} && m_{n-2,-2} && m_{n-1,-1} && m_{n,0} && m_{n+1,1}   & 	
	}
}
\]

This frieze has $n+1$ rows and is said to have width $n$ (where $n\geq 2$ throughout).

\begin{definition}
A \emph{frieze} of width $n$ over a ring $R$  is a function 
\[
\mathbf{F}\colon \{(i,j)\in\mathbb{Z}\times\mathbb{Z}:0\leq i-j\leq n\}\longrightarrow R
\]
with entries $m_{i,j}=\mathbf{F}(i,j)$ such that  
\begin{itemize}
\item $m_{i,i}=m_{n+i,i}=0$, for $i\in\mathbb{Z}$ (top and bottom rows of zeros),
\item $m_{i,j}m_{i+1,j+1}-m_{i,j+1}m_{i+1,j}=1$, for $0< i-j<n$ (diamond rule).
\end{itemize}
\end{definition}

The friezes originally considered by Coxeter \cite{Co1971}  had positive integer entries, except for the top and bottom rows of zeros. With these restrictions, it is straightforward to see from the diamond rule that the second and second-last rows of Coxeter's original friezes were bi-infinite sequences of 1's. In general, we will refer to a frieze over a ring $R$ as a \emph{regular frieze} if the second and second-last rows comprise 1's only (that is, $m_{i+1,i}=m_{n+i-1,i}=1$, for $i\in\mathbb{Z}$), and we refer to a frieze as a \emph{semiregular frieze} if the second row comprises 1's only but there is no restriction on the second-last row. We note that some other works choose $n-1$ or $n-3$ rather than $n$ to be the width of a frieze, and some consider a frieze to be what we have called a regular frieze. There are significant differences between the classes of semiregular and regular friezes, the most striking being that, in the former class, there are tame friezes that are \emph{not} periodic -- see Figure~\ref{fig76} for an example (the definition of tame will follow shortly).

\begin{figure}[ht]
	\centering
		\(
		\vcenter{\small
			\xymatrix @-0.8pc @!0 {
				& 0 && 0 && 0 && 0 && 0 && 0 && 0 && 0 && 0 && 0 & \\
				&& 1 && 1 && 1 && 1 && 1 && 1 && 1 && 1 && 1 && \\
				\cdots  & 4 && \tfrac34 && 2 && \tfrac32 && 1 && 3 && \tfrac12 && 6 && \tfrac14 && 12 & \cdots  \\
				&& 2 && \tfrac12 && 2 && \tfrac12 && 2 && \tfrac12 && 2 && \tfrac12 && 2 && \\      	
				& 0 && 0 && 0 && 0 && 0 && 0 && 0 && 0 && 0 && 0 & \\
			}
		}
		\)
	\caption{A tame semiregular frieze over $\mathbb{Q}$}
	\label{fig76}
\end{figure}

Closely connected to friezes are $\text{SL}_2$-tilings, and both objects relate to  the special linear group of degree 2 over $R$, namely 
\[
\text{SL}_2(R) = \left\{ \begin{pmatrix}a&b\\ c& d\end{pmatrix} : a,b,c,d\in R,\, ad-bc=1\right\}.
\]
Informally speaking, an $\text{SL}_2$-tiling over $R$ is a bi-infinite matrix in which any 2-by-2 submatrix of contiguous entries belongs to $\text{SL}_2(R)$. A formal definition follows.

\begin{definition}
An \emph{$\text{SL}_2$-tiling} over a ring $R$ is a function $\mathbf{M}\colon \mathbb{Z}\times\mathbb{Z}\longrightarrow R$ such that
	\[
	\begin{pmatrix}m_{i,j} & m_{i,j+1}\\ m_{i+1,j} & m_{i+1,j+1} \end{pmatrix}\in \text{SL}_2(R),\quad\text{for }i,j\in\mathbb{Z},
	\]
	where $m_{i,j}=\mathbf{M}(i,j)$.
\end{definition}

We think of $m_{i,j}$ as the entry of $\mathbf{M}$ in the $i$-th row and $j$-th column, with the row index increasing downwards and the column index increasing from left to right -- the usual conventions for matrices.

Any frieze $\mathbf{F}$ can be embedded in an $\text{SL}_2$-tiling $\mathbf{M}$ by defining $\mathbf{M}(i,j)=\textbf{F}(i,j)$, for $0\leq i-j\leq n$, and then completing the remaining entries of $\mathbf{M}$ in some suitable manner. There are many ways to complete $\mathbf{M}$, however, only one method is of interest to us, which we will come to shortly. Informally speaking, it helps to visualise an embedding of $\mathbf{F}$ in $\mathbf{M}$ by rotating $\mathbf{F}$ through $45^\circ$ clockwise to form a diagonal strip  in $\mathbf{M}$, with the top row of $\mathbf{F}$ on the leading diagonal of~$\mathbf{M}$.

The next definition is widely used in $\text{SL}_2$-tilings theory; it concerns $\text{SL}_2$-tilings with the property that every 3-by-3 submatrix of contiguous entries has determinant 0.

\begin{definition}
	An $\textnormal{SL}_2$-tiling $\mathbf{M}$ is \emph{tame} if
	\[
	\det\!\begin{pmatrix}m_{i-1,j-1} & m_{i-1,j} & m_{i-1,j+1}\\ m_{i,j-1} & m_{i,j} & m_{i,j+1}\\  m_{i+1,j-1} & m_{i+1,j} & m_{i+1,j+1}\end{pmatrix}=0,
	\]
	for $i,j\in\mathbb{Z}$.
\end{definition}

Similarly, we say that a frieze $\mathbf{F}$ of width $n$ is tame if each of these 3-by-3 determinants is 0, for $1<i-j<n-1$. In Section~\ref{section8} we will show that for each tame frieze $\mathbf{F}$ there is a unique tame $\text{SL}_2$-tiling $\mathbf{M}$ with $\mathbf{M}(i,j)=\textbf{F}(i,j)$, for $0\leq i-j\leq n$; we call $\mathbf{M}$ the \emph{extension} of $\mathbf{F}$. 

\begin{figure}[H]
	\begin{subfigure}[b]{0.5\textwidth}
		\begingroup
		{\footnotesize
			\begin{tikzpicture}[x=.55cm,y=.55cm]
				
				\foreach \i in {0,...,12}
				\foreach \j in {0,...,12}
				\coordinate (X\i\j) at (\i,12-\j);
				
				\coordinate (dd11) at (1,11.5);
				\coordinate (dd12) at (1.5,12);
				\coordinate (dd13) at (12,1.5);
				\coordinate (dd14) at (11.5,1);

				\coordinate (dd21) at (1,6.5);
				\coordinate (dd22) at (1.5,7);
				\coordinate (dd23) at (7,1.5);
				\coordinate (dd24) at (6.5,1);
				
				\coordinate (dd31) at (6,11.5);
				\coordinate (dd32) at (6.5,12);
				\coordinate (dd33) at (12,6.5);
				\coordinate (dd34) at (11.5,6);
				
				\coordinate (c11) at (1.15,1.15);
				\coordinate (c12) at (1.15,1.85);
				\coordinate (c13) at (1.85,1.85);
				\coordinate (c14) at (1.85,1.15);
				
				\coordinate (c21) at (11.15,11.15);
				\coordinate (c22) at (11.15,11.85);
				\coordinate (c23) at (11.85,11.85);
				\coordinate (c24) at (11.85,11.15);

				\fill[gray!30,rounded corners] (dd11) -- (dd12) --(dd13)--(dd14)--(dd11)--(dd12);
				\fill[gray!30,rounded corners] (dd21) -- (dd22) --(dd23)--(dd24)--(dd21)--(dd22);
				\fill[gray!30,rounded corners] (dd31) -- (dd32) --(dd33)--(dd34)--(dd31)--(dd32);
				\fill[gray!30,rounded corners] (c11) -- (c12) --(c13)--(c14)--(c11)--(c12);
				\fill[gray!30,rounded corners] (c21) -- (c22) --(c23)--(c24)--(c21)--(c22);

				\node (D) at ($(X612)+(.5,.8)$) {$\vdots$};
				\node (U) at ($(X60)+(.5,.7)$) {$\vdots$};
				\node (L) at ($(X06)+(.5,.5)$) {$\cdots$};
				\node (R) at ($(X126)+(.5,.5)$) {$\cdots$};

				\node at ($(11,11)+(.5,.5)$) {$0$};
				\node at ($(1,1)+(.5,.5)$) {$0$};
				
				\node at ($(10,11)+(.5,.5)$) {$1$};
				\node at ($(11,10)+(.5,.5)$) {$1$};
				\node at ($(1,2)+(.35,.5)$) {$-1$};
				\node at ($(2,1)+(.35,.5)$) {$-1$};
				
				\foreach \i in {1,...,10}{
					\node at ($(X\i\i)+(.5,.5)$) {$0$};
					\node at ($(X\i\i)+(1.35,.5)$) {$-1$};
					\node at ($(X\i\i)+(.5,-.5)$) {$1$};
				};
				\node at ($(X1111)+(.5,.5)$) {$0$};

				\foreach \i in {1,...,5}{
					\node at ($(X\i\i)+(5.5,.5)$) {$0$};
					\node at ($(X\i\i)+(6.5,.5)$) {$1$};
					\node at ($(X\i\i)+(5.35,-.5)$) {$-1$};
				};
				\node at ($(X116)+(.5,.5)$) {$0$};
				\node at ($(X117)+(.35,.5)$) {$-1$};
				\node at ($(X51)+(.35,.5)$) {$-1$};

				\foreach \i in {1,...,5}{
					\node at ($(X\i\i)+(.5,-4.5)$) {$0$};
					\node at ($(X\i\i)+(.5,-3.5)$) {$1$};
					\node at ($(X\i\i)+(.35,-5.5)$) {$-1$};
				};
				\node at ($(X611)+(.5,.5)$) {$0$};
				\node at ($(X610)+(.5,.5)$) {$1$};
				\node at ($(X711)+(.5,.5)$) {$1$};
				
				\newcommand*\List{{"$1$", "$3$", "$1$", "$2$", "$2$", "$1$", "$3$", "$1$", "$2$"}}
				\newcommand*\Listneg{{"$-1$", "$-3$", "$-1$", "$-2$", "$-2$", "$-1$", "$-3$", "$-1$", "$-2$"}}
				
				\newcommand*\Lists{{"$2$", "$2$", "$1$", "$3$", "$1$", "$2$", "$2$", "$1$", "$3$", "$1$", "$2$"}}
				\newcommand*\Listsneg{{"$-2$", "$-2$", "$-1$", "$-3$", "$-1$", "$-2$", "$-2$", "$-1$", "$-3$", "$-1$", "$-2$"}}
				
				\foreach \i in {1,...,9}{
					\node at ($(X\i\i)+(2.35,.5)$) {\pgfmathprint{\Listneg[\i-1]}};
					\node at ($(X\i\i)+(.5,-1.5)$) {\pgfmathprint{\List[\i-1]}};
				};
				
				\foreach \i in {1,...,8}{
					\node at ($(X\i\i)+(3.35,.5)$) {\pgfmathprint{\Listsneg[\i-1]}};
					\node at ($(X\i\i)+(.5,-2.5)$) {\pgfmathprint{\Lists[\i-1]}};
				};
				
				\foreach \i in {1,...,4}{
					\node at ($(X\i\i)+(7.5,.5)$) {\pgfmathprint{\List[\i-1]}};
					\node at ($(X\i\i)+(.35,-6.5)$) {\pgfmathprint{\Listneg[\i-1]}};
				};
				
				\foreach \i in {1,...,3}{
					\node at ($(X\i\i)+(8.5,.5)$) {\pgfmathprint{\Lists[\i-1]}};
					\node at ($(X\i\i)+(.35,-7.5)$) {\pgfmathprint{\Listsneg[\i-1]}};
				};
				
			\end{tikzpicture}
		}
		\endgroup
		\vspace*{-0.6cm}
		\caption{}
		\label{fig5a}
	\end{subfigure}
	\begin{subfigure}[b]{0.5\textwidth}
		\begingroup
		{\footnotesize
			$
			\vcenter{
				\xymatrix @-0.7pc @!0 {
					& &  &  &&& \vdots & && &  & &\\
					& 2& 0 & 3 & 1 & 4 & 2 & 0 & 3 & 1 & 4 &2&\\
					&  1  &  3 & 0 & 2 & 4 & 1 &3 & 0 & 2 & 4 & 1&\\
					&2    &  2  &  2 & 2 & 2 & 2 & 2 &2 & 2 & 2 & 2 &\\
					& 2& 0 & 3 & 1 & 4 & 2 & 0 & 3 & 1 & 4 &2&\\
					&  1  &  3 & 0 & 2 & 4 & 1 &3 & 0 & 2 & 4 &1& \\
					\dotsb   &2    &  2  &  2 & 2 & 2 & 2 & 2 &2 & 2 & 2 & 2 &\dotsb \\
					& 2& 0 & 3 & 1 & 4 & 2 & 0 & 3 & 1 & 4 &2&\\
					&  1  &  3 & 0 & 2 & 4 & 1 &3 & 0 & 2 & 4 & 1&\\
					&2    &  2  &  2 & 2 & 2 & 2 & 2 &2 & 2 & 2 & 2 &\\
					& 2& 0 & 3 & 1 & 4 & 2 & 0 & 3 & 1 & 4 &2&\\
					&  1  &  3 & 0 & 2 & 4 & 1 &3 & 0 & 2 & 4 &1& \\
					&  &  &  &&&   \vdots &  &  & &&& 
				}
			}
			$
		}
		\endgroup
		\vspace*{-0.6cm}
		\caption{}
		\label{fig5b}
	\end{subfigure}
	\caption{Tame $\text{SL}_2$-tilings over $\mathbb{Z}$ (left) and $\mathbb{Z}/5\mathbb{Z}$ (right)}
	\label{fig2}
\end{figure}

Figure~\ref{fig2}(a) illustrates the tame $\text{SL}_2$-tiling over $\mathbb{Z}$ that is the extension of the frieze $\mathbf{F}$ of Figure~\ref{fig11}. It comprises alternating positive and negative diagonal strips of the frieze $\mathbf{F}$. The backgrounds of the diagonals of 0s are shaded grey to distinguish the strips. Figure~\ref{fig2}(b) illustrates a tame $\text{SL}_2$-tiling over the ring $\mathbb{Z}/5\mathbb{Z}$ of integers modulo $5$; in this tiling the entry 4 (say) represents the congruence class of the integer~4. 

There is a rich literature on friezes and cognate subjects; we briefly review this literature and refer the reader to Baur's introduction \cite{Ba2021} and Morier-Genoud's survey \cite{Mo2015} for further exploration. Friezes were first studied by Coxeter \cite{Co1971} and Conway and Coxeter \cite{CoCo1973} in the 1970s. They described a correspondence between positive integer friezes and triangulated polygons, which has been at the heart of  research in the field ever since. Interest in friezes took off in the 2000s following work of Fomin and Zelevinsky \cite{FoZe2003}, Caldero and Chapoton \cite{CaCh2006}, and Assem, Reutenauer, and Smith \cite{AsReSm2010} relating friezes to cluster algebras and cluster categories. The last three authors considered $\text{SL}_2$-tilings, and Bergeron and Reutenauer developed the theory of $\text{SL}_k$-tilings in \cite{BeRe2010}. The theory was advanced by Morier-Genoud et al.\ \cite{MoOvScTa2014} who established a relationship between complex friezes, moduli spaces of points in projective space, and spaces of linear difference equations. In \cite{BeHoJo2017}, Bessenrodt, Holm, and J\o rgensen offered a combinatorial classification of all positive integer $\text{SL}_2$-tilings, and Baur, Parsons, and Tschabold provided a classification of infinite friezes in \cite{BaPaTs2016}. Guo introduced tropical friezes in \cite{Gu2013} and, more recently, Morier-Genoud and Ovsienko initiated the study of $q$-deformed friezes \cite{MoOv2020} leading to a solution (with Veselov) of a problem on Bureau representations of the braid group on three strands \cite{MoOvVe2024}.

For works complementary to this one, we highlight that of Holm and J\o rgensen \cite{HoJo2020} on friezes associated to Hecke groups, that of Cuntz, Holm, and J\o rgensen \cite{CuHoJo2020} on friezes with coefficients,  and that of Cuntz, Holm, and Pagano \cite{CuHoPa2023} on friezes over certain rings of algebraic numbers. The geometric approach we follow here was inspired by work of Morier-Genoud, Ovsienko, and Tabachnikov \cite{MoOvTa2015} on representing friezes by paths in the Farey graph in the hyperbolic plane. This approach was championed by the first author in \cite{Sh2023} and developed in three dimensions by Felikson et al.\ in \cite{FeKaSeTu2023}.

To classify friezes and $\text{SL}_2$-tilings over arbitrary rings, we introduce the Farey complex of a ring. Recall that a \emph{unit} in a ring $R$ is an invertible element of $R$. The collection of units $R^\times$ in $R$ forms a group under multiplication. A pair of elements $(a,b)$ in $R\times R$ is said to be a \emph{unimodular pair} if $aR+bR=R$ (or, equivalently, if there exist $x,y\in R$ with $ax+by=1$). Any subgroup $U$ of $R^\times$ acts on the collection of unimodular pairs by the rule $(a,b)\longmapsto (\lambda a,\lambda b)$, where $\lambda \in U$. We write $\lambda (a,b)$ for $(\lambda a,\lambda b)$, and we also use the notation
\[
U(a,b) =\{\lambda (a,b):\lambda \in U\}
\]
for the orbit of $(a,b)$ under $U$.

\begin{definition}
Let $R$ be a ring and let $U$ be a group of units in $R$. The \emph{Farey complex} $\mathscr{F}_{R,U}$ of $R$ with units $U$ is the directed graph with the following vertices and edges.
\begin{itemize}
\item The vertices are orbits $U(a,b)$, where $(a,b)$ is a unimodular pair in $R\times R$. We write vertices as formal fractions $a/b$. 
\item There is a directed edge from $a/b$ to $c/d$ if $ad-bc\in U$.
\end{itemize}
\end{definition}

We represent the directed edge from $a/b$ to $c/d$ by $a/b\to c/d$. Observe that there are no directed edges in $\mathscr{F}_{R,U}$ with the same source and sink, and there is at most one directed edge from one vertex to another.

The Farey complexes of most interest to us are those with $U=\{1\}$ and $U=\{\pm 1\}$. For brevity, we write $\mathscr{E}_R$ in place of $\mathscr{F}_{R,\{1\}}$ and $\mathscr{F}_R$ in place of $\mathscr{F}_{R,\{\pm1\}}$. The vertices of $\mathscr{E}_R$ are simply unimodular pairs from $R\times R$ (so we write them as pairs rather than formal fractions). 

When $-1\in U$, there is a directed edge from vertex $u$ to vertex $v$ if and only if there is a directed edge from $v$ to $u$. Thus, directed edges come in inverse pairs, and we can think of each pair as a single undirected edge. In this case, the Farey complex $\mathscr{F}_{R,U}$ is the quotient of $\mathscr{F}_{R}$ under the action of $U$. We define a \emph{face} of $\mathscr{F}_{R}$ to consist of a triple of vertices $u$, $v$, and $w$ that are adjacent in pairs (that is, they form a triangle). We then define the faces of $\mathscr{F}_{R,U}$ to be the images of the faces of $\mathscr{F}_{R}$ under the quotient map. The Farey complex $\mathscr{F}_{R,U}$ is then a 2-complex, and we can examine its geometric properties, as we shall see in due course.

The most familiar Farey complex is $\mathscr{F}_{\mathbb{Z}}$, the Farey complex of the integers (with units $\{\pm 1\}$), shown in Figure~\ref{fig1}. As usual, this complex is illustrated in the upper half-plane, with edges represented by hyperbolic lines. They are displayed as undirected edges because $-1\in U$.

\begin{figure}[ht]
\centering
\begin{tikzpicture}[scale=4]
\begin{scope}
    \clip(-0.9,-0.2) rectangle (2.3,0.75);
	\fareygraph[thin]{-1}{2}{0}
\end{scope}
\end{tikzpicture}
\caption{Part of the Farey complex $\mathscr{F}_{\mathbb{Z}}$}
\label{fig1}
\end{figure}

Another familiar Farey complex is $\mathscr{F}_{\mathbb{Z}[i],U}$, where  $\mathbb{Z}[i]$ is the ring of Gaussian integers and $U=\{\pm1, \pm i\}$, the full collection of units in $\mathbb{Z}[i]$. The vertices of this Farey complex are the Gaussian rationals. The 1-skeleton of $\mathscr{F}_{\mathbb{Z}[i],U}$ gives rise to a tessellation of hyperbolic 3-space by ideal octahedra. (The 2-skeleton of $\mathscr{F}_{\mathbb{Z}[i],U}$ differs from that of the octahedral tessellation because, for convenience, in defining a Farey complex we restricted the triangles in the complex that are designated as faces). A related Farey complex is $\mathscr{F}_{\mathbb{Z}[\omega],U}$, where  $U$ is now the cyclic group generated by $\omega=e^{\pi i/3}$; this complex gives a tessellation of hyperbolic 3-space by ideal tetrahedra. It was considered in the context of $\text{SL}_2$-tilings by Felikson et al.\ in \cite{FeKaSeTu2023}. More generally, when $R$ is the ring of integers of the imaginary quadratic field $\mathbb{Q}(\sqrt{-d})$ with $d=1,2,3,7,11$, the Farey complex $\mathscr{F}_{R,R^\times}$ gives rise to a tessellation of hyperbolic 3-space by ideal polyhedra associated with Bianchi groups. These tessellations were described by Hatcher \cite{Ha1983} (among others); see \cite{Vu1999} for a more general approach.

Much of this work concerns Farey complexes over finite rings, with focus on the ring of integers modulo $N$, denoted by $\mathbb{Z}/N\mathbb{Z}=\{0,1,\dots,N-1\}$. The Farey complexes $\mathscr{F}_{\mathbb{Z}/N\mathbb{Z}}$ (or, more briefly, $\mathscr{F}_N$), for $N=2,3,4,5,6$, are illustrated in Figure~\ref{fig29}. These are, in order, a triangle, tetrahedron, octahedron, icosahedron, and a triangulation of the hexagonal torus (opposite sides are identified). These complexes were studied by Ivrissimtzis, Singerman, and Strudwick in \cites{IvSi2005,SiSt2020} (and other works) as quotients of $\mathscr{F}_\mathbb{Z}$; we return to this later. By contrast, the complexes  $\mathscr{E}_{\mathbb{Z}/N\mathbb{Z}}$ do not in general arise naturally as surface complexes; for example, the underlying undirected graph of $\mathscr{E}_{\mathbb{Z}/3\mathbb{Z}}$ is the 1-skeleton of the regular polytope known as the 16-cell, which is the 4-dimensional analogue of the octahedron.

\begin{figure}[ht] %
	\begin{tikzpicture}[scale=0.6,line join=bevel,z=-5.5,font=\footnotesize,inner sep=0.2mm]
	
	\begin{scope}[xshift=-8cm]
	\pgfmathsetmacro{\r}{2.5}
	\coordinate (A) at (90:\r);
	\coordinate (B) at (210:\r);
	\coordinate (C) at (-30:\r);

	\draw (A) node[anchor=south]{$\tfrac10$} -- (B) node[anchor=east]{$\tfrac01$}-- (C) node[anchor=west]{$\tfrac11$} -- cycle;	
	
	\node at (0,-0.8*\r) {\normalsize$\mathscr{F}_2$};
	\end{scope}
	
	\tdplotsetmaincoords{0}{0}
	\begin{scope}[tdplot_main_coords]
	\pgfmathsetmacro{\r}{2.6}
	\pgfmathsetmacro{\s}{\r/sqrt(2)}
	
	\coordinate (A) at (\r,0,-\s);
	\coordinate (B) at (-\r,0,-\s);
	\coordinate (C) at (0,\r,\s);
	\coordinate (D) at (0,-\r,\s);
	
	\draw (D) node[anchor=north]{$\tfrac01$} -- (B) node[anchor=east]{$\tfrac11$}-- (C) node[anchor=south]{$\tfrac10$} -- cycle;
	\draw (C) -- (A) node[anchor=west]{$\tfrac21$}-- (D);
	\draw[dashed,gray](A)  -- (B);

	\node at (0,-1.5*\r,0) {\normalsize$\mathscr{F}_3$};
	\end{scope}

	\tdplotsetmaincoords{70}{165}
	\begin{scope}[tdplot_main_coords,xshift=8cm]
	\pgfmathsetmacro{\r}{3}
	
	\tdplotsetcoord{A}{\r}{90}{0}    
	\tdplotsetcoord{B}{\r}{90}{90}   
	\tdplotsetcoord{C}{\r}{90}{180}  
	\tdplotsetcoord{D}{\r}{90}{270}  
	\tdplotsetcoord{E}{\r}{0}{0}     
	\tdplotsetcoord{F}{\r}{180}{0}   
	
	\draw (A) node[anchor=east]{$\tfrac01$} -- (B) node[anchor=north west]{$\tfrac11$}-- (C) node[anchor=west]{$\tfrac21$};
	\draw (E) node[anchor=south]{$\tfrac10$} -- (A) -- (F) node[anchor=north]{$\tfrac12$};
	\draw (E) -- (B) -- (F);
	\draw (E) -- (C) -- (F);
	\draw[dashed,gray] (C) -- (D) node[anchor=south east]{\textcolor{black}{$\tfrac31$}} -- (A);
	\draw[dashed,gray](E)  -- (D) -- (F);
	
	\node at (0,0,-1.4*\r) {\normalsize$\mathscr{F}_4$};
	\end{scope}

	\tdplotsetmaincoords{50}{90}
	\begin{scope}[tdplot_main_coords,xshift=-4cm,yshift=-9cm]
	\pgfmathsetmacro{\a}{2}
	\pgfmathsetmacro{\phi}{\a*(1+sqrt(5))/2}
	\path 
	coordinate[label=right:{{$\tfrac21$}}](A) at (0,\phi,\a)
	coordinate[label=right:{{$\tfrac02$}}](B) at (0,\phi,-\a)
	coordinate[label=left:{{$\tfrac01$}}](C) at (0,-\phi,\a)
	coordinate[label=left:{{$\tfrac42$}}](D) at (0,-\phi,-\a)
	coordinate[label={[yshift=2pt]above right:{{$\tfrac11$}}}](E) at (\a,0,\phi)
	coordinate[label={[yshift=15.5pt,xshift=-3pt]below:{{$\tfrac20$}}}](F) at (\a,0,-\phi)
	coordinate[label=above:{{$\tfrac10$}}](G) at (-\a,0,\phi)
	coordinate[label={[yshift=-2pt]below right:{{$\tfrac22$}}}](H) at (-\a,0,-\phi)
	coordinate[label={[xshift=-5pt]above left:{{$\tfrac32$}}}](I) at (\phi,\a,0)
	coordinate[label={[xshift=5pt]above right:{{$\tfrac12$}}}](J) at (\phi,-\a,0)
	coordinate[label={[xshift=-4pt]left:{{$\tfrac31$}}}](K) at (-\phi,\a,0)
	coordinate[label={[xshift=4pt]right:{{$\tfrac41$}}}](L) at (-\phi,-\a,0); 
	\draw[dashed,gray]    (B) -- (H) -- (K) -- cycle 
	(D) -- (L) -- (H) --cycle 
	(K) -- (L) -- (G) --cycle
	(A)--(K) (L)--(C) (F)--(H);
	
	\draw(A) -- (I) -- (B) --cycle 
	(F) -- (I) -- (B) --cycle 
	(F) -- (I) -- (J) --cycle
	(F) -- (D) -- (J) --cycle
	(C) -- (D) -- (J) --cycle
	(C) -- (E) -- (J) --cycle
	(I) -- (E) -- (J) --cycle
	(I) -- (E) -- (A) --cycle
	(G) -- (E) -- (A) --cycle
	(G) -- (E) -- (C) --cycle; 
	
	\node at (1.23*\a,0,-1.23*\phi) {\normalsize$\mathscr{F}_5$};
	\end{scope}
	
	\begin{scope}[xshift=6.0cm,yshift=-9cm,scale=2]
	
	\coordinate[label={[shift={(2:0.33)}]$\tfrac10$}] (Z) at (0,0);
	\coordinate[label={[shift={(2:0.33)}]$\tfrac01$}](A) at (0:1);
	\coordinate[label=right:{$\tfrac13$}](B) at (0:2);
	\coordinate[label={[right,shift={(35:0.08)}]{{$\tfrac12$}}}](C) at (30:1.73);
	\coordinate[label={[shift={(2:0.33)}]$\tfrac11$}](D) at (60:1);
	\coordinate[label={[above,shift={(90:0.05)}]{{$\tfrac23$}}}](E) at (60:2);
	\coordinate[label={[above,shift={(90:0.05)}]{{$\tfrac32$}}}](F) at (90:1.73);
	\coordinate[label={[shift={(178:0.33)}]$\tfrac21$}](G) at (120:1);
	\coordinate[label={[above,shift={(90:0.05)}]{{$\tfrac13$}}}](H) at (120:2);
	\coordinate[label={[left,shift={(145:0.08)}]{{$\tfrac14$}}}](I) at (150:1.73);
	\coordinate[label={[shift={(178:0.33)}]$\tfrac31$}](J) at (180:1);
	\coordinate[label=left:\textcolor{col1}{$\tfrac23$}](K) at (180:2);
	\coordinate[label={[left,shift={(215:0.08)}]{{$\tfrac12$}}}](L) at (210:1.73);
	\coordinate[label={[shift={(235:0.53)}]$\tfrac41$}](M) at (240:1);
	\coordinate[label={[below,shift={(-90:0.05)}]{{$\tfrac13$}}}](N) at (240:2);
	\coordinate[label={[below,shift={(-90:0.05)}]{{$\tfrac32$}}}](O) at (270:1.73);
	\coordinate[label={[shift={(305:0.53)}]$\tfrac51$}](P) at (300:1);
	\coordinate[label={[below,shift={(-90:0.05)}]{{$\tfrac23$}}}](Q) at (300:2);
	\coordinate[label={[right,shift={(-35:0.08)}]{{$\tfrac14$}}}](R) at (330:1.73);
	
	\draw (E)--(H) (C)--(I) (B)--(K) (R)--(L) (Q)--(N)
	(Q)--(B) (O)--(C) (N)--(E) (L)--(F) (K)--(H)
	(E)--(B) (F)--(R) (H)--(Q) (I)--(O) (K)--(N);
	
	\node at (0,-2.4) {\normalsize$\mathscr{F}_6$};
	\end{scope}
	\end{tikzpicture}
	\caption{Farey complexes $\mathscr{F}_2$, $\mathscr{F}_3$, $\mathscr{F}_4$, $\mathscr{F}_5$, and $\mathscr{F}_6$}
	\label{fig29}
\end{figure}

Two further Farey complexes are illustrated in Figure~\ref{fig3}. Shown in Figure~\ref{fig3}(a) is the Farey complex $\mathscr{F}_R$, where $R$ is the field of order 4, which can be represented conveniently on a projective icosidodecahedron (opposite vertices on the outside are identified). In Figure~\ref{fig3}(b) is the Farey complex $\mathscr{F}_{R,U}$ for $R=\mathbb{Z}[i]/3\mathbb{Z}[i]$, which is the field of order 9, with $U=\{\pm 1, \pm i\}$. For clarity, the edges from $1/0$ to the other black vertices are omitted, as are the edges from $(1+i)/0$ to the other white vertices. By using the full collection of 120 triangles as faces we obtain a 3-complex comprising 30 octahedral facets with 9 incident to any given vertex, 4 incident to any given edge, and 2 incident to any given triangular face. With this representation, the complex (excluding its vertices) can be realised as a 3-manifold; in fact, it can be shown that this 3-manifold is $\text{SL}_2(\mathbb{Z}[i],3\mathbb{Z}[i])\backslash \mathbb{H}^3$, where $\text{SL}_2(\mathbb{Z}[i],3\mathbb{Z}[i])$ is the principal congruence subgroup of $\text{SL}_2(\mathbb{Z}[i])$ for the ideal $3\mathbb{Z}[i]$. We refer the reader to \cite{BaRe2014} for more on these connections to hyperbolic 3-manifolds and for a proof (see \cite{BaRe2014}*{Theorem~1.1}) that $\text{SL}_2(\mathbb{Z}[i],3\mathbb{Z}[i])\backslash \mathbb{H}^3$ is topologically a 20-component link complement in the three-sphere.

\begin{figure}[ht]
     \centering
     \begin{subfigure}[b]{0.4\textwidth}
         \centering
         \begin{tikzpicture}[scale=1.3,line join=bevel,z=-5.5,font=\footnotesize,inner sep=0.2mm]
	   	\pgfmathsetmacro{\L}{1.7}
		\pgfmathsetmacro{\M}{2.2}
		\pgfmathsetmacro{\G}{13}
		
		\coordinate (A) at (18:1);
		\coordinate (B) at (90:1);
		\coordinate (C) at (162:1);
		\coordinate (D) at (234:1);
		\coordinate (E) at (306:1);
		
		\coordinate (F) at (54:\L);
		\coordinate (G) at (126:\L);
		\coordinate (H) at (198:\L);
		\coordinate (I) at (270:\L);
		\coordinate (J) at (342:\L);

		\coordinate (K) at (90-\G:\M);
		\coordinate (L) at (90+\G:\M);
		\coordinate (M) at (162-\G:\M);
		\coordinate (N) at (162+\G:\M);
		\coordinate (O) at (234-\G:\M);
		\coordinate (P) at (234+\G:\M);
		\coordinate (Q) at (306-\G:\M);
		\coordinate (R) at (306+\G:\M);
		\coordinate (S) at (18-\G:\M);
		\coordinate (T) at (18+\G:\M);
	
		\draw (A) -- (B) -- (C) -- (D) -- (E) -- (A) -- (F) -- (B) -- (G) -- (C) -- (H) -- (D) -- (I) -- (E) -- (J) -- (A);
		
		\draw (K) -- (L) -- (M) -- (N) -- (O) -- (P) -- (Q) -- (R) -- (S) -- (T) -- (K);

		\draw (L) -- (G) -- (M); 
		\draw (T) -- (F) -- (K);
		\draw (P) -- (I) -- (Q);	
		\draw (N) -- (H) -- (O);
		\draw (R) -- (J) -- (S);
					
		\draw (A) node[anchor=south west,yshift=-3pt]{$\tfrac{a}{1}$};
		\draw (B) node[anchor=south,yshift=1pt]{$\tfrac{a}{a}$};
		\draw (C) node[anchor=south east,yshift=-3pt]{$\tfrac{1}{a}$};
		\draw (D) node[anchor=north east,xshift=1.5pt]{$\tfrac{0}{1}$};
		\draw (E) node[anchor=north west,xshift=-1.5pt]{$\tfrac{1}{0}$};
		\draw (F) node[anchor=south west,xshift=-3pt,yshift=0pt]{$\tfrac{0}{b}$};
		\draw (G) node[anchor=south east,xshift=3pt,yshift=0pt]{$\tfrac{b}{0}$};	
		\draw (H) node[anchor=east,yshift=-2pt]{$\tfrac{1}{b}$};
		\draw (I) node[anchor=south,yshift=2pt]{$\tfrac{1}{1}$};		
		\draw (J) node[anchor=west,yshift=-2pt]{$\tfrac{b}{1}$};
		\draw (K) node[anchor=south,xshift=2pt]{$\tfrac{a}{b}$};
		\draw (L) node[anchor=south,xshift=-2pt]{$\tfrac{b}{a}$};	
		\draw (M) node[anchor=east,yshift=1pt]{$\tfrac{0}{a}$};
		\draw (N) node[anchor=east]{$\tfrac{b}{b}$};	
		\draw (O) node[anchor=east,yshift=-1pt]{$\tfrac{a}{0}$};
		\draw (P) node[anchor=north east,xshift=2pt]{$\tfrac{a}{b}$};		
		\draw (Q) node[anchor=north west,xshift=-2pt]{$\tfrac{b}{a}$};
		\draw (R) node[anchor=west,yshift=-1pt]{$\tfrac{0}{a}$};
		\draw (S) node[anchor=west,xshift=0pt]{$\tfrac{b}{b}$};	
		\draw (T) node[anchor=west,yshift=1pt]{$\tfrac{a}{0}$};
		
	   \end{tikzpicture}
         \caption{$\mathscr{F}_{R}$, where $R$ is the field of order 4}
      \end{subfigure}
	\hfill
     \begin{subfigure}[b]{0.5\textwidth}
         \centering
         \begin{tikzpicture}[scale=1.2,line join=bevel,font=\scriptsize,inner sep=0.2mm,blackdot/.style={circle,draw,thick,fill=black,minimum size=3},whitedot/.style={circle,draw,thick,fill=white,minimum size=3}]
         	\pgfmathsetmacro{\d}{1.4}
         	\foreach \i in {0,...,3} 
	   	{
	   	   \pgfmathsetmacro{\ip}{\i+0.5}
	   	   \draw[gray] (\i*\d,0) --++(0,3*\d);
	   	   \draw[gray] (0,\i*\d) --++(3*\d,0);
	   	   \draw[gray] (\ip*\d,0.5*\d) --++(0,3*\d);
	   	   \draw[gray] (0.5*\d,\ip*\d) --++(3*\d,0);
	   	   \draw[gray] (0,{(3-\i)*\d}) --++ ({(\i+0.5)*\d},{(\i+0.5)*\d}) --++({(3-\i)*\d},{-(3-\i)*\d})--++ ({-(\i+0.5)*\d},{-(\i+0.5)*\d})--++({-(3-\i)*\d},{(3-\i)*\d}) ;
        	}
         	
	   	\foreach \i in {0,...,3} 
	   	   \foreach \j in {0,...,3} 
       	   {
       	   	\pgfmathsetmacro{\ip}{\i+0.5}
       	   	\pgfmathsetmacro{\jp}{\j+0.5}
       		\node [blackdot]  at (\i*\d,\j*\d) {};
       		\node [whitedot]  at (\ip*\d,\jp*\d) {};
       	   } 
		\draw (0,0) node[anchor=north east,xshift=0pt,yshift=0pt]{$\tfrac{0}{1}$};
       	\draw (\d,0) node[anchor=north east,xshift=0pt,yshift=0pt]{$\tfrac{1}{1}$};
       	\draw (2*\d,0) node[anchor=north east,xshift=0pt,yshift=0pt]{$\tfrac{2}{1}$};
       	\draw (3*\d,0) node[anchor=north east,xshift=0pt,yshift=0pt]{$\tfrac{0}{1}$};
       	
       	\draw (0,\d) node[anchor=north east,xshift=0pt,yshift=0pt]{$\tfrac{i}{1}$};
       	\draw (\d,\d) node[anchor=north east,xshift=-4pt,yshift=0pt]{$\tfrac{1+i}{1}$};
       	\draw (2*\d,\d) node[anchor=north east,xshift=-4pt,yshift=0pt]{$\tfrac{2+i}{1}$};
       	\draw (3*\d,\d) node[anchor=north east,xshift=-6pt,yshift=0pt]{$\tfrac{i}{1}$};
       	
       	\draw (0,2*\d) node[anchor=north east,xshift=0pt,yshift=0pt]{$\tfrac{2i}{1}$};
       	\draw (\d,2*\d) node[anchor=north east,xshift=-4pt,yshift=0pt]{$\tfrac{1+2i}{1}$};
       	\draw (2*\d,2*\d) node[anchor=north east,xshift=-4pt,yshift=0pt]{$\tfrac{2+2i}{1}$};
       	\draw (3*\d,2*\d) node[anchor=north east,xshift=-4.5pt,yshift=0pt]{$\tfrac{2i}{1}$};
       	
       	\draw (0,3*\d) node[anchor=north east,xshift=0pt,yshift=0pt]{$\tfrac{0}{1}$};
       	\draw (\d,3*\d) node[anchor=north east,xshift=-6pt,yshift=0pt]{$\tfrac{1}{1}$};
       	\draw (2*\d,3*\d) node[anchor=north east,xshift=-6pt,yshift=0pt]{$\tfrac{2}{1}$};
       	\draw (3*\d,3*\d) node[anchor=north east,xshift=-6pt,yshift=0pt]{$\tfrac{0}{1}$};
       	
       	\begin{scope}[shift={(0.5*\d,0.5*\d)}]
	       	\draw (0,0) node[anchor=north west,xshift=6pt,yshift=0pt]{$\tfrac{i}{1+i}$};
	       	\draw (\d,0) node[anchor=north west,xshift=5.5pt,yshift=0pt]{$\tfrac{1+2i}{1+i}$};
	       	\draw (2*\d,0) node[anchor=north west,xshift=6pt,yshift=0pt]{$\tfrac{2}{1+i}$};
	       	\draw (3*\d,0) node[anchor=north west,xshift=0pt,yshift=0pt]{$\tfrac{i}{1+i}$};
	       	
	       	\draw (0,\d) node[anchor=north west,xshift=6pt,yshift=0pt]{$\tfrac{2+2i}{1+i}$};
	       	\draw (\d,\d) node[anchor=north west,xshift=6pt,yshift=0pt]{$\tfrac{0}{1+i}$};
	       	\draw (2*\d,\d) node[anchor=north west,xshift=6pt,yshift=0pt]{$\tfrac{1+i}{1+i}$};
	       	\draw (3*\d,\d) node[anchor=north west,xshift=0pt,yshift=0pt]{$\tfrac{2+2i}{1+i}$};
	       	
	       	\draw (0,2*\d) node[anchor=north west,xshift=6pt,yshift=0pt]{$\tfrac{1}{1+i}$};
	       	\draw (\d,2*\d) node[anchor=north west,xshift=6pt,yshift=0pt]{$\tfrac{2+i}{1+i}$};
	       	\draw (2*\d,2*\d) node[anchor=north west,xshift=6pt,yshift=0pt]{$\tfrac{2i}{1+i}$};
	       	\draw (3*\d,2*\d) node[anchor=north west,xshift=0pt,yshift=0pt]{$\tfrac{1}{1+i}$};
	       	
	       	\draw (0,3*\d) node[anchor=south west,xshift=0pt,yshift=0pt]{$\tfrac{i}{1+i}$};
	       	\draw (\d,3*\d) node[anchor=south west,xshift=0pt,yshift=0pt]{$\tfrac{1+2i}{1+i}$};
	       	\draw (2*\d,3*\d) node[anchor=south west,xshift=0pt,yshift=0pt]{$\tfrac{2}{1+i}$};
	       	\draw (3*\d,3*\d) node[anchor=south west,xshift=0pt,yshift=0pt]{$\tfrac{i}{1+i}$};
        	 \end{scope}   
        	 
	    \node [blackdot,label={[below,yshift=-5pt]:{$\tfrac10$}}]  at (-0.5*\d,1.5*\d) {};
	    \node [whitedot,label={[below,yshift=-5pt]:{$\tfrac{1+i}0$}}]  at (4*\d,2*\d) {};

	   \end{tikzpicture}
         \caption{$\mathscr{F}_{R,U}$, where $R=\mathbb{Z}[i]/3\mathbb{Z}[i]$ and $U=\{\pm 1,\pm i\}$}
         \label{fig:five over x}
     \end{subfigure}
        \caption{Farey complexes of the fields of orders 4 and 9}
        \label{fig3}
\end{figure}

For any Farey complex $\mathscr{F}_{R,U}$, there is a left action of $\text{SL}_2(R)$ on $\mathscr{F}_{R,U}$ given by 
\[
\frac{x}{y}\longmapsto \frac{ax+by}{cx+dy},\quad\text{where }A=\begin{pmatrix}a &b \\ c&d\end{pmatrix}.
\]
This action is transitive on directed edges.  We denote the image of a vertex $v$ (and matrix $A$) under the action by $Av$. 

The Farey complex $\mathscr{F}_{R,U}$ is the quotient of $\mathscr{E}_R$ under the action of $U$ (ignoring the faces, which are undefined for $\mathscr{E}_R$). Let $\pi_{U}\colon\mathscr{E}_{R}\longrightarrow\mathscr{F}_{R,U}$ be the associated quotient map, a covering map, given by $(x,y)\longmapsto U(x,y)$. This map is equivariant under the action of $\text{SL}_2(R)$, in the sense that, for any $A\in\text{SL}_2(R)$, the following diagram commutes.
\[
\begin{tikzcd}[row sep=4.5em,column sep=5.5em]
\mathscr{E}_{R} \arrow[r, "A",pos=0.5] \arrow[d,"\pi_{U}" left]
& \mathscr{E}_{R} \arrow[d,"\pi_{U}" right] \\
\mathscr{F}_{R,U} \arrow[r,"A" below]
& \mathscr{F}_{R,U}
\end{tikzcd}
\]

A \emph{bi-infinite path} in $\mathscr{F}_{R,U}$ is a bi-infinite sequence of vertices $\dotsc,v_{-1},v_0,v_1,v_2,\dotsc$ such that $v_{i-1}\to v_i$ is a directed edge in $\mathscr{F}_{R,U}$, for $i\in\mathbb{Z}$. We denote this path by $\langle\, \dotsc,v_{-1},v_0,v_1,v_2,\dotsc\rangle$. We will also use finite paths, usually referred to simply as `paths', with corresponding notation. The \emph{length} of the finite path $\langle v_0,v_1,\dots,v_n\rangle$ is $n$. Given any path $\gamma$ in $\mathscr{F}_{R,U}$ -- whether finite or bi-infinite -- there is a \emph{lift} of this path to $\mathscr{E}_R$; that is, there is a path $\tilde{\gamma}$ in $\mathscr{E}_R$ with $\pi_{U}(\tilde{\gamma})=\gamma$, and usually there are many lifts (these are specified in Lemma~\ref{lem61}). 

We sketch proofs of these facts about Farey complexes and group actions in Sections~\ref{section2} and~\ref{section3}. We will prove (in Section~\ref{section3}) that there is a one-to-one correspondence between bi-infinite paths in $\mathscr{E}_R$ up to $\text{SL}_2(R)$ equivalence and bi-infinite sequences in $R$. In Section~\ref{section4} we prove that between any two vertices of a Farey complex over a finite ring there is a path of length at most 3. This generalises \cite{SiSt2020}*{Theorem~12} in which the same result was established for the Farey complexes $\mathscr{F}_N$ (of $\mathbb{Z}/N\mathbb{Z}$ with units $\{\pm 1\}$). Then in Section~\ref{section5} we prove that the Farey complexes $\mathscr{F}_N$ (some of which were illustrated in Figure~\ref{fig29}) are the \emph{only} Farey complexes of finite rings that topologically are surfaces.

The chief objective of this paper is to use Farey complexes to provide combinatorial models for tame $\text{SL}_2$-tilings and friezes (over any ring), and the next three theorems, which are the principal results of this paper, achieve this objective, starting with $\text{SL}_2$-tilings.

The method we employ is motivated by that of \cite{Sh2023} for $\text{SL}_2$-tilings over the integers. Let $\mathscr{P}$ denote the collection of bi-infinite paths in $\mathscr{E}_R$, and let $\SL$ denote the collection of tame $\text{SL}_2$-tilings over $R$. We define a function $\widetilde{\Phi}\colon \mathscr{P}\times \mathscr{P}\longrightarrow \SL$ as follows. Given bi-infinite paths $\gamma$ and $\delta$ in  $\mathscr{P}$, with vertices $(a_i,b_i)$ and $(c_j,d_j)$, respectively, we define $\widetilde{\Phi}(\gamma,\delta)$ to be the $\text{SL}_2$-tiling $\mathbf{M}$ with entries
\[
m_{i,j}=a_id_j-b_ic_j, \quad\text{for $i,j\in\mathbb{Z}$}.
\]
We will prove that $\mathbf{M}$ is a tame $\text{SL}_2$-tiling in Section~\ref{section6}, where we will also see that $\widetilde{\Phi}(A{\gamma},A{\delta})=\widetilde{\Phi}(\gamma,{\delta})$, for $A\in\text{SL}_2(R)$.

Consider now some subgroup $U$ of $R^\times$. Let $\mathscr{P}_U$ denote the collection of bi-infinite paths in $\mathscr{F}_{R,U}$ (so $\mathscr{P}=\mathscr{P}_{\{1\}}$). The group $\text{SL}_2(R)$ acts on $\mathscr{P}_U\times\mathscr{P}_U$ by the rule $(\gamma,\delta)\longmapsto (A\gamma,A\delta)$, where $A\in\text{SL}_2(R)$. There is also an action of the group $U\times U$ on $\SL$ given by 
\[
m_{i,j} \longmapsto \lambda^{(-1)^{i}}\mu^{(-1)^{j}}m_{i,j},
\]
where $(\lambda,\mu)\in U\times U$. Under this action, informally speaking, we multiply even columns of $\mathbf{M}$ by $\lambda$ and odd columns by $\lambda^{-1}$, and we multiply even rows of $\mathbf{M}$ by $\mu$ and odd rows by $\mu^{-1}$. This action is not usually faithful; it has kernel $\{(\lambda,\lambda):\lambda^2=1\}$.

Let $(\gamma,\delta)$ and $(\gamma',\delta')$ be two pairs in $\mathscr{P}_U\times\mathscr{P}_U$ that are equivalent under the action of $\text{SL}_2(R)$. We will see in Section~\ref{section6} that the $\widetilde{\Phi}$-images of  any two lifts to $\mathscr{P}\times \mathscr{P}$ of two such pairs are equivalent in $\SL$ under the action of $U\times U$. Hence $\widetilde{\Phi}$ induces a map  $\Phi_U\colon \text{SL}_2(R)\backslash(\mathscr{P}_U\times\mathscr{P}_U)\longrightarrow (U\times U)\backslash \SL$ such that the following diagram commutes. 

\begin{tikzcd}[row sep=4.5em,column sep=5.5em]
\mathscr{P}\times\mathscr{P} \arrow[r, "\widetilde{\Phi}",pos=0.5,,shorten= 1.0ex,shorten >=1.0ex] \arrow[d]
& \SL \arrow[d] \\
\text{SL}_2(R)\backslash(\mathscr{P}_U\times\mathscr{P}_U) \arrow[r,"\Phi_U" below]
& (U\times U)\backslash \SL
\end{tikzcd}

Our first main result is that $\Phi_U$ is a bijection.

\begin{theorem}\label{thm3}
The map $\Phi_U$ is a one-to-one correspondence between
\[
\textnormal{SL}_2(R)\Big\backslash\mleft\{\parbox{2.8cm}{\centering\textnormal{pairs~of~bi-infinite paths~in~$\mathscr{F}_{R,U}$}}\mright\}\quad \longleftrightarrow\quad (U\times U)\Big\backslash\mleft\{\parbox{2.6cm}{\centering\textnormal{tame~$\text{SL}_2$-tilings over~$R$}}\mright\}.
\]
\end{theorem}

This result generalises \cite{Sh2023}*{Theorem~1.1}, which is the special case $R=\mathbb{Z}$ and $U=\{\pm 1\}$, and \cite{FeKaSeTu2023}*{Theorem~5.20}, which is the case $R=\mathbb{Z}[\omega]$ and $U=\{\omega^j:j=0,1,\dots,5\}$. It is proven in Section~\ref{section6}, and a similar theorem for infinite friezes is proven in Section~\ref{section7}.

To illustrate Theorem~\ref{thm3}, consider the periodic bi-infinite paths $\gamma=\langle\,\dotsc,1/0,1/1,2/1,\dotsc\rangle$ (of period 3) and $\delta=\langle\,\dotsc,1/2,3/2,0/2,2/2,4/2,\dotsc\rangle$ (of period 5) shown in Figure~\ref{fig5}. The paths $\langle\,\dotsc,(1,0),(1,1),(3,4),\dotsc\rangle$ and $\langle\,\dotsc,(1,2), (3,2), (0,2), (2,2), (4,2),\dotsc\rangle$ are lifts of $\gamma$ and $\delta$ to $\mathscr{E}_5$. Applying the map $\widetilde{\Phi}$ to this pair of lifted paths gives the $\text{SL}_2$-tiling over $\mathbb{Z}/5\mathbb{Z}$ shown in Figure~\ref{fig2}(b), up to equivalence.

\begin{figure}[H]
	\centering
	\begin{tikzpicture}[scale=0.7,line join=bevel,z=-5.5,font=\footnotesize,inner sep=0.2mm]
	\tdplotsetmaincoords{50}{90}
	
	\begin{scope}[tdplot_main_coords]
	\pgfmathsetmacro{\a}{2}
	\pgfmathsetmacro{\phi}{\a*(1+sqrt(5))/2}
	\path 
	coordinate[label=right:{\textcolor{black}{$\tfrac21$}}](A) at (0,\phi,\a)
	coordinate[label=right:{\textcolor{black}{$\tfrac02$}}](B) at (0,\phi,-\a)
	coordinate[label=left:{\textcolor{black}{$\tfrac01$}}](C) at (0,-\phi,\a)
	coordinate[label=left:{\textcolor{black}{$\tfrac42$}}](D) at (0,-\phi,-\a)
	coordinate[label={[yshift=2pt]above right:{\textcolor{black}{$\tfrac11$}}}](E) at (\a,0,\phi)
	coordinate[label={[yshift=15.5pt,xshift=-3pt]below:{\textcolor{black}{$\tfrac20$}}}](F) at (\a,0,-\phi)
	coordinate[label=above:{\textcolor{black}{$\tfrac10$}}](G) at (-\a,0,\phi)
	coordinate[label={[yshift=-2pt]below right:{\textcolor{black}{$\tfrac22$}}}](H) at (-\a,0,-\phi)
	coordinate[label={[xshift=-6pt,yshift=0.5pt]above left:{\textcolor{black}{$\tfrac32$}}}](I) at (\phi,\a,0)
	coordinate[label={[xshift=6pt,yshift=0.5pt]above right:{\textcolor{black}{$\tfrac12$}}}](J) at (\phi,-\a,0)
	coordinate[label={[xshift=-4pt]left:{\textcolor{black}{$\tfrac31$}}}](K) at (-\phi,\a,0)
	coordinate[label={[xshift=4pt]right:{\textcolor{black}{$\tfrac41$}}}](L) at (-\phi,-\a,0); 
	\draw[dashed,gray]    (B) -- (H) -- (K) -- cycle 
	(D) -- (L) -- (H) --cycle 
	(K) -- (L) -- (G) --cycle
	(A)--(K) (L)--(C) (F)--(H);
	
	\draw(A) -- (I) -- (B) --cycle 
	(F) -- (I) -- (B) --cycle 
	(F) -- (I) -- (J) --cycle
	(F) -- (D) -- (J) --cycle
	(C) -- (D) -- (J) --cycle
	(C) -- (E) -- (J) --cycle
	(I) -- (E) -- (J) --cycle
	(I) -- (E) -- (A) --cycle
	(G) -- (E) -- (A) --cycle
	(G) -- (E) -- (C) --cycle; 
	\end{scope}
		
	\draw[very thick=2,draw=col2,postaction={decorate,decoration={markings,mark=at position 0.6 with {\arrow[arrowstyle]{{latex}}}}}] (G) -- (E);
	\draw[very thick=2,draw=col2,postaction={decorate,decoration={markings,mark=at position 0.6 with {\arrow[arrowstyle]{{latex}}}}}] (E) -- (A);
	\draw[very thick=2,draw=col2,postaction={decorate,decoration={markings,mark=at position 0.6 with {\arrow[arrowstyle]{{latex}}}}}] (A) -- (G) node[col2,midway,above right]{$\gamma$};

	\draw[very thick=2,draw=col1,postaction={decorate,decoration={markings,mark=at position 0.55 with {\arrow[arrowstyle]{{latex}}}}}] (J) -- (I) node[midway,below,col1,xshift=-3pt,yshift=-2pt]{$\delta$};
	\draw[very thick=2,draw=col1,postaction={decorate,decoration={markings,mark=at position 0.55 with {\arrow[arrowstyle]{{latex}}}}}] (I) -- (B); 
	\draw[very thick=2,draw=col1,postaction={decorate,decoration={markings,mark=at position 0.55 with {\arrow[arrowstyle]{{latex}}}}}] (B) -- (H);
	\draw[very thick=2,draw=col1,postaction={decorate,decoration={markings,mark=at position 0.55 with {\arrow[arrowstyle]{{latex}}}}}] (H) -- (D);
	\draw[very thick=2,draw=col1,postaction={decorate,decoration={markings,mark=at position 0.55 with {\arrow[arrowstyle]{{latex}}}}}] (D) -- (J);
	
	\end{tikzpicture}
	\caption{Two periodic bi-infinite paths on $\mathscr{F}_5$}
	\label{fig5}
\end{figure}

Next we discuss two results of a similar type to Theorem~\ref{thm3} for friezes rather than $\text{SL}_2$-tilings.

We will use the function $\widetilde{\Psi}\colon\mathscr{P}\longrightarrow\SL$ obtained by composing $\widetilde{\Phi}$ with the map from $\mathscr{P}\longrightarrow\mathscr{P}\times\mathscr{P}$ given by $\gamma\longmapsto (\gamma,-\gamma)$. Explicitly, $\widetilde{\Psi}$ sends the path with vertices $(a_i,b_i)$ to the $\text{SL}_2$-tiling with entries $m_{i,j}=a_jb_i-b_ja_i$. 

Consider now some subgroup $U$ of $R^\times$, and let $n$ be a positive integer, at least 2. We say that two vertices $u$ and $v$ of $\mathscr{F}_{R,U}$ are \emph{equivalent} if there exists $\lambda\in R^\times$ such that $v=\lambda u$. We define $\mathscr{C}_{n,U}$ to be the set of paths of length $n$ in $\mathscr{F}_{R,U}$ with initial and final vertices that are equivalent. We write $\mathscr{C}_n$ for $\mathscr{C}_{n,\{1\}}$. There is an embedding of $\mathscr{C}_{n,U}$ into $\mathscr{P}_U$ in which the path $\langle v_0,v_1,\dots,v_n\rangle$ (where $v_n=\lambda v_0$, for some $\lambda \in R^\times$) is identified with the bi-infinite path with vertices $v_i$ satisfying $v_{i+n}=\lambda^{(-1)^i} v_i$, for $i\in\mathbb{Z}$. To reduce notation, we write $\mathscr{C}_{n,U}$ for the image under this embedding; thus, we may, for example, specify an element of $\mathscr{C}_{n,U}$ by a finite path but consider it to be a bi-infinite path. 

Let $\FR$ denote the subset of $\SL$ of tame friezes of width $n$, and let $\FRS$ denote the subset of tame semiregular friezes (second row 1's) of width $n$, where we identify a tame frieze with its extension to a tame $\text{SL}_2$-tiling. Consider the map from $\FR$ onto $\FRS$ given by $m_{i,j}\longmapsto \alpha^{(-1)^{i}}m_{i,j}$, where $\alpha=m_{1,0}$ (multiply even rows by $\alpha$ and odd rows by $\alpha^{-1}$). This induces a one-to-one correspondence $R^\times\backslash\FR\longrightarrow \FRS$. In this manner, every tame frieze of width $n$ can be normalised to give a tame semiregular frieze of width $n$.

We will see in Section~\ref{section8} that by restricting the domain of definition of $\widetilde{\Psi}$ to $\mathscr{C}_n$ we obtain a map $\widetilde{\Psi}\colon \mathscr{C}_n\longrightarrow \FRS$. We will also see that the $\widetilde{\Psi}$-images of any two lifts to $\mathscr{C}_n$ of two $\text{SL}_2(R)$-equivalent paths in $\mathscr{C}_{n,U}$ are $U$-equivalent in $\FRS$ under the action 
\[
m_{i,j}\longmapsto \lambda^{(-1)^{i}+(-1)^{j}}m_{i,j},
\]
where $\lambda\in U$ (this map preserves the second row of 1's). As a consequence, we obtain an induced map $\Psi_U\colon  \text{SL}_2(R)\backslash \mathscr{C}_{n,U}\longrightarrow U\backslash\FRS$ such that the following diagram commutes. 

\[
\begin{tikzcd}[row sep=4.5em,column sep=5.5em]
\mathscr{C}_n \arrow[r, "\widetilde{\Psi}",pos=0.5,shorten= 1.5ex,shorten >=1.5ex] \arrow[d]
& \FRS \arrow[d] \\
\text{SL}_2(R)\backslash\mathscr{C}_{n,U} \arrow[r,"\Psi_U" below]
& U\backslash \FRS
\end{tikzcd}
\]

Our second main result is that $\Psi_U$ is a bijection.

\begin{theorem}\label{thm4}
The map $\Psi_U$ is a one-to-one correspondence between
\[
\hspace*{-10pt}\textnormal{SL}_2(R)\Big\backslash\mleft\{\parbox{4.25cm}{\centering\textnormal{paths~of length~$n$~between equivalent~vertices~in~$\mathscr{F}_{R,U}$}}\mright\}\quad \longleftrightarrow\quad U\Big\backslash\mleft\{\parbox{3.85cm}{\centering\textnormal{tame~semiregular~friezes over~$R$~of~width~$n$}}\mright\}.
\]
\end{theorem}

This result generalises \cite{Sh2023}*{Theorem~1.6}, which is the special case $R=\mathbb{Z}$ and $U=\{\pm 1\}$.

To illustrate Theorem~\ref{thm4}, consider the tame semiregular frieze over $\mathbb{Z}/5\mathbb{Z}$ shown in Figure~\ref{fig56}.

\begin{figure}[ht]
	\centering
	\begin{tikzpicture}
	\node at (0,0) {
		\(
		\vcenter{\footnotesize
			\xymatrix @-0.8pc @!0 {
				&&  &&  &&  &&  &&  &&  &&  && \textcolor{darkgrey}{4} &&  && \\      
				& 0 && 0 && 0 && 0 && 0 && 0 && 0 && 0 && 0 && 0 & \\
				&& 1 && 1 && 1 && 1 && 1 && 1 && 1 && 1 && 1 && \\
				\cdots  & 4 && 1 && 3 && 3 && 1 && 4 && 2 && 2 && 4 && 1 & \cdots  \\
				&& 3 && 2 && 3 && 2 && 3 && 2 && 3 && 2 && 3 && \\      	
				& 0 && 0 && 0 && 0 && 0 && 0 && 0 && 0 && 0 && 0 & \\
			}
		}
		\)
	};
	\draw[rounded corners=3pt,shift={(1.5,-1.6)},rotate=45,draw=col2] (0,0) rectangle (3.375,1.2);
	\draw[rounded corners=3pt,draw=col4] (-3.8,-0.5) rectangle (3.8,0);
	
	\end{tikzpicture}	
	\caption{A tame semiregular frieze over $\mathbb{Z}/5\mathbb{Z}$}
	\label{fig56}
\end{figure}

The additional entry 4 above the frieze is a single entry from the extension of the frieze to an $\text{SL}_2$-tiling. Consider now the slanted box. By placing the upper row over the lower row from this box we obtain the path
\[
\frac20 \to \frac23 \to \frac12 \to \frac01 \to \frac40
\]
in $\mathscr{F}_5$, which is illustrated in Figure~\ref{fig4}. Note that $4/0$ and $1/0$ are equal in $\mathscr{F}_5$, as are $2/3$ and~$3/2$.

\begin{figure}[ht]
	\centering
	\begin{tikzpicture}[scale=0.7,line join=bevel,z=-5.5,font=\footnotesize,inner sep=0.2mm]
	\tdplotsetmaincoords{50}{90}
	
	\begin{scope}[tdplot_main_coords]
	\pgfmathsetmacro{\a}{2}
	\pgfmathsetmacro{\phi}{\a*(1+sqrt(5))/2}
	\path 
	coordinate[label=right:{\textcolor{col1}{$\tfrac21$}}](A) at (0,\phi,\a)
	coordinate[label=right:{\textcolor{col1}{$\tfrac02$}}](B) at (0,\phi,-\a)
	coordinate[label=left:{\textcolor{col1}{$\tfrac01$}}](C) at (0,-\phi,\a)
	coordinate[label=left:{\textcolor{col1}{$\tfrac42$}}](D) at (0,-\phi,-\a)
	coordinate[label={[yshift=2pt]above right:{\textcolor{col1}{$\tfrac11$}}}](E) at (\a,0,\phi)
	coordinate[label={[yshift=15.5pt,xshift=-3pt]below:{\textcolor{col1}{$\tfrac20$}}}](F) at (\a,0,-\phi)
	coordinate[label=above:{\textcolor{col1}{$\tfrac10$}}](G) at (-\a,0,\phi)
	coordinate[label={[yshift=-2pt]below right:{\textcolor{col1}{$\tfrac22$}}}](H) at (-\a,0,-\phi)
	coordinate[label={[xshift=-6pt,yshift=0.5pt]above left:{\textcolor{col1}{$\tfrac32$}}}](I) at (\phi,\a,0)
	coordinate[label={[xshift=6pt,yshift=0.5pt]above right:{\textcolor{col1}{$\tfrac12$}}}](J) at (\phi,-\a,0)
	coordinate[label={[xshift=-4pt]left:{\textcolor{col1}{$\tfrac31$}}}](K) at (-\phi,\a,0)
	coordinate[label={[xshift=4pt]right:{\textcolor{col1}{$\tfrac41$}}}](L) at (-\phi,-\a,0); 
	\draw[dashed,gray]    (B) -- (H) -- (K) -- cycle 
	(D) -- (L) -- (H) --cycle 
	(K) -- (L) -- (G) --cycle
	(A)--(K) (L)--(C) (F)--(H);
	
	\draw(A) -- (I) -- (B) --cycle 
	(F) -- (I) -- (B) --cycle 
	(F) -- (I) -- (J) --cycle
	(F) -- (D) -- (J) --cycle
	(C) -- (D) -- (J) --cycle
	(C) -- (E) -- (J) --cycle
	(I) -- (E) -- (J) --cycle
	(I) -- (E) -- (A) --cycle
	(G) -- (E) -- (A) --cycle
	(G) -- (E) -- (C) --cycle; 
	\end{scope}
		
	\draw[very thick=2,draw=col2,postaction={decorate,decoration={markings,mark=at position 0.6 with {\arrow[arrowstyle]{{latex}}}}}] (C) -- (G);
	\draw[very thick=2,draw=col2,postaction={decorate,decoration={markings,mark=at position 0.55 with {\arrow[arrowstyle]{{latex}}}}}] (J) -- (C);
	\draw[very thick=2,draw=col2,postaction={decorate,decoration={markings,mark=at position 0.6 with {\arrow[arrowstyle]{{latex}}}}}] (I) -- (J);
	\draw[very thick=2,draw=col2,postaction={decorate,decoration={markings,mark=at position 0.55 with {\arrow[arrowstyle]{{latex}}}}}] (F) -- (I);
	
	\end{tikzpicture}
	\caption{A path between equivalent vertices on $\mathscr{F}_5$}
	\label{fig4}
\end{figure}

The path $\gamma$ of Figure~\ref{fig4} corresponds to the frieze $\mathbf{F}$ of Figure~\ref{fig56} under $\Psi_U$ (with $U=\{\pm 1\}$). The path has length 4, and accordingly $\mathbf{F}$ has width 4. Considered as a bi-infinite path $\gamma$ satisfies $v_{i+4}=2^{(-1)^i}v_i$, for $i\in\mathbb{Z}$, so $v_{i+8}=-v_i=v_i$. Therefore $\gamma$ has period 8 and, correspondingly, the frieze $\mathbf{F}$ also has period 8 (indicated by the horizontal box in Figure~\ref{fig56}).

 We can obtain a similar result to Theorem~\ref{thm4} for tame \emph{regular} friezes over $R$; however, for these, single paths in $\mathscr{F}_{R,U}$ cannot alone be used to distinguish regular and non-regular friezes, unless $U=\{1\}$. Consequently, we state our third main result using the Farey complex $\mathscr{E}_R$. We define a \emph{semiclosed} path in $\mathscr{E}_R$ to be a path with initial vertex $v$ and final vertex $-v$, for some vertex $v$ in $\mathscr{E}_R$. The semiclosed paths form a subcollection of $\mathscr{C}_n$, and by restricting $\Psi$ (where $\Psi=\Psi_{\{1\}}$) to this subcollection we obtain the following theorem, proven in Section~\ref{section9}.

\begin{theorem}\label{thm5}
The map $\Psi$ is a one-to-one correspondence between
\[
\textnormal{SL}_2(R)\Big\backslash\mleft\{\parbox{3.1cm}{\centering\textnormal{semiclosed~paths~of length~$n$~in~$\mathscr{E}_R$}}\mright\}\quad \longleftrightarrow\quad \mleft\{\parbox{3.2cm}{\centering\textnormal{tame~regular~friezes over~$R$~of~width~$n$}}\mright\}.
\]
\end{theorem}
 
The strength of Theorems~\ref{thm4} and~\ref{thm5} is that they allow us to represent tame friezes by paths, for which we can take advantage of the algebraic, combinatorial, and geometric properties of Farey complexes. We give three applications; the first is about quiddity cycles, introduced by Coxeter in \cite{Co1971}. 

Consider any tame semiregular frieze $\mathbf{F}$ over a finite ring $R$. By applying Theorem~\ref{thm4} we can find a path $\gamma$ corresponding to $\mathbf{F}$ with vertices $v_i$ that satisfy $v_{i+n}=\lambda^{(-1)^i} v_i$, for $i\in\mathbb{Z}$ and $\lambda\in R^\times$. Since $R$ is finite, the path $\gamma$ is periodic, so $\mathbf{F}$ is also periodic (by contrast, the frieze of Figure~\ref{fig76} over the infinite ring $\mathbb{Q}$ is not periodic). A \emph{quiddity sequence} for $\mathbf{F}$ is a period $a_1,a_2,\dots,a_k$ from the third row $m_{i+1,i-1}$, for $i\in\mathbb{Z}$, of $\mathbf{F}$. For example, a quiddity sequence for the frieze of Figure~\ref{fig56} is indicated by the horizontal box.  The \emph{quiddity cycle} of $\mathbf{F}$ is the collection of all cyclic permutations of a quiddity sequence $a_1,a_2,\dots,a_k$. 


Conway and Coxeter's celebrated observation \cite{CoCo1973} was that the quiddity cycles of tame friezes with positive integer entries are in correspondence with the collection of triangulated polygons. Here we prove that, for a finite ring, \emph{any} finite sequence is a quiddity sequence of some frieze.

\begin{theorem}\label{thm34}
Any finite sequence in a finite ring $R$ is a quiddity sequence for some tame semiregular frieze over $R$.
\end{theorem}

Theorem~\ref{thm34} could be established from related results in the literature, such as \cite{Mo2015}*{Theorem 1.15}. The merit of our approach is that the correspondence with paths facilitates a short, intuitive proof, in Section~\ref{section10}.

For our second application of Theorems~\ref{thm4} and~\ref{thm5}, we enumerate tame friezes of a given width over a finite field. Theorem~\ref{thm4} reduces this task to that of enumerating closed paths in a complete graph, which is straightforward.

\begin{theorem}\label{thm9}
The number of tame friezes of width $n$ over a finite field of order $q$ is
\[
\frac{(q-1)(q^{n-1}+(-1)^n)}{q+1}.
\]
\end{theorem}

Using Theorems~\ref{thm4} and~\ref{thm9} we also reprove a result  of Morier-Genoud \cite{Mo2021}*{Theorem~1} on the number of tame \emph{regular} friezes of a given width over a finite field; see Section~\ref{section11}.

The third and most significant application of Theorems~\ref{thm4} and~\ref{thm5} concerns lifting tame $\text{SL}_2$-tilings and friezes from $\mathbb{Z}/N\mathbb{Z}$ to $\mathbb{Z}$. We will see in Section~\ref{section12} that \emph{any} tame $\text{SL}_2$-tiling  over $\mathbb{Z}/N\mathbb{Z}$ lifts to a tame $\text{SL}_2$-tiling  over $\mathbb{Z}$. In fact, we will see that it is possible to choose a lift with all entries positive. 

Lifting friezes is more complex; it is not always possible to lift a tame frieze over $\mathbb{Z}/N\mathbb{Z}$ to a tame frieze over $\mathbb{Z}$ of the same width. For example, the tame frieze over $\mathbb{Z}/5\mathbb{Z}$ shown in Figure~\ref{fig56} cannot be lifted to a tame frieze over $\mathbb{Z}$, since the entries 2 and 3 in the second-last row are not congruent to $\pm1$ modulo 5. For another example, consider the frieze $\mathbf{F}$ over $\mathbb{Z}/6\mathbb{Z}$ shown in Figure~\ref{fig34}. Alongside the frieze is a corresponding path $\gamma$ in $\mathscr{F}_6$ (specified by the correspondence $\Psi_U$ of Theorem~\ref{thm4}). This path winds once round the handle of the torus that underlies the surface complex of $\mathscr{F}_6$. Consequently, any lift $\tilde{\gamma}$ of $\gamma$ to $\mathscr{F}_\mathbb{Z}$ is not closed, in which case $\tilde{\gamma}$ does not specify a frieze, by Theorem~\ref{thm4}, and hence $\mathbf{F}$ cannot be lifted to a tame frieze over $\mathbb{Z}$. This reasoning will be made precise in Section~\ref{section12}.

\begin{figure}[ht]
\begin{subfigure}[c]{0.45\textwidth}
\(\cdots\
\begin{matrix} 
0 && 0 && 0 && 0 && 0 &   \\ 
&1 && 1 && 1 && 1 && 1    \\ 
2 && 4 && 2 && 4 && 2 &   \\
&1 && 1 && 1 && 1 && 1    \\ 
0 && 0 && 0 && 0 && 0 &    
\end{matrix}
\ \cdots
\)
\end{subfigure}
\hfill
\begin{subfigure}[c]{0.45\textwidth}
\centering
\begin{tikzpicture}[scale=1.3,line join=bevel,z=-5.5,font=\footnotesize,inner sep=0.2mm]
	\coordinate[label={[shift={(2:0.33)}]$\tfrac10$}] (Z) at (0,0);
	\coordinate[label={[shift={(2:0.33)}]$\tfrac01$}](A) at (0:1);
	\coordinate[label=right:\textcolor{col3}{$\tfrac13$}](B) at (0:2);
	\coordinate[label={[right,shift={(35:0.08)}]{{$\tfrac12$}}}](C) at (30:1.73);
	\coordinate[label={[shift={(2:0.33)}]$\tfrac11$}](D) at (60:1);
	\coordinate[label={[above,shift={(90:0.05)}]{{$\tfrac23$}}}](E) at (60:2);
	\coordinate[label={[above,shift={(90:0.05)}]{{$\tfrac32$}}}](F) at (90:1.73);
	\coordinate[label={[shift={(178:0.33)}]$\tfrac21$}](G) at (120:1);
	\coordinate[label={[above,shift={(90:0.05)}]{{$\tfrac13$}}}](H) at (120:2);
	\coordinate[label={[left,shift={(145:0.08)}]{{$\tfrac14$}}}](I) at (150:1.73);
	\coordinate[label={[shift={(178:0.33)}]$\tfrac31$}](J) at (180:1);
	\coordinate[label=left:\textcolor{col1}{$\tfrac23$}](K) at (180:2);
	\coordinate[label={[left,shift={(215:0.08)}]{{$\tfrac12$}}}](L) at (210:1.73);
	\coordinate[label={[shift={(235:0.53)}]$\tfrac41$}](M) at (240:1);
	\coordinate[label={[below,shift={(-90:0.05)}]{{$\tfrac13$}}}](N) at (240:2);
	\coordinate[label={[below,shift={(-90:0.05)}]{{$\tfrac32$}}}](O) at (270:1.73);
	\coordinate[label={[shift={(305:0.53)}]$\tfrac51$}](P) at (300:1);
	\coordinate[label={[below,shift={(-90:0.05)}]{{$\tfrac23$}}}](Q) at (300:2);
	\coordinate[label={[right,shift={(-35:0.08)}]{{$\tfrac14$}}}](R) at (330:1.73);
	
	\draw (E)--(H) (C)--(I) (B)--(K) (R)--(L) (Q)--(N)
	(Q)--(B) (O)--(C) (N)--(E) (L)--(F) (K)--(H)
	(E)--(B) (F)--(R) (H)--(Q) (I)--(O) (K)--(N);

	\draw[very thick=2,draw=col2,postaction={decorate,decoration={markings,mark=at position 0.6 with {\arrow[arrowstyle]{{latex}}}}}] (J) -- (K);
	\draw[very thick=2,draw=col2,postaction={decorate,decoration={markings,mark=at position 0.6 with {\arrow[arrowstyle]{{latex}}}}}] (G) -- (J);
	\draw[very thick=2,draw=col2,postaction={decorate,decoration={markings,mark=at position 0.6 with {\arrow[arrowstyle]{{latex}}}}}] (D) -- (G);
	\draw[very thick=2,draw=col2,postaction={decorate,decoration={markings,mark=at position 0.6 with {\arrow[arrowstyle]{{latex}}}}}] (E) -- (D);
	
	\end{tikzpicture}
\end{subfigure}
\caption{Tame regular frieze over $\mathbb{Z}/6\mathbb{Z}$ with corresponding closed path in $\mathscr{F}_6$}
\label{fig34}
\end{figure}        
 
To classify which tame friezes over $\mathbb{Z}/N\mathbb{Z}$ lift to friezes over $\mathbb{Z}$, we introduce the following topological concepts. We say that a closed path in a Farey complex $\mathscr{F}_R$ is \emph{strongly contractible} if it can be transformed to a point by applying a finite number of the following two elementary homotopies. The first elementary homotopy is the removal of a single spur; that is, we replace a subpath $\langle v,u,v\rangle$ of a closed path $\gamma$ with the subpath $\langle v \rangle$, thereby decreasing the length of $\gamma$ by 2. The second elementary homotopy is an elementary deformation over a triangular face of $\mathscr{F}_R$; that is, we replace a subpath $\langle u,v,w\rangle$ of $\gamma$ with $\langle u,w\rangle$, where $u$, $v$, and $w$ are mutually adjacent vertices, thereby decreasing the length of $\gamma$ by 1.

Notice that being strongly contractible is \emph{not} the same as being null homotopic (in the usual sense) because for strongly contractible paths we are only allowed to remove edges; we cannot add spurs or triangles. For example, consider the closed paths $\gamma$ and $\delta$ shown in Figure~\ref{fig94}; both are null homotopic, however, $\gamma$ is strongly contractible but $\delta$ is not.
  
\begin{figure}[H]
\centering
\begin{tikzpicture}[scale=1.3,line join=bevel,z=-5.5,font=\footnotesize,inner sep=0.2mm]
	\coordinate[label={[shift={(2:0.33)}]$\tfrac10$}] (Z) at (0,0);
	\coordinate[label={[shift={(2:0.33)}]$\tfrac01$}](A) at (0:1);
	\coordinate[label=right:\textcolor{col3}{$\tfrac13$}](B) at (0:2);
	\coordinate[label={[right,shift={(35:0.08)}]{{$\tfrac12$}}}](C) at (30:1.73);
	\coordinate[label={[shift={(2:0.33)}]$\tfrac11$}](D) at (60:1);
	\coordinate[label={[above,shift={(90:0.05)}]{{$\tfrac23$}}}](E) at (60:2);
	\coordinate[label={[above,shift={(90:0.05)}]{{$\tfrac32$}}}](F) at (90:1.73);
	\coordinate[label={[shift={(178:0.33)}]$\tfrac21$}](G) at (120:1);
	\coordinate[label={[above,shift={(90:0.05)}]{{$\tfrac13$}}}](H) at (120:2);
	\coordinate[label={[left,shift={(145:0.08)}]{{$\tfrac14$}}}](I) at (150:1.73);
	\coordinate[label={[shift={(178:0.33)}]$\tfrac31$}](J) at (180:1);
	\coordinate[label=left:\textcolor{col1}{$\tfrac23$}](K) at (180:2);
	\coordinate[label={[left,shift={(215:0.08)}]{{$\tfrac12$}}}](L) at (210:1.73);
	\coordinate[label={[shift={(235:0.53)}]$\tfrac41$}](M) at (240:1);
	\coordinate[label={[below,shift={(-90:0.05)}]{{$\tfrac13$}}}](N) at (240:2);
	\coordinate[label={[below,shift={(-90:0.05)}]{{$\tfrac32$}}}](O) at (270:1.73);
	\coordinate[label={[shift={(305:0.53)}]$\tfrac51$}](P) at (300:1);
	\coordinate[label={[below,shift={(-90:0.05)}]{{$\tfrac23$}}}](Q) at (300:2);
	\coordinate[label={[right,shift={(-35:0.08)}]{{$\tfrac14$}}}](R) at (330:1.73);
	
	\draw (E)--(H) (C)--(I) (B)--(K) (R)--(L) (Q)--(N)
	(Q)--(B) (O)--(C) (N)--(E) (L)--(F) (K)--(H)
	(E)--(B) (F)--(R) (H)--(Q) (I)--(O) (K)--(N);

	\draw[very thick=2,draw=col2,postaction={decorate,decoration={markings,mark=at position 0.6 with {\arrow[arrowstyle]{{latex}}}}}] (F) -- (D);
	\draw[very thick=2,draw=col2,postaction={decorate,decoration={markings,mark=at position 0.6 with {\arrow[arrowstyle]{{latex}}}}}] (D) -- (Z);
	\draw[very thick=2,draw=col2,postaction={decorate,decoration={markings,mark=at position 0.6 with {\arrow[arrowstyle]{{latex}}}}}] (Z) -- (G);
	\draw[very thick=2,draw=col2,postaction={decorate,decoration={markings,mark=at position 0.6 with {\arrow[arrowstyle]{{latex}}}}}] (G) -- (J);
	\draw[very thick=2,draw=col2,postaction={decorate,decoration={markings,mark=at position 0.6 with {\arrow[arrowstyle]{{latex}}}}}] (J) -- (I);
	\draw[very thick=2,draw=col2,postaction={decorate,decoration={markings,mark=at position 0.6 with {\arrow[arrowstyle]{{latex}}}}}] (I) -- (H);
	\draw[very thick=2,draw=col2,postaction={decorate,decoration={markings,mark=at position 0.6 with {\arrow[arrowstyle]{{latex}}}}}] (H) -- (F) node[col2,midway,above,yshift=2pt]{$\gamma$};
	
	\draw[very thick=2,draw=col1,postaction={decorate,decoration={markings,mark=at position 0.6 with {\arrow[arrowstyle]{{latex}}}}}] (Z) -- (A);
	\draw[very thick=2,draw=col1,postaction={decorate,decoration={markings,mark=at position 0.6 with {\arrow[arrowstyle]{{latex}}}}}] (A) -- (R);
	\draw[very thick=2,draw=col1,postaction={decorate,decoration={markings,mark=at position 0.6 with {\arrow[arrowstyle]{{latex}}}}}] (R) -- (Q);
	\draw[very thick=2,draw=col1,postaction={decorate,decoration={markings,mark=at position 0.6 with {\arrow[arrowstyle]{{latex}}}}}] (Q) -- (O) node[col1,midway,below,yshift=-2pt]{$\delta$};;
	\draw[very thick=2,draw=col1,postaction={decorate,decoration={markings,mark=at position 0.6 with {\arrow[arrowstyle]{{latex}}}}}] (O) -- (M);
	\draw[very thick=2,draw=col1,postaction={decorate,decoration={markings,mark=at position 0.6 with {\arrow[arrowstyle]{{latex}}}}}] (M) -- (Z);
\end{tikzpicture}

\caption{A strongly contractible closed path $\gamma$ and a not strongly contractible closed path $\delta$}
\label{fig94}
\end{figure}        

Since $\pm 1$ are the only units in $\mathbb{Z}$, any frieze $\mathbf{F}$ over $\mathbb{Z}/N\mathbb{Z}$ that lifts to a frieze $\widetilde{\mathbf{F}}$ over $\mathbb{Z}$ must have entries $1$ or $-1$ in its second and second-last rows. Notice that $\widetilde{\mathbf{F}}$ is a lift of $\mathbf{F}$ if and only if $-\widetilde{\mathbf{F}}$ is a lift of $-\mathbf{F}$, so we can restrict our attention to semiregular friezes. We say that $\gamma$ is a path in $\mathscr{F}_N$ \emph{corresponding} to the tame semiregular frieze $\mathbf{F}$ if the image under $\Psi_{\{\pm 1\}}$ of the $\text{SL}_2(\mathbb{Z}/N\mathbb{Z})$-orbit of $\gamma$ is $\mathbf{F}$.

\begin{theorem}\label{thm100}
A tame semiregular frieze $\mathbf{F}$ over $\mathbb{Z}/N\mathbb{Z}$ lifts to a tame frieze over $\mathbb{Z}$ of the same width if and only if any path $\gamma$ in $\mathscr{F}_N$ corresponding to $\mathbf{F}$ is a strongly contractible closed path.
\end{theorem}    
        
A consequence of Theorem~\ref{thm100} is that all tame friezes over $\mathbb{Z}/2\mathbb{Z}$ and $\mathbb{Z}/3\mathbb{Z}$ lift to tame friezes over $\mathbb{Z}$, because these rings are both fields (so equivalent vertices are equal) and the graphs of $\mathscr{F}_2$ and $\mathscr{F}_3$ are complete graphs, so all closed paths in these graphs are strongly contractible. However, the same cannot be said of $\mathbb{Z}/4\mathbb{Z}$ since the closed path
\[
\frac01 \to \frac11 \to \frac21 \to \frac31 \to \frac01 
\]
is not strongly contractible.

\section{Group actions on Farey complexes}\label{section2}

In this section we discuss some of the properties of the action of $\text{SL}_2(R)$ on $\mathscr{F}_{R,U}$ that were stated but not proven in the introduction. These properties are elementary, so we skip some details.

Let $R$ be a ring and let $U$ be a subgroup of $R^\times$. We consider the left group action $\theta\colon \text{SL}_2(R)\times \mathscr{F}_{R,U}\longrightarrow \mathscr{F}_{R,U}$ given by the rule $(A,x/y)\longmapsto (ax+by)/(cx+dy)$, where 
\[
A=\begin{pmatrix}a & b\\ c&d\end{pmatrix}\in \textnormal{SL}_2(R).
\]
We write $\theta_A(x/y)$ for $\theta(A,x/y)$. That this truly is a group action (for which $\theta_A$ is an automorphism of $\mathscr{F}_{R,U}$) will be established shortly. Remember that the formal fraction $x/y$ represents the orbit $U(x,y)=\{\lambda (x,y):\lambda \in U\}$. With this orbit notation, we have $\theta_A(U(x,y))=U(ax+by,cx+dy)$. 

The symbol $\theta$ for the group action is used in this section alone, to facilitate the proofs. Outside this section we abandon $\theta$ and simply write $Av$ for $\theta_A(v)$ and so forth.

\begin{proposition}
The function $\theta$ is a left group action of $\textnormal{SL}_2(R)$ on $\mathscr{F}_{R,U}$.
\end{proposition}
\begin{proof}
First we check that if $(x,y)$ is a unimodular pair, then so is $(x',y')=(ax+by,cx+dy)$. We know that there exist $r,s\in R$ with $rx+sy=1$. We define $r',s'\in R$ by 
\(\begin{pmatrix}r' & s'\end{pmatrix}= \begin{pmatrix}r & s\end{pmatrix}A^{-1}\). Then
\[
r'x'+s'y'=\begin{pmatrix}r' & s'\end{pmatrix}\begin{pmatrix}x' \\ y'\end{pmatrix}=\begin{pmatrix}r & s\end{pmatrix}A^{-1}A\begin{pmatrix}x \\ y\end{pmatrix}=1,
\]
so $(x',y')$ is a unimodular pair.

Next we  check that if $x_1/y_1=x_2/y_2$, then  $(ax_1+by_1)/(cx_1+dy_1)=(ax_2+by_2)/(cx_2+dy_2)$. This is indeed so, because if $x_1/y_1=x_2/y_2$ then there exists $\lambda\in U$ with $(x_1,y_1)=\lambda(x_2,y_2)$, and the rest of the argument follows easily. We have now shown that $\theta$ is a function. The left group action axioms can readily be seen to be satisfied.

It remains to show that $\theta_A$ is a graph automorphism, for each $A\in\text{SL}_2(R)$. For this we must show that $\theta_A$ preserves incidence in $\mathscr{F}_{R,U}$. Suppose that there is a directed edge from $x_1/y_1$ to $x_2/y_2$. Then
\[
x_1y_2-y_1x_2 = \begin{pmatrix}x_1 & y_1\end{pmatrix}J\begin{pmatrix}x_2 \\ y_2\end{pmatrix}\in U,\quad\text{where } J=\begin{pmatrix}0 & 1\\ -1&0\end{pmatrix}\!.
\]
Let $(x_1',y_1')=(ax_1+by_1,cx_1+dy_1)$ and $(x_2',y_2')=(ax_2+by_2,cx_2+dy_2)$. Then
\[
x_1'y_2'-y_1'x_2'=\begin{pmatrix}x_1' & y_1'\end{pmatrix}J\begin{pmatrix}x_2' \\ y_2'\end{pmatrix}=\begin{pmatrix}x_1 & y_1\end{pmatrix}A^TJA\begin{pmatrix}x_2 \\ y_2\end{pmatrix}\in U,
\]
because $A^TJA=J$ (where $A^T$ is the transpose of $A$). Hence there is a directed edge from $x_1'/y_1'$ to $x_2'/y_2'$, as required.
\end{proof}

Typically, the action $\theta$ is not faithful. Following the usual terminology for group actions, we consider the kernel of $\theta$ to be the subgroup of $\text{SL}_2(R)$ of matrices $A$ that satisfy $\theta_A(v)=v$, for all vertices $v\in \mathscr{F}_{R,U}$. 

\begin{lemma}
The kernel of $\theta$ is
\[
\left\{ \begin{pmatrix}\lambda & 0\\ 0&\lambda\end{pmatrix} : \lambda\in U,\, \lambda^2=1\right\}.\qedhere
\]
\end{lemma}

This can be proved by checking images of $1/0$, $0/1$, and $1/1$; we omit the details.

We remark that when $U=R^\times$, the kernel of $\theta$ is the group of all scalar matrices in $\text{SL}_2(R)$. In this case we obtain a faithful action of $\text{PSL}_2(R)$ on $\mathscr{F}_{R,R^\times}$ (and the vertices of $\mathscr{F}_{R,R^\times}$ form the projective line over $R$).

The action $\theta$ is transitive on vertices, in the sense that, for any $v,w\in \mathscr{F}_{R,U}$ there exists $A\in\text{SL}_2(R)$ with $\theta_A(v)=w$. In fact, if we consider $\theta$ to act not on the vertices of $\mathscr{F}_{R,U}$ but on directed edges (in the obvious way) then we can obtain the following stronger statement.

\begin{proposition}\label{prop2}
The action $\theta$ is transitive on directed edges.
\end{proposition}
\begin{proof}
Consider any directed edge from $a/b$ to $c/d$. Let $\lambda=ad-bc$; then $\lambda\in U$. We define $a'=\lambda^{-1}a$ and $b'=\lambda^{-1}b$. Then
\[
A=\begin{pmatrix}a' & c\\ b'&d\end{pmatrix}\in \textnormal{SL}_2(R),
\]
and $\theta_A(1/0)=a'/b'=a/b$ and $\theta_A(0/1)=c/d$. Consequently, any directed edge is the image of the directed edge from $1/0$ to $0/1$, and the proposition then follows from the group action axioms.
\end{proof}

A consequence of Proposition~\ref{prop2} is that, when $R$ is finite, the directed graph $\mathscr{F}_{R,U}$ is regular, which means that all indegrees and outdegrees of vertices are equal. Since all the directed edges with source $1/0$ have the form $1/0\to a/1$, for $a\in R$, we see that the outdegree of each vertex is equal to the order of $R$, and the same can be said of the indegree. Notice that the indegree and outdegree are independent of $U$.

Consider now the map $\pi_{U}\colon \mathscr{E}_{R}\longrightarrow \mathscr{F}_{R,U}$ given by $(x,y)\longmapsto U(x,y)$, where $U(x,y)=\{\lambda (x,y):\lambda \in U\}$, from the introduction. It is straightforward to check that this is a covering map. It was stated in the introduction that this map is equivariant under $\theta$; we now prove this fact.

\begin{proposition}\label{prop3}
For any $A\in\textnormal{SL}_2(R)$, we have $\pi_{U} \theta_A=\theta_A \pi_{U}$.
\end{proposition}
\begin{proof}
Let $(x,y)$ be any vertex of $\mathscr{E}_{R}$. Then $\theta_A(x,y)=U(ax+by,cx+d)$, so
\[
\pi_{U} \theta_A(x,y)=U(ax+by,cx+d)=\theta_A(U(x,y))=\theta_A \pi_{U}(x,y),
\]
as required.
\end{proof}

There was a sleight of hand in the statement of Proposition~\ref{prop3}, in that the two occurrences of $\theta_A$ differ: the left one acts on $\mathscr{E}_{R}$ and the right one acts on $\mathscr{F}_{R,U}$.

\section{Paths and itineraries}\label{section3}

We gather here some basic properties of paths in Farey complexes. The first result we record is about lifting paths from the Farey complex $\mathscr{F}_{R,U}$ of the ring $R$ with units $U$ to the Farey complex $\mathscr{E}_R$ under the covering map $\pi_U\colon\mathscr{E}_R\longrightarrow \mathscr{F}_{R,U}$.

\begin{lemma}\label{lem61}
Let $\gamma$ be a bi-infinite path in $\mathscr{F}_{R,U}$. Then $\gamma$ lifts to a bi-infinite path $\tilde{\gamma}$ in $\mathscr{E}_{R}$ with vertices $(a_i,b_i)$, and every other lift of $\gamma$ to $\mathscr{E}_R$ has the form $\lambda^{(-1)^i}(a_i,b_i)$, for $\lambda\in U$.
\end{lemma}
\begin{proof}
Certainly $\gamma$ has a lift to some path $\tilde{\gamma}$ because $\pi_U$ is a covering map. Let $(a_i,b_i)$ be the vertices of $\tilde{\gamma}$. Suppose that another lift of $\gamma$ has vertices $(c_i,d_i)$. Then $(a_i,b_i)$ and $(c_i,d_i)$ project to the same vertex in $\mathscr{F}_{R,U}$, so $(c_i,d_i)=\lambda_i(a_i,b_i)$, where $\lambda_i\in U$. Consequently,
\[
\lambda_i\lambda_{i+1}=\lambda_i\lambda_{i+1}(a_ib_{i+1}-b_{i+1}a_i)=c_id_{i+1}-d_ic_{i+1}=1.
\]
It follows that $\lambda_i=\lambda_0^{(-1)^i}$, for $i\in\mathbb{Z}$, as required.
\end{proof}

Next we explore the concept of itineraries of paths, which appeared in \cite{Sh2023} for the Farey complex $\mathscr{F}_\mathbb{Z}$. Itineraries are well known and explored in other contexts: they are closely related to (for example) quiddity sequences for friezes, partial quotient sequences for continued fractions, and integer coefficient sequences $(a_i)$ for discrete Sturm--Liouville equations $V_{i+1}=a_iV_i-V_{i-1}$. Coxeter appreciated the relationship between quiddity sequences and partial quotient sequences in his original work on friezes \cite{Co1971}, and this relationship was exploited by Conway and Coxeter in \cite{CoCo1973}. The connection between quiddity sequences and coefficient sequences for discrete Sturm--Liouville equations (and certain moduli spaces of curves with marked points) is expounded in the appendix of \cite{MoOvTa2012}; see also the survey \cite{Mo2015} where these sequences appear in all their various guises.

Our approach to itineraries develops that of \cites{BeHoSh2012,Ha2022,MoOvTa2015}. In those works, the convergents of a continued fraction determine a path in the usual Farey complex and the partial quotients of the continued fraction are what we call the itinerary of the path. The results in this section should seem familiar to those acquainted with the theory of continued fractions. A crucial difference is that we work with arbitrary commutative rings; it is a noteworthy feature of this theory that it does not require the cancellation property of integral domains.

\begin{definition}
Let $\gamma$ be a bi-infinite path with vertices $(a_i,b_i)$ in the Farey complex $\mathscr{E}_R$ of the ring $R$. The \emph{itinerary} of $\gamma$, denoted $\Sigma(\gamma)$, is the bi-infinite sequence $(e_i)$ in $R$ given by
\[
e_i=a_{i-1}b_{i+1}-b_{i-1}a_{i+1},
\]
for $i\in\mathbb{Z}$.
\end{definition}

The following alternative characterisation of itineraries (which is reminiscent of similar results in the theory of continued fractions) will prove useful.

\begin{lemma}\label{lem85}
The itinerary of a bi-infinite path $\gamma$ in $\mathscr{E}_R$ with vertices $(a_i,b_i)$ is the unique bi-infinite sequence $(e_i)$ in $R$ that satisfies
\[
a_{i-1}+a_{i+1}=e_ia_i\quad\text{and}\quad b_{i-1}+b_{i+1}=e_ib_i,
\]
for $i\in\mathbb{Z}$.
\end{lemma}
\begin{proof}
Let $(e_i)$ be the itinerary of $\gamma$. Then
\[
a_{i-1}+a_{i+1}=a_{i-1}(a_{i}b_{i+1}-b_{i}a_{i+1})+a_{i+1}(a_{i-1}b_{i}-b_{i-1}a_{i})=e_ia_i,
\]
and similarly $b_{i-1}+b_{i+1}=e_ib_i$. For uniqueness, let $(r_i)$ be a bi-infinite sequence that satisfies $a_{i-1}+a_{i+1}=r_ia_i$ and $b_{i-1}+b_{i+1}=r_ib_i$, for $i\in\mathbb{Z}$. Then $(r_i-e_i)a_i=(r_i-e_i)b_i=0$, so 
\[
r_i-e_i=(r_i-e_i)(a_{i}b_{i+1}-b_{i}a_{i+1})=0.
\]
Hence $(r_i)$ and $(e_i)$ coincide.
\end{proof}

We recall that $\mathscr{P}$ denotes the collection of bi-infinite paths in $\mathscr{E}_R$. The next lemma shows that the itinerary map $\Sigma\colon \mathscr{P}\longrightarrow R^\mathbb{Z}$ is surjective. 

\begin{lemma}\label{lem17}
For any bi-infinite sequence  $(e_i)$ in $R$ and directed edge $(a^*,b^*)\to (c^*,d^*)$ in $\mathscr{E}_R$ there is a unique bi-infinite path $(a_i,b_i)$ in $\mathscr{E}_R$ with $(a_0,b_0)=(a^*,b^*)$, $(a_1,b_1)=(c^*,d^*)$, and itinerary $(e_i)$.
\end{lemma}
\begin{proof}
First we establish the existence of a path with the specified properties. We define $(a_i,b_i)$, for $i\in\mathbb{Z}$, from the equations
\[
\begin{pmatrix}a_i &  a_{i+1}\\ b_i & b_{i+1}\end{pmatrix}=\begin{pmatrix}a_{i-1} &  a_{i}\\ b_{i-1} & b_{i}\end{pmatrix}\begin{pmatrix}0 & -1 \\ 1 & e_i\end{pmatrix},\quad \begin{pmatrix}a_{0} &  a_{1}\\ b_{0}& b_{1}\end{pmatrix}=\begin{pmatrix}a^*&c^*\\ b^*&d^*\end{pmatrix}.
\]
By taking determinants we see that $(a_i,b_i)$ determine a path $\gamma$. Moreover, since $a_{i-1}+a_{i+1}=e_ia_i$ and $b_{i-1}+b_{i+1}=e_ib_i$ it follows from Lemma~\ref{lem85} that $(e_i)$ is the itinerary of $\gamma$.

Next we show that this is the unique path with the specified properties. For this we can apply the recurrence relations $a_{i-1}+a_{i+1}=e_ia_i$ and $b_{i-1}+b_{i+1}=e_ib_i$ to deduce that $(a_i,b_i)$ is determined uniquely from the itinerary $(e_i)$ and directed edge $(a^*,b^*)\to (c^*,d^*)$.
\end{proof}

The itinerary map $\Sigma$ is invariant under $\text{SL}_2(R)$, in the sense that $\Sigma(A\gamma)=\Sigma(\gamma)$, for $A\in\gamma$. To see that this is so, notice that
\[
e_i=\begin{pmatrix}a_{i-1} & b_{i-1}\end{pmatrix}J\begin{pmatrix}a_{i+1} \\ b_{i+1}\end{pmatrix},\quad \text{where }J=\begin{pmatrix}0 & 1\\ -1 & 0 \end{pmatrix}.
\]
Then the itinerary $(e_i')$ of $A\gamma$ satisfies
\[
e_i'= \begin{pmatrix}a_{i-1} & b_{i-1}\end{pmatrix}A^TJA\begin{pmatrix}a_{i+1} \\ b_{i+1}\end{pmatrix}=e_i,
\]
because $A^TJA=J$, as required. It follows that $\Sigma$ induces a map $\text{SL}_2(R)\backslash \mathscr{P}\longrightarrow R^\mathbb{Z}$. In fact, this map is bijective. 

\begin{theorem}\label{thm55}
The map $\Sigma$ induces a one-to-one correspondence between $\textnormal{SL}_2(R)\backslash \mathscr{P}$ and $R^\mathbb{Z}$.
\end{theorem}
\begin{proof}
We have only to show that the induced map is injective. Suppose then that $\gamma$ and $\delta$ are bi-infinite paths in $\mathscr{P}$ with the same itinerary $(r_i)$. Let $\gamma$ have vertices $(a_i,b_i)$ and $\delta$ have vertices $(c_j,d_j)$. By Proposition~\ref{prop2}, there is a matrix $A\in\text{SL}_2(R)$ that maps the directed edge $(a_0,b_0)\to (a_1,b_1)$ to the directed edge $(c_0,d_0)\to (c_1,d_1)$. Consequently, $A\gamma$ and $\delta$ have the same itinerary and the same 0th and 1st vertices. Hence they are equal, by Lemma~\ref{lem17}, so the $\text{SL}_2(R)$-orbits of $\gamma$ and $\delta$ are equal.
\end{proof}

Theorem~\ref{thm55} tells us that, up to $\text{SL}_2(R)$ equivalence, bi-infinite paths correspond to bi-infinite sequences in $R^\mathbb{Z}$. This correspondence parallels the correspondence between convergents and partial quotients of a continued fraction. The theorem will not be used explicitly in the sequel.

We have explored itineraries on the Farey complex $\mathscr{E}_R$, and a similar account could be given for $\mathscr{F}_R$. For other Farey complexes $\mathscr{F}_{R,U}$, with a larger group of units $U$, the story is complicated by the presence of squares of units in the definition of an itinerary (which vanish if $U=\{\pm 1\}$); there is no need for us to go into this. 

We remark that  itineraries can be defined for finite or half-infinite paths, and there are analogues of the results above in these restricted settings.

\section{Graph distance on Farey complexes}\label{section4}

In this section we prove the result promised in the introduction that between any two vertices of a Farey complex over a finite ring there is a path of length at most 3. To this end, let $R$ be a finite ring and $U$ a group of units in $R$. We denote by $\delta(u,v)$ the length of the shortest path from a vertex $u$ to another vertex $v$ in $\mathscr{F}_{R,U}$. The \emph{diameter} of $\mathscr{F}_{R,U}$ is the maximum of $\delta(u,v)$ among all pairs of vertices $u$ and $v$ in $\mathscr{F}_{R,U}$. We will prove that the diameter of $\mathscr{F}_{R,U}$ is finite, from which it follows that $\mathscr{F}_{R,U}$ is connected.

\begin{theorem}\label{thm1}
The diameter of the Farey complex $\mathscr{E}_R$ of a finite ring $R$ satisfies
\[ 
\diam{\mathscr{E}_R} =
\begin{cases}
 1 & \text{if $R$ is $\mathbb{Z}/2\mathbb{Z}$},\\
 2 & \text{if $R$ is $\mathbb{Z}/3\mathbb{Z}$, $\mathbb{Z}/4\mathbb{Z}$, or $(\mathbb{Z}/2\mathbb{Z})^n$, for $n>1$},\\
 3 & \text{otherwise.}
\end{cases}
\]
\end{theorem}

Here $(\mathbb{Z}/2\mathbb{Z})^n$ is the direct product of $n$ copies of $\mathbb{Z}/2\mathbb{Z}$. Since the map $\mathscr{E}_R\longrightarrow \mathscr{F}_{R,U}$ is a covering map, it follows that $\diam{\mathscr{F}_{R,U}}\leq \diam{\mathscr{E}_R}\leq 3$; that is, the diameter of \emph{any} Farey complex over a finite ring is at most 3.

In order to prove Theorem~\ref{thm1}, we first classify Farey complexes $\mathscr{E}_R$ of diameter 1. Before doing so, it is worth observing that if the characteristic of $R$ (denoted $\cha R$) is $2$, then $\mathscr{E}_R=\mathscr{F}_R$.

\begin{lemma}\label{lem1}
The Farey complex $\mathscr{E}_R$ of a finite ring $R$ has diameter 1 if and only if $R=\mathbb{Z}/2\mathbb{Z}$.
\end{lemma}
\begin{proof}
Certainly $\mathscr{E}_{\mathbb{Z}/2\mathbb{Z}}$ has diameter 1 because it is the complete directed graph on three vertices. Suppose now that $R\neq \mathbb{Z}/2\mathbb{Z}$. Then $R$ contains an element $x$ other than $0$ or $1$, in which case there is no directed edge from $(1,0)$ to $(1,x)$, so $\mathscr{E}_R$ does not have diameter 1, as required.
\end{proof}

A ring $R$ is a \emph{local ring} if it contains a unique maximal ideal $M$. We will use the well-known property of a local ring that if $x,y\in R$ and $x+y=1$, then one of $x$ or $y$ must be a unit. To see why this property holds, suppose  that $x$ and $y$ are elements of a local ring $R$ that are not units. Then $Rx\neq R$ and $Ry\neq R$, and hence $Rx$ and $Ry$ are contained in the unique maximal ideal $M$. Consequently, $Rx+Ry\neq R$, so $x+y\neq 1$. 

\begin{lemma}\label{lem2}
The only finite local rings with no units besides $\pm 1$ are $\mathbb{Z}/2\mathbb{Z}$, $\mathbb{Z}/3\mathbb{Z}$, and $\mathbb{Z}/4\mathbb{Z}$. 
\end{lemma}
\begin{proof}
Suppose that $R$ is a finite local ring with no units other than $\pm 1$. Suppose also that $R\neq \mathbb{Z}/2\mathbb{Z}, \mathbb{Z}/3\mathbb{Z}$. Then there exists an element $x$ of $R$ other than $0$ or $\pm 1$. Since $R$ is a local ring, $1-x$ is a unit, in which case $1-x=-1$, so $x=2$. Hence $R=\mathbb{Z}/4\mathbb{Z}$. 
\end{proof}

It is now straightforward to classify the finite local rings for which $\mathscr{E}_R$ has diameter 2.

\begin{lemma}\label{lem3}
The Farey complex $\mathscr{E}_R$ of a finite local ring $R$ has diameter 2 if and only if $R$ is $\mathbb{Z}/3\mathbb{Z}$ or $\mathbb{Z}/4\mathbb{Z}$.
\end{lemma}
\begin{proof}
To see that $\mathscr{E}_{\mathbb{Z}/3\mathbb{Z}}$ has diameter 2, observe that there are paths
\[
(1,0)\to(0,1)\to(2,0),\quad (1,0)\to (a,1),\quad\text{and}\quad (1,0)\to(2a+2,1)\to(a,2)
\]
of lengths 1 or 2 from $(1,0)$ to $(2,0)$, $(a,1)$, and $(a,2)$, respectively, for $a\in\mathbb{Z}/3\mathbb{Z}$. By applying suitable elements of $\text{SL}_2(\mathbb{Z}/3\mathbb{Z})$ we can see that there is a path of length 1 or 2 from any vertex of $\mathscr{E}_{\mathbb{Z}/3\mathbb{Z}}$ to another. Hence $\mathscr{E}_{\mathbb{Z}/3\mathbb{Z}}$ has diameter 2. Reasoning similarly we can see that $\mathscr{E}_{\mathbb{Z}/4\mathbb{Z}}$ has diameter 2.

Suppose now that $R$ is a finite local ring other than $\mathbb{Z}/2\mathbb{Z}$, $\mathbb{Z}/3\mathbb{Z}$, or $\mathbb{Z}/4\mathbb{Z}$. By Lemma~\ref{lem2}, $R$ contains a unit $x$ not equal to $\pm 1$. A quick calculation shows that there is no path of length 1 or 2 from $(1,0)$ to $(x,0)$. Hence $\diam\mathscr{E}_R\geq 3$. This confirms that $\mathscr{E}_{\mathbb{Z}/3\mathbb{Z}}$ and $\mathscr{E}_{\mathbb{Z}/4\mathbb{Z}}$ are the only Farey complexes of finite local rings with diameter 2.
\end{proof}

We can now prove Theorem~\ref{thm1}.

\begin{proof}[Proof of Theorem~\ref{thm1}]
First we will prove that $\diam \mathscr{E}_R \leq 3$.

Assume for the moment that $R$ is a (finite) local ring. Let $(a,b)$ be a vertex of $\mathscr{E}_R$; then $ax+by=1$ for some elements $x$ and $y$ of $R$, so one of $ax$ or $by$ must be a unit. Hence one of $a$ or $b$ must be a unit. We will find a path from $(1,0)$ to $(a,b)$ of length 3. Note that we include the possibility that $(a,b)=(1,0)$.

Suppose first that $a$ is a unit, with inverse $q$. Let $r=-q(1+b)$. Then 
\[
(1,0) \to (0,1) \to  (-1,r) \to (a,b)
\]
is such a path. Suppose now that $b$ is a unit, with inverse $q$. Let $r = q(1+a)$. Then 
\[
(1,0) \to (r+1,1) \to (r,1) \to (a,b)
\]
is a suitable path. Since $\text{SL}_2(R)$ acts transitively on the vertices of $\mathscr{E}_R$, we see that there is a path of length 3 from any vertex to another.

Now let $R$ be any finite ring; then $R$ is a direct product of local rings (see, for example, \cite{BiFl2002}*{Theorem~3.1.4}). By taking products of paths, we can obtain a path of length 3 from any vertex to another. Consequently,  $\diam \mathscr{E}_R \leq 3$. 

It remains to identify those Farey complexes $\mathscr{E}_R$ with diameters 1 and 2. By Lemma~\ref{lem1}, the only finite ring $R$ for which $\diam \mathscr{E}_R=1$ is $R=\mathbb{Z}/2\mathbb{Z}$.

Suppose now that $R$ is either equal to  $\mathbb{Z}/3\mathbb{Z}$ or $\mathbb{Z}/4\mathbb{Z}$, or else it is equal to $(\mathbb{Z}/2\mathbb{Z})^n$, for $n>1$. In the former case $\diam\mathscr{E}_R=2$, by Lemma~\ref{lem2}. In the latter case, observe that we can find a path of length exactly 2 from any vertex $u$ of $\mathscr{E}_{\mathbb{Z}/2\mathbb{Z}}$ to any other vertex $v$, even when $v=u$. Hence we can find a path of length 2 from any vertex of $\mathscr{E}_R$ to another, so $\diam \mathscr{E}_R=2$.

Finally, consider any finite ring $R$ for which $\diam \mathscr{E}_R=2$. We can express $R$ as a product of finite local rings. The diameter of the Farey complex of each of these local rings must be at most 2. Hence, by Lemmas~\ref{lem1} and~\ref{lem2}, each local ring must be one of $\mathbb{Z}/2\mathbb{Z}$, $\mathbb{Z}/3\mathbb{Z}$, or $\mathbb{Z}/4\mathbb{Z}$. However, in $\mathscr{E}_{\mathbb{Z}/3\mathbb{Z}}$ and $\mathscr{E}_{\mathbb{Z}/4\mathbb{Z}}$ there is a path of length 3 from any vertex to itself but not a path of length 2, as one can easily verify. Consequently, if there is more than one local ring in the product $R$, and if one of those local rings is $\mathbb{Z}/3\mathbb{Z}$ or $\mathbb{Z}/4\mathbb{Z}$, then we can find two distinct vertices in $\mathscr{E}_R$ not connected by a path of length 1 or 2.  (For example, if $u\in \mathscr{E}_{\mathbb{Z}/3\mathbb{Z}}$ and $v$ and $w$ are distinct vertices in a ring $S$, then there is no path of length 1 or 2 from $(u,v)$ to $(u,w)$ in $\mathscr{E}_R$, where $R=\mathbb{Z}/3\mathbb{Z}\times S$.) In this case, then, $\diam \mathscr{E}_R=3$.

This argument shows that the only finite rings $R$ for which $\diam \mathscr{E}_R=2$ are $\mathbb{Z}/3\mathbb{Z}$, $\mathbb{Z}/4\mathbb{Z}$, and $(\mathbb{Z}/2\mathbb{Z})^n$, for $n>1$.
\end{proof}

Next we state a similar result to Theorem~\ref{thm1} for the Farey complex $\mathscr{F}_R$ instead of $\mathscr{E}_R$.

\begin{theorem}\label{thm7}
The diameter of the Farey complex $\mathscr{F}_R$ of a finite ring $R$ satisfies
\[ 
\diam{\mathscr{F}_R} =
\begin{cases}
 1 & \text{if $R$ is  $\mathbb{Z}/2\mathbb{Z}$ or $\mathbb{Z}/3\mathbb{Z}$,}\\
 2 & \text{if $R$ is a direct product of rings $\mathbb{Z}/2\mathbb{Z}$, $\mathbb{Z}/3\mathbb{Z}$, and $\mathbb{Z}/4\mathbb{Z}$}\\
 & \text{(and $R\neq \mathbb{Z}/2\mathbb{Z},\mathbb{Z}/3\mathbb{Z}$),}\\
 3 & \text{otherwise.}
\end{cases}
\]
\end{theorem}

Here the direct product in the diameter 2 case can involve multiple copies of $\mathbb{Z}/2\mathbb{Z}$, $\mathbb{Z}/3\mathbb{Z}$, and $\mathbb{Z}/4\mathbb{Z}$; for example, the diameter of $\mathscr{F}_R$ is 2, where $R=\mathbb{Z}/2\mathbb{Z}\times\mathbb{Z}/2\mathbb{Z}\times\mathbb{Z}/3\mathbb{Z}$.

We prove Theorem~\ref{thm7} in a similar way to how we proved Theorem~\ref{thm1}, beginning by classifying the Farey complexes $\mathscr{F}_R$ of diameter 1.

\begin{lemma}\label{lem1x}
The Farey complex $\mathscr{F}_R$ of a finite ring $R$ has diameter 1 if and only if $R$ is $\mathbb{Z}/2\mathbb{Z}$ or $\mathbb{Z}/3\mathbb{Z}$.
\end{lemma}
\begin{proof}
Certainly $\mathscr{F}_2$ and $\mathscr{F}_3$ have diameter 1 (see Figure~\ref{fig29}). Suppose now that $R$ is not equal to $\mathbb{Z}/2\mathbb{Z}$ or $\mathbb{Z}/3\mathbb{Z}$. Then $R$ contains an element $x$ other than $0$ or $\pm 1$, in which case the vertices $1/0$ and $1/x$ of $\mathscr{F}_R$ are not adjacent. Hence $\mathscr{F}_R$ does not have diameter 1, as required.
\end{proof}

Next we classify the finite local rings for which $\mathscr{F}_R$ has diameter 2.

\begin{lemma}\label{lem3x}
The Farey complex $\mathscr{F}_R$ of a finite local ring $R$ has diameter 2 if and only if $R$ is~$\mathbb{Z}/4\mathbb{Z}$.
\end{lemma}
\begin{proof}
Certainly $\mathscr{F}_4$ has diameter 2 since it is an octahedron, as shown in  Figure~\ref{fig29}, and $\mathscr{F}_2$ and $\mathscr{F}_3$ have diameter 1.

Suppose now that $R$ is a finite local ring other than $\mathbb{Z}/2\mathbb{Z}$, $\mathbb{Z}/3\mathbb{Z}$, or $\mathbb{Z}/4\mathbb{Z}$. By Lemma~\ref{lem2}, $R$ contains a unit $x$ not equal to $\pm 1$. Then one can check that the graph distance between $1/0$ and $x/0$ is at least 3, so $\diam \mathscr{F}_3\neq 2$, as required.
\end{proof}

We can now prove Theorem~\ref{thm7}.

\begin{proof}[Proof of Theorem~\ref{thm7}]
We know that $\diam \mathscr{E}_R \leq 3$, by Theorem~\ref{thm1}, and since $\pi_{\{\pm 1\}}\colon \mathscr{E}_R\longrightarrow \mathscr{F}_R$ is a covering map we see that $\diam \mathscr{F}_R \leq 3$. By Lemma~\ref{lem1x}, among Farey complexes of finite rings only $\mathscr{F}_2$ and $\mathscr{F}_3$ have diameter 1. 

It remains to determine the Farey complexes $R$ for which $\diam \mathscr{F}_R=2$. To this end, suppose that $R$ is some direct product of copies of $\mathbb{Z}/2\mathbb{Z}$, $\mathbb{Z}/3\mathbb{Z}$, and $\mathbb{Z}/4\mathbb{Z}$ (and $R\neq \mathbb{Z}/2\mathbb{Z},\mathbb{Z}/3\mathbb{Z}$). Between any two (possibly equal) vertices of any one of $\mathscr{F}_2$, $\mathscr{F}_3$, and $\mathscr{F}_4$ we can find a path of length exactly 2. Hence we can find a path of length 2 between any two vertices of $\mathscr{F}_R$, so $\diam \mathscr{F}_R=2$.

Finally, consider any finite ring $R$ for which $\diam \mathscr{F}_R=2$. We can express $R$ as a product of finite local rings. The diameter of the Farey complex of each of these local rings must be at most 2. Hence, by Lemmas~\ref{lem1x} and~\ref{lem3x}, each local ring must be one of $\mathbb{Z}/2\mathbb{Z}$, $\mathbb{Z}/3\mathbb{Z}$, or $\mathbb{Z}/4\mathbb{Z}$, as required.
\end{proof}

\section{Farey complexes and surfaces}\label{section5}

We saw in Figure~\ref{fig29} that the Farey complexes $\mathscr{F}_N$  are surfaces for low values of $N$, and indeed $\mathscr{F}_N$ is a surface for all values of $N$, as we will see shortly. Moreover, we will prove that the Farey complexes over $\mathbb{Z}/N\mathbb{Z}$ are the \emph{only} Farey complexes of finite rings that give rise to surfaces. 

Let us begin by defining a 2-complex and surface complex formally (following \cite{CoGrKuZi1998}*{Chapter~3}, although simplified slightly in our context).

\begin{definition}
A \emph{2-complex} $C$ consists of a collection of vertices $V$, a collection of edges $E$ comprising subsets of $V$ of cardinality 2, and a collection of faces $F$ comprising subsets of $V$ of cardinality 3, where each subset of a face that has cardinality 2 belongs to $E$. 
\end{definition}

This definition does not allow multiple edges between a pair of vertices and nor does it allow edges with only one vertex. Familiar concepts like connectedness can be defined for 2-complexes in the usual way. The Farey complex $\mathscr{F}_{R,U}$ is a finite 2-complex, because $\pm 1\in U$ (so directed edges come in inverse pairs, which can be considered as undirected edges). It is connected, by Theorem~\ref{thm1}. 

\begin{definition}
A \emph{surface complex} $C$ is a connected 2-complex that satisfies the following two properties. 
\begin{enumerate}[label=(\arabic*)]
\item\label{p:e12f} Each edge of $C$ is incident to either one or two faces.
\item\label{p:fseq} For any two edges $u$ and $v$ incident to the same vertex $w$ in $C$ there is a sequence of distinct neighbours $v_1,v_2,\dots,v_n$ of $w$ with $v_1=u$ and $v_n=v$ such that $w$, $v_{i-1}$, and $v_i$ form a face, for $i=2,3,\dots,n$.
\end{enumerate}
\end{definition}

A surface complex can be realised as a topological surface, possibly with boundary. An edge lies on the boundary if it is incident to only one face.

\begin{theorem}\label{thm2}
The Farey complex $\mathscr{F}_{R,U}$ of a finite ring $R$ with units $U$, where $-1\in U$, is a surface complex if and only if $R=\mathbb{Z}/N\mathbb{Z}$, for $N=2,3,\dotsc$. 
\end{theorem}

To prove Theorem~\ref{thm2}, we first discuss the Farey complex $\mathscr{F}_N$ for $\mathbb{Z}/N\mathbb{Z}$; as we shall see, this 2-complex can be realised as a triangulation of a hyperbolic surface with the vertices of $\mathscr{F}_N$ located at the cusps of the surface and the edges of $\mathscr{F}_N$ represented by hyperbolic geodesics between pairs of cusps. We summarise the details; a full explanation can be found in~\cite{IvSi2005}.

The group $\text{SL}_2(\mathbb{Z})$ acts on the extended complex plane by the familiar rule for M\"obius transformations
\[
z\longmapsto \frac{az+b}{cz+d},\quad \text{where } \begin{pmatrix}a&b\\ c& d\end{pmatrix}\in\text{SL}_2(\mathbb{Z}), 
\]
with the usual conventions for the point $\infty$. This group also acts on the extended rationals $\mathbb{Q}_\infty=\mathbb{Q}\cup\{\infty\}$ and on the upper half-plane $\mathbb{H}=\{z:\Ima z>0\}$. The action of  $\text{SL}_2(\mathbb{Z})$ on $\mathbb{H}$ preserves the hyperbolic metric $|dz|/\Ima z$ on $\mathbb{H}$.  For $N=2,3,\dotsc$, the \emph{principal congruence subgroup} of  $\text{SL}_2(\mathbb{Z})$ of level $N$ is the group 
\[
\Gamma_N =  \left\{ \begin{pmatrix}a&b\\ c& d\end{pmatrix}\in\text{SL}_2(\mathbb{Z}) : \begin{pmatrix}a&b\\ c& d\end{pmatrix}\equiv \begin{pmatrix}1&0\\ 0& 1\end{pmatrix}\tpmod{N}\right\}.
\]
 It is a normal subgroup of $\text{SL}_2(\mathbb{Z})$ because it is the kernel of the homomorphism 
\[
\text{SL}_2(\mathbb{Z})\longrightarrow \text{SL}_2(\mathbb{Z}/N\mathbb{Z})
\]
induced by the reduction homomorphism $\mathbb{Z}\longrightarrow\mathbb{Z}/N\mathbb{Z}$ that sends an integer $a$ to the class $\bar{a}$ of integers congruent to  $a$ modulo $N$. The group $\Gamma_N$ is discrete and acts freely on $\mathbb{H}$, so the quotient $S_N=\Gamma_N\backslash \mathbb{H}$ is a hyperbolic surface. Topologically, $S_N$ is a compact surface with finitely many punctures. Geometrically, each puncture is said to be a \emph{cusp} of $S_N$; there is a neighbourhood of each cusp isometric to the hyperbolic punctured unit disc. For $N>2$, we have
\[
\text{genus }= 1+ \frac{N^2(N-6)}{24}\prod_{p|N}\mleft(1-\frac{1}{p^2}\mright),\qquad \text{number of cusps }=\frac{N^2}{2}\prod_{p|N}\mleft(1-\frac{1}{p^2}\mright),
\]
where both products are taken over the prime divisors $p$ of $N$. When $N=2$ the genus is $0$ and there are 3 cusps; see \cite{DiSh2005}*{Section~3.9} for these formulae.

The surface $S_N$ can be compactified by adjoining one additional point to $S_N$ for each cusp. This can be realised by considering the action of $\Gamma_N$ on $\overline{\mathbb{H}}=\mathbb{H}\cup\mathbb{Q}_\infty$. Under this action, the cusps of $S_N$ are identified with the orbits of $\Gamma_N$ on $\mathbb{Q}_\infty$. We denote the resulting compactification of $S_N$ by $\overline{S_N}=\Gamma_N\backslash \overline{\mathbb{H}}$; full details of its construction can be found in \cite{DiSh2005}*{Chapter~2}. Under the projection $\sigma\colon \overline{\mathbb{H}}\longrightarrow \overline{S_N}$ the usual Farey complex $\mathscr{F}_\mathbb{Z}$ on $\overline{\mathbb{H}}$ (see Figure~\ref{fig1}) is mapped to a triangulation of $\overline{S_N}$ with vertices at the cusps of $S_N$ and edges represented by hyperbolic geodesics between certain pairs of cusps. We will show that this triangulation is a realisation of the Farey complex $\mathscr{F}_N$ for $N>2$. The reader can consult \cite{IvSi2005} for a more complete explanation.

We recall that vertices of $\mathscr{F}_N$ have the form $\pm (x,y)$, where $x,y\in\mathbb{Z}/N\mathbb{Z}$. There is a one-to-one correspondence between the cusps $\Gamma_N\backslash \mathbb{Q}_\infty$ and the vertices of $\mathscr{F}_N$ given by 
\[
\Gamma_N(a/b) \longmapsto \pm(\bar{a},\bar{b}).
\]
Two cusps $\Gamma_N(a/b)$ and $\Gamma_N(c/d)$ are adjacent in $\sigma(\mathscr{F}_\mathbb{Z})$ if and only if there exist $A,B\in \Gamma_N$ with 
\[
\begin{pmatrix}a & b\end{pmatrix}A^TJB\binom{c}{d}=\pm 1, \quad\text{where }J=\begin{pmatrix}0 & 1\\ -1 & 0 \end{pmatrix}\!,
\]
and this occurs if and only if $\bar{a}\bar{d}-\bar{b}\bar{c}=\pm 1$. 

For $N>2$, the image of every face of $\mathscr{F}_\mathbb{Z}$ is uniquely specified by the image of its three vertices. The case $N=2$ is special: $\mathscr{F}_2$ consists of a single face, while $\overline{S_2}$ is topologically a sphere arising from two copies of that face glued along the boundary.

This completes our summary of how $\mathscr{F}_N$ arises as a surface complex, indeed one endowed with a hyperbolic structure and with vertices located at the cusps of the surface. The Farey complex $\mathscr{F}_{\mathbb{Z}/N\mathbb{Z},U}$, where $U$ is a subgroup of $(\mathbb{Z}/N\mathbb{Z})^\times$ that contains $\{\pm 1\}$, is the quotient of $\mathscr{F}_N$ by the (free) action of the finite group $U$ on $\mathscr{F}_N$ -- so it is also a surface, possibly nonorientable (for example, $\mathscr{F}_{\mathbb{Z}/5\mathbb{Z},(\mathbb{Z}/5\mathbb{Z})^\times}$ is the projective icosahedron). There are similar geometric constructions of Farey complexes in three dimensions, exemplified by the octahedral tessellation of a hyperbolic 3-manifold in Figure~\ref{fig3}(b). 

We can now complete the proof of Theorem~\ref{thm2}. 

\begin{proof}[Proof of Theorem~\ref{thm2}]
Let $R$ be a finite ring and let $U$ be a subgroup of $R^\times$ that contains $\{\pm 1\}$. We have just seen that $\mathscr{F}_{R,U}$ is a surface complex when $R=\mathbb{Z}/N\mathbb{Z}$. Suppose instead that $R\neq \mathbb{Z}/N\mathbb{Z}$. Then we can find $a\in R - \{0,1,\dots,N-1\}$, where $N=\cha R$. Let $w=1/0$. For each neighbour $x/1$ of $w$, the only neighbours of both $x/1$ and $w$ are $(x\pm 1)/1$. It follows that $u=0/1$ and $v=a/1$ fail Property~\ref{p:fseq} from the definition of a surface complex. Hence $\mathscr{F}_{R,U}$ is not a surface complex, as required.
\end{proof}

\section{Tame $\text{SL}_2$-tilings}\label{section6}

In this section we prove Theorem~\ref{thm3}. Before that we will establish the following theorem, which gives necessary and sufficient conditions for an $\text{SL}_2$-tiling to be tame. Results of this type have been known since $\text{SL}_2$-tilings were first studied -- perhaps the earliest source is \cite[Lemma~2]{BeRe2010}, which is framed in the context of $\text{SL}_k$-tilings over fields. We are not aware of any such result for general rings, so we include a short proof, which circumvents divisibility arguments often present in related work on integers or fields.

\begin{theorem}\label{thm19}
An $\textnormal{SL}_2$-tiling $\mathbf{M}$ over a ring $R$ is tame if and only if there are bi-infinite sequences $(r_i)$ and $(s_j)$ in $R$ such that
\[
m_{i-1,j}+m_{i+1,j}=r_im_{i,j}\quad\text{and}\quad m_{i,j-1}+m_{i,j+1}=s_jm_{i,j},
\]
for $i,j\in\mathbb{Z}$.	Furthermore, if $\mathbf{M}$ is tame, then $(r_i)$ and $(s_j)$ are uniquely determined by $\mathbf{M}$.
\end{theorem}

Key to proving Theorem~\ref{thm19} is the following elementary lemma (not proven here), which is a special case of a procedure known as `Dodgson condensation' for calculating determinants.

\begin{lemma}\label{lemK}
In any ring $R$ we have
\[
e\det\!\begin{pmatrix}a&b&c\\ d&e&f\\ g&h&i\end{pmatrix}=\det\!\begin{pmatrix}a&b\\ d&e\end{pmatrix}\det\!\begin{pmatrix}e&f\\ h&i\end{pmatrix}-\det\!\begin{pmatrix}b&c\\ e&f\end{pmatrix}\det\!\begin{pmatrix}d&e\\ g&h\end{pmatrix}\!.
\]
\end{lemma}

We require another lemma before proving Theorem~\ref{thm19}.

\begin{lemma}\label{lemL}
Consider any 3-by-3 matrix
\[
A=\begin{pmatrix}a&b&c\\ d&e&f\\ g&h&i\end{pmatrix}
\]
over a ring $R$ with determinant 0, and suppose that 
\[
\det\!\begin{pmatrix}a&b\\ d&e\end{pmatrix}=\det\!\begin{pmatrix}b&c\\ e&f\end{pmatrix}=\det\!\begin{pmatrix}d&e\\ g&h\end{pmatrix}=\det\!\begin{pmatrix}e&f\\ h&i\end{pmatrix}=1.
\]
Then
\(
d+f  = \Delta e,
\)
where $\Delta = af-cd=di-fg$.
\end{lemma}
\begin{proof}
Let $\Delta_1 = af-cd$ and $\Delta_2 = di-fg$. By moving the left column of $A$ over to the right and applying Lemma~\ref{lemK} we obtain
\[
f\det A=\Delta_1-\Delta_2.
\]
Since $\det A=0$, we see that $\Delta_1=\Delta_2$. Then
\[
d+f=(bf-ce)d+(ae-bd)f=(af-cd)e=\Delta e,
\]
as required. 
\end{proof}

Now we can prove Theorem~\ref{thm19}.

\begin{proof}[Proof of Theorem~\ref{thm19}]
Suppose first that $m_{i+1,j}+m_{i-1,j}=r_im_{i,j}$ and $m_{i,j+1}+m_{i,j-1}=s_jm_{i,j}$, for $i,j\in\mathbb{Z}$. Then the rows (or columns) of any 3-by-3 submatrix of $\mathbf{M}$ are linearly dependent, so the determinant of that matrix is 0. Hence $\mathbf{M}$ is tame.

Suppose now that $\mathbf{M}$ is tame. Then, by Lemma~\ref{lemL}, for any $i,j\in\mathbb{Z}$ we have
\[
m_{i,j-1}+m_{i,j+1}  = \Delta_{i,j}m_{i,j},
\]
where $\Delta_{i,j} = m_{i-1,j-1}m_{i,j+1}-m_{i-1,j+1}m_{i,j-1}=m_{i,j-1}m_{i+1,j+1}-m_{i,j+1}m_{i+1,j-1}$. This last pair of equations shows that, for any integer $j$, we have $\Delta_{i,j}=\Delta_{k,j}$, for all $i,k\in\mathbb{Z}$. Consequently, there is $s_j\in R$ with 
\[
m_{i,j-1}+m_{i,j+1}=s_jm_{i,j}, \quad\text{for $i\in\mathbb{Z}$,}
\]
as required. A similar argument gives the existence of the required sequence $(r_i)$.

It remains to prove that $(r_i)$ and $(s_j)$ are uniquely determined by $\mathbf{M}$. Again, we prove this for $(s_j)$; the proof for $(r_i)$ is similar. Suppose then that there is another sequence $(s_j')$ in $R$ with $m_{i,j-1}+m_{i,j+1}=s_j'm_{i,j}$, for $i,j\in\mathbb{Z}$. Choose any integer $j$. Then $(s_j-s_j')m_{i,j}=0$, for $i\in\mathbb{Z}$. Hence
\[
s_j-s_j'=(s_j-s_j')(m_{0,j}m_{1,j+1}-m_{0,j+1}m_{1,j})=0,
\]
so $s_j=s_j'$, as required.
\end{proof}

Next we set about proving Theorem~\ref{thm3}. Recall (from the introduction) that $\mathscr{P}$ denotes the collection of bi-infinite paths in $\mathscr{E}_R$ and $\SL$ denotes the collection of tame $\text{SL}_2$-tilings over $R$. The function $\widetilde{\Phi}\colon \mathscr{P}\times\mathscr{P}\longrightarrow \SL$ sends the pair of paths $\gamma$ and $\delta$, with vertices $(a_i,b_i)$ and $(c_j,d_j)$, to the $\text{SL}_2$-tiling $\mathbf{M}$ with entries 
\[
m_{i,j}=\begin{pmatrix}a_i & b_i\end{pmatrix}J\begin{pmatrix}c_j \\ d_j\end{pmatrix}=a_id_j-b_ic_j,
\quad
\text{where } 
J=\begin{pmatrix}0 & 1\\ -1 & 0 \end{pmatrix}\!.
\]
We can see that $\mathbf{M}$ is indeed an $\text{SL}_2$-tiling by observing that
\begin{equation}\label{eqn2}
\begin{pmatrix}m_{i,j} & m_{i,j+1}\\ m_{i+1,j} & m_{i+1,j+1} \end{pmatrix}=
\begin{pmatrix}a_i & b_i \\a_{i+1} & b_{i+1}\end{pmatrix}J
\begin{pmatrix}c_{j} & c_{j+1}\\ d_j & d_{j+1}\end{pmatrix}
\in \text{SL}_2(R).
\end{equation}
We must also show that $\mathbf{M}$ is tame. By Lemma~\ref{lem85}, the itinerary $(e_i)$ of $\gamma$ satisfies $a_{i-1}+a_{i+1}=e_ia_i$ and $b_{i-1}+b_{i+1}=e_ib_i$, for $i\in\mathbb{Z}$. It follows that any three consecutive rows of $\mathbf{M}$ are linearly dependent, so the determinant of any 3-by-3 submatrix of $\mathbf{M}$ is 0, as required. 

\begin{lemma}\label{lem78}
The map $\widetilde{\Phi}$ satisfies $\widetilde{\Phi}(A{\gamma},A{\delta})=\widetilde{\Phi}(\gamma,{\delta})$, for $A\in\textnormal{SL}_2(R)$.
\end{lemma}
\begin{proof}
Let
\[
\begin{pmatrix}a_i' \\ b_i'\end{pmatrix}=A\begin{pmatrix}a_i \\ b_i\end{pmatrix}\!,
\quad
\begin{pmatrix}c_j' \\ d_j'\end{pmatrix}=A\begin{pmatrix}c_j \\ d_j\end{pmatrix}\!,
\quad\text{and}\quad
m'_{i,j}=a_i'd_j'-b_i'c_j',
\]
for $i,j\in\mathbb{Z}$. Then we can use the formula $A^TJA=J$ to give
\[
m_{i,j}' = \begin{pmatrix}a_i' & b_i'\end{pmatrix}J\begin{pmatrix}c_j' \\ d_j'\end{pmatrix}
=\begin{pmatrix}a_i & b_i\end{pmatrix}A^TJA\begin{pmatrix}c_j \\ d_j\end{pmatrix}=m_{i,j},
\]
as required.
\end{proof}

Lemma~\ref{lem78} shows that the map $\widetilde{\Phi}$ induces a function $\text{SL}_2(R)\backslash (\mathscr{P}\times\mathscr{P})\longrightarrow \SL$. The next lemma demonstrates that this function is bijective.

\begin{lemma}\label{lem45}
The function $\widetilde{\Phi}\colon \mathscr{P}\times\mathscr{P}\longrightarrow \SL$ is surjective. Furthermore, if $\widetilde{\Phi}(\gamma,\delta)=\widetilde{\Phi}(\gamma',\delta')$, then $(\gamma',\delta')=(A\gamma,A\delta)$, for some $A\in\textnormal{SL}_2(R)$. 
\end{lemma}
\begin{proof}
First we prove that $\widetilde{\Phi}$ is surjective. Let $\mathbf{M}$ be a tame $\text{SL}_2$-tiling over $R$ and let $(r_i)$ and $(s_j)$ be the sequences in $R$ associated to $\mathbf{M}$ specified by Theorem~\ref{thm19}. We define $\gamma$ to be the path in $\mathscr{P}$ with itinerary $(r_i)$ and vertices $(a_i,b_i)$ that satisfy $(a_0,b_0)=(m_{0,0},m_{0,1})$ and $(a_1,b_1)=(m_{1,0},m_{1,1})$, and we define $\delta$ to be the path with itinerary $(s_j)$ and vertices $(c_j,d_j)$ that satisfy $(c_0,d_0)=(0,1)$ and $(c_1,d_1)=(-1,0)$. Now let $\mathbf{M}'$ be the tame $\text{SL}_2$-tiling $\widetilde{\Phi}(\gamma,\delta)$, with entries $m_{i,j}'=a_id_j-b_ic_j$. Then $m_{i,j}'=m_{i,j}$, for $i,j=0,1$. What is more, applying Lemma~\ref{lem85}, we have
\[
m_{i-1,j}'+m_{i+1,j}'= (a_{i-1}+a_{i+1})d_j-(b_{i-1}+b_{i+1})c_j=r_ia_id_j-r_ib_ic_j=r_im_{i,j}',
\]
and similarly $m_{i,j-1}'+m_{i,j+1}'=s_jm_{i,j}'$. These recurrence relations specify $\mathbf{M}'$ uniquely from $m_{0,0}, m_{1,0},m_{0,1},m_{1,1}$, so $\mathbf{M}'=\mathbf{M}$.

For the second part, suppose that $\widetilde{\Phi}(\gamma,\delta)=\widetilde{\Phi}(\gamma',\delta')$, and let $\gamma,\delta,\gamma',\delta'\in\mathscr{P}$ have vertices $(a_i,b_i)$, $(c_j,d_j)$, $(a_i',b_i')$, and $(c_j',d_j')$. By replacing $(\gamma,\delta)$ with $(B\gamma,B\delta)$ and replacing $(\gamma',\delta')$ with $(B'\gamma',B'\delta')$, for suitable matrices $B,B'\in\text{SL}_2(R)$, and using Lemma~\ref{lem78}, we can assume that $(c_0,d_0)=(c_0',d_0')=(0,1)$ and $(c_1,d_1)=(c_1',d_1')=(-1,0)$. Let $\mathbf{M}=\widetilde{\Phi}(\gamma,\delta)$. Then, as before, we can see that the itinerary $(e_i)$ of $\gamma$ satisfies $m_{i-1,j}+m_{i+1,j}=e_im_{i,j}$, for $i\in\mathbb{Z}$, and similarly for the itinerary of $\gamma'$. From the uniqueness part of Theorem~\ref{thm19} it follows that the itineraries of $\gamma$ and $\gamma'$ are equal, and likewise so are those of $\delta$ and $\delta'$. Hence $\delta$ and $\delta'$ are equal, by Lemma~\ref{lem17}. Also, because $(c_0,d_0)=(c_0',d_0')=(0,1)$, we have
\[
a_0=a_0d_0-b_0c_0=a_0'd_0'-b_0'c_0'=a_0',
\]
and similarly $b_0=b_0'$, $a_1=a_1'$, and $b_1=b_1'$. Therefore $\gamma$ and $\gamma'$ are equal also, as required.
\end{proof}

To facilitate the formal definition of the map $\Phi_U$ from Theorem~\ref{thm3}, we introduce an intermediate function 
$\widetilde{\Phi}_U\colon  \mathscr{P}_U\times\mathscr{P}_U\longrightarrow(U\times U)\backslash \SL$. For this we define $\tau_U\colon \SL\longrightarrow (U\times U)\backslash \SL$ to be the map that takes a tame $\text{SL}_2$-tiling $\mathbf{M}$ to its orbit under the action of $U\times U$ on $\SL$ (given by $m_{i,j}\longmapsto \lambda^{(-1)^i}\mu^{(-1)^j}m_{i,j}$, for $(\lambda,\mu)\in U\times U$). We recall that $\pi_U\colon \mathscr{E}_R\longrightarrow \mathscr{F}_{R,U}$ is the usual covering map. Then we define
\[
\widetilde{\Phi}_U(\pi_U(\gamma),\pi_U(\delta))=\tau_U\widetilde{\Phi}(\gamma,\delta),\quad \text{for }(\gamma,\delta)\in\mathscr{P}\times \mathscr{P}.
\]
To see that $\widetilde{\Phi}_U$ is well-defined, first observe that, because $\pi_U$ is a covering map, any pair in $\mathscr{P}_U\times\mathscr{P}_U$ lifts to a pair in $\mathscr{P}\times\mathscr{P}$. Next, suppose that $(\pi_U(\gamma),\pi_U(\delta))=(\pi_U(\gamma'),\pi_U(\delta'))$. Let $\gamma$ have vertices $(a_i,b_i)$ and $\delta$ have vertices $(c_j,d_j)$; then $\widetilde{\Phi}(\gamma,\delta)$ has entries $m_{i,j}=a_id_j-b_ic_j$. By Lemma~\ref{lem61}, $\gamma'$ and $\delta'$ have vertices $\lambda^{(-1)^{i}}(a_i,b_i)$ and $\mu^{(-1)^{j}}(c_j,d_j)$, respectively, for some $\lambda,\mu\in U$, in which case $\widetilde{\Phi}(\gamma',\delta')$ has entries
\[
m_{i,j}'=\lambda^{(-1)^{i}}\mu^{(-1)^{j}}(a_id_j-b_ic_j)=\lambda^{(-1)^{i}}\mu^{(-1)^{j}}m_{i,j}.
\]
Hence $\tau_U\widetilde{\Phi}(\gamma,\delta)=\tau_U\widetilde{\Phi}(\gamma',\delta')$, as required.

The next lemma proves an $\text{SL}_2(R)$-invariance property of $\widetilde{\Phi}_U$ that has already been established for $\widetilde{\Phi}$ in Lemma~\ref{lem78}.

\begin{lemma}\label{lem79}
The map $\widetilde{\Phi}_U$ satisfies $\widetilde{\Phi}_U(A{\gamma},A{\delta})=\widetilde{\Phi}_U(\gamma,{\delta})$, for $A\in\textnormal{SL}_2(R)$.
\end{lemma}
\begin{proof}
We choose $(\tilde{\gamma},\tilde{\delta})\in\mathscr{P}\times\mathscr{P}$ with $\pi_U(\tilde{\gamma})=\gamma$ and $\pi_U(\tilde{\delta})=\delta$.
Recall from Proposition~\ref{prop3} that $\pi_U A=A\pi_U$ (using the current terminology). Then $\widetilde{\Phi}_U(A\pi_U(\tilde{\gamma}),A\pi_U(\tilde{\delta}))=\widetilde{\Phi}_U(\pi_U(A\tilde{\gamma}),\pi_U(A\tilde{\delta}))$, so
\[
\widetilde{\Phi}_U(A\pi_U(\tilde{\gamma}),A\pi_U(\tilde{\delta}))=\tau_U\widetilde{\Phi}(A\tilde{\gamma},A\tilde{\delta})=\tau_U\widetilde{\Phi}(\tilde{\gamma},\tilde{\delta})=\widetilde{\Phi}_U(\pi_U(\tilde{\gamma}),\pi_U(\tilde{\gamma})),
\]
where we have applied Lemma~\ref{lem78}.
\end{proof}

From Lemma~\ref{lem79} we see that $\widetilde{\Phi}_U$ induces a function 
\[
\Phi_U\colon \text{SL}_2(R)\backslash(\mathscr{P}_U\times\mathscr{P}_U)\longrightarrow (U\times U)\backslash \SL,
\]
which maps the $\text{SL}_2(R)$-orbit of a pair $(\gamma,\delta)$ in $\mathscr{P}_U\times\mathscr{P}_U$ to $\widetilde{\Phi}_U(\gamma,\delta)$. This map is surjective, since $\widetilde{\Phi}$ is surjective. To complete the proof of Theorem~\ref{thm3} we have only to show that $\Phi_U$ is injective.

Suppose then that $\widetilde{\Phi}_U(\pi_U(\gamma),\pi_U(\delta))=\widetilde{\Phi}_U(\pi_U(\gamma'),\pi_U(\delta'))$, for paths $\gamma,\delta,\gamma',\delta'\in\mathscr{P}$ with vertices $(a_i,b_i)$, $(c_j,d_j)$, $(a_i',b_i')$, and $(c_j',d_j')$. The $\text{SL}_2$-tilings $\mathbf{M}$ and $\mathbf{M}'$ with entries $m_{i,j}=a_id_j-b_ic_j$ and $m_{i,j}'=a_i'd_j'-b_i'c_j'$ lie in the same orbit of $U\times U$, so there exist $\lambda,\mu\in U$ with $m_{i,j}=\lambda^{(-1)^i}\mu^{(-1)^j}m_{i,j}'$, for $i,j\in\mathbb{Z}$. Let  $\gamma'',\delta''\in\mathscr{P}$ have vertices $\lambda^{(-1)^i}(a_i',b_i')$ and $\mu^{(-1)^j}(c_j',d_j')$, respectively. Then $(\pi_U(\gamma''),\pi_U(\delta''))=(\pi_U(\gamma'),\pi_U(\delta'))$, by Lemma~\ref{lem61}, and the $\text{SL}_2$-tiling $\mathbf{M}''$ with entries $m_{i,j}''=\lambda^{(-1)^i}a_i'\mu^{(-1)^j}d_j'-\lambda^{(-1)^i}b_i'\mu^{(-1)^j}c_j'=m_{i,j}$ coincides with $\mathbf{M}$. We can then apply Lemma~\ref{lem45} to see that there exists $A\in\text{SL}_2(R)$ with $(\gamma'',\delta'')=(A\gamma,A\delta)$. Consequently, $(\pi_U(\gamma''),\pi_U(\delta''))=(A\pi_U(\gamma),A\pi_U(\delta))$, so $\Phi_U$ is indeed injective. This completes the proof of Theorem~\ref{thm3}.

\section{Tame infinite friezes}\label{section7}

As a stepping stone from $\text{SL}_2$-tilings to friezes we look at infinite friezes, which have been considered by Baur, Parsons, and Tschabold \cite{BaPaTs2016} among others.

\begin{definition}
An infinite frieze over a ring $R$ is a function $\mathbf{F}\colon \{(i,j)\in\mathbb{Z}^2:i\geq j\}\longrightarrow R$ with entries  $m_{i,j}=\mathbf{F}(i,j)$ such that
\begin{itemize}
\item $m_{i,i}=0$, for $i\in\mathbb{Z}$ (top row of zeros),
\item $m_{i,j}m_{i+1,j+1}-m_{i,j+1}m_{i+1,j}=1$, for $i>j$ (diamond rule).
\end{itemize}
\end{definition}

The infinite frieze $\mathbf{F}$ is \emph{semiregular} if $m_{i+1,i}=1$, for $i\in\mathbb{Z}$ (second row of 1's), and \emph{tame} if the usual 3-by-3 determinant 0 condition holds for $\mathbf{F}$, wherever it makes sense. The main objective of this section is to prove the following theorem (featuring a map $\Psi$ defined formally later).

\begin{theorem}\label{thm51}
The map $\Psi$ induces a one-to-one correspondence between
\[
\textnormal{SL}_2(R)\Big\backslash\mleft\{\parbox{1.9cm}{\centering\textnormal{bi-infinite paths~in~$\mathscr{E}_R$}}\mright\}\quad \longleftrightarrow\quad \mleft\{\parbox{3.9cm}{\centering\textnormal{tame~semiregular~infinite friezes~over~$R$}}\mright\}.
\]
\end{theorem}

The second row of an infinite frieze can be calculated from any single entry using the diamond rule. This observation is encapsulated in the next lemma.

\begin{lemma}\label{lem95}
Let $\mathbf{F}$ be a frieze or infinite frieze. Then $m_{i+1,i}=\alpha^{(-1)^i}$, for $i\in\mathbb{Z}$, where $\alpha=m_{1,0}$.
\end{lemma}
\begin{proof}
Since $\mathbf{F}$ is a frieze or infinite frieze we have $m_{i,i-1}m_{i+1,i}-m_{i,i}m_{i+1,i-1}=1$, for each $i\in\mathbb{Z}$. But $m_{i,i}=0$, so $m_{i,i-1}m_{i+1,i}=1$. The equation $m_{i+1,i}=\alpha^{(-1)^i}$ then follows by induction.
\end{proof}

Any tame infinite frieze can be extended to a tame $\text{SL}_2$-tiling in the following manner.

\begin{lemma}\label{lem91}
Let $\mathbf{F}$ be a tame infinite frieze and let $\mathbf{M}\colon \mathbb{Z}\times\mathbb{Z}\longrightarrow R$ (with entries $m_{i,j}$) satisfy $m_{i,j}=\mathbf{F}(i,j)$, for $i\geq j$, and $m_{j,i}=-\alpha^{(-1)^i-(-1)^j}m_{i,j}$, for $i,j\in\mathbb{Z}$, where $\alpha=m_{1,0}$. Then $\mathbf{M}$ is a tame $\textnormal{SL}_2$-tiling.
\end{lemma}
\begin{proof}
Let us first check that the definition of $\mathbf{M}$ does indeed specify a function from $\mathbb{Z}\times\mathbb{Z}$ to $R$. To see this, observe that $\mathbf{M}$ is uniquely defined by the requirements that  $m_{i,j}=\mathbf{F}(i,j)$, for $i\geq j$, and $m_{j,i}=-\alpha^{(-1)^i-(-1)^j}m_{i,j}$, for $i<j$. It is then straightforward to check that in fact the formula $m_{j,i}=-\alpha^{(-1)^i-(-1)^j}m_{i,j}$ holds for \emph{all} indices $i,j\in\mathbb{Z}$, as specified in the original definition of $\mathbf{M}$.

Next we check that the diamond rule $m_{i,j}m_{i+1,j+1}-m_{i,j+1}m_{i+1,j}=1$ is satisfied for all $i,j\in\mathbb{Z}$. Certainly it is satisfied when $i>j$, by definition of a tame infinite frieze. If $i<j$, then one can check that the rule is satisfied by applying $m_{j,i}=-\alpha^{(-1)^i-(-1)^j}m_{i,j}$. Last, for $i=j$, we observe that $m_{i,i+1}=-\alpha^{(-1)^{i+1}-(-1)^i}m_{i+1,i}=-\alpha^{-2(-1)^i}m_{i+1,i}$ and $m_{i+1,i}=\alpha^{(-1)^i}$, so
\[
m_{i,i}m_{i+1,i+1}-m_{i,i+1}m_{i+1,i}=\alpha^{-2(-1)^i}m_{i+1,i}^2=1.
\]
It remains to check that $\mathbf{M}$ is tame; that is, we must check that $\det A=0$, where
\[
A= \begin{pmatrix}m_{i-1,j-1} & m_{i-1,j} & m_{i-1,j+1}\\ m_{i,j-1} & m_{i,j} & m_{i,j+1}\\  m_{i+1,j-1} & m_{i+1,j} & m_{i+1,j+1}\end{pmatrix}\!,
\]
for each $i,j\in\mathbb{Z}$. One can check that this is true if $|i-j|>1$ by using the tame property of $\mathbf{F}$ and the relation $m_{j,i}=-\alpha^{(-1)^i-(-1)^j}m_{i,j}$. Suppose that $i=j$. Then each entry on the leading diagonal is 0, and a short calculation shows that $\det A=m_{j+1,j-1}+m_{j-1,j+1}=0$. Suppose now that $i=j+1$. Then $m_{i,j}$ is a unit, by Lemma~\ref{lem95}, so $\det A=0$, by Lemma~\ref{lemK}, using the diamond rule. The case $i=j-1$ can be handled similarly. Hence $\mathbf{M}$ is a tame $\text{SL}_2$-tiling, as required.
\end{proof}

We refer to the $\text{SL}_2$-tiling $\mathbf{M}$ obtained from the tame infinite frieze $\mathbf{F}$ as the \emph{extension} of $\mathbf{F}$ to a tame $\text{SL}_2$-tiling.

Theorem~\ref{thm3} taught us that pairs of bi-infinite paths can be used to classify tame $\text{SL}_2$-tilings; here we will see that single bi-infinite paths can be used to classify infinite friezes. This observation was made by the first author for the usual Farey complex over the integers in \cite{Sh2023}. We extend this to all rings and groups of units. Let us start with an elementary lemma.

\begin{lemma}\label{lem66}
Let $(a,b)$ and $(c,d)$ be vertices of $\mathscr{E}_R$ with $ad-bc=0$. Then $(c,d)=\lambda(a,b)$, for some $\lambda\in R^\times$.
\end{lemma}
\begin{proof}
Since $(a,b)\in\mathscr{E}_R$, we have $pa+qb=1$, for some $p,q\in R$. Let $\lambda=pc+qd$. Then 
\[
\lambda a = (pc+qd)a=pca+qbc=(pa+qb)c=c,
\]
and similarly $\lambda b=d$, as required.
\end{proof}

The next lemma is central to our classification of tame infinite friezes using bi-infinite paths.

\begin{lemma}\label{lem25}
Let $\mathbf{F}$ be a tame infinite frieze, and let $\mathbf{M}$ be the extension of $\mathbf{F}$ to a tame $\textnormal{SL}_2$-tiling. Then there exists a bi-infinite path $\gamma$ in $\mathscr{E}_R$ with vertices $(a_i,b_i)$ such that 
\[
m_{i,j}=\alpha^{(-1)^j}(a_jb_i-b_ja_i),\quad \text{for }i,j\in\mathbb{Z},
\]
where $\alpha=m_{1,0}$. 
\end{lemma}
\begin{proof}
By Lemma~\ref{lem45}, there are paths $\gamma,\delta\in \mathscr{P}$ with vertices $(a_i,b_i)$ and $(c_j,d_j)$ such that $m_{i,j}=a_id_j-b_ic_j$. Since $m_{i,i}=0$, we have that $a_id_i-b_ic_i=0$, for each $i\in\mathbb{Z}$. Lemma~\ref{lem66} tells us that there is $\lambda_j\in R$ with $(c_j,d_j)=\lambda_j (a_j,b_j)$. From Lemma~\ref{lem95} we have $m_{j+1,j}=\alpha^{(-1)^j}$, so 
\[
\lambda_j=-\lambda_j(a_{j+1}b_j-b_{j+1}a_j)=-(a_{j+1}d_j-b_{j+1}c_j)=-m_{j+1,j}=-\alpha^{(-1)^j}.
\]
Hence $m_{i,j}=\alpha^{(-1)^j}(a_jb_j-b_ja_i)$, as required.
\end{proof}

The next theorem shows that we can construct a tame infinite frieze uniquely from any suitable second row and any choice of third row whatsoever.

\begin{theorem}\label{thm84}
Given any bi-infinite sequence $(r_i)$ in $R$ and $\alpha\in R^\times$ there is a unique tame infinite frieze with $m_{i+1,i}=\alpha^{(-1)^i}$ and $m_{i+1,i-1}=r_i$, for $i\in\mathbb{Z}$.
\end{theorem}
\begin{proof}
First we prove the existence of such an infinite frieze. By Lemma~\ref{lem17}, there is a bi-infinite path $\gamma$ in $\mathscr{E}_R$ with itinerary $(\alpha^{(-1)^i}r_i)$; let $(a_i,b_i)$ be the vertices of this path and let $\delta$ be the path with vertices $-\alpha^{(-1)^j}(a_j,b_j)$. Then $\mathbf{M}=\widetilde{\Phi}(\gamma,\delta)$ is a tame $\text{SL}_2$-tiling with entries $m_{i,j}=\alpha^{(-1)^j}(a_jb_i-b_ja_i)$. Hence $m_{i,i}=0$, $m_{i+1,i}=\alpha^{(-1)^i}$, and
\[
m_{i+1,i-1}=\alpha^{-(-1)^i}(a_{i-1}b_{i+1}-b_{i-1}a_{i+1})=r_i,
\]
by definition of the itinerary of $\gamma$. Thus the restriction of $\mathbf{M}$ to $\{(i,j)\in\mathbb{Z}^2:i\geq j\}$ is a tame infinite frieze with the required properties.

It remains to prove the uniqueness part of the theorem. Suppose then that $\mathbf{F}$ is a tame  infinite frieze with $m_{i+1,i}=\alpha^{(-1)^i}$ and $m_{i+1,i-1}=r_i$, for $i\in\mathbb{Z}$, and let $\mathbf{M}$ be the extension of $\mathbf{F}$ to a tame $\text{SL}_2$-tiling. From Lemma~\ref{lem25} we can find a bi-infinite path $\gamma$ in $\mathscr{E}_R$ with vertices $(a_i,b_i)$ such that $m_{i,j}=\alpha^{(-1)^j}(a_jb_i-b_ja_i)$. Let $(e_i)$ be the itinerary of $\gamma$. Then
\[
e_i=a_{i-1}b_{i+1}-b_{i-1}a_{i+1}=\alpha^{(-1)^i}m_{i+1,i-1}=\alpha^{(-1)^i}r_i.
\]
We can now apply Lemma~\ref{lem17} to see that $\gamma$ is uniquely specified by its itinerary (itself determined by $(r_i)$ and $\alpha$) and an initial directed edge. However, up to $\text{SL}_2(R)$ equivalence we can choose any edge in $\mathscr{E}_R$ as the initial directed edge, and the formula $m_{i,j}=\alpha^{(-1)^j}(a_jb_i-b_ja_i)$ is preserved under the action of $\text{SL}_2(R)$. It follows that $\mathbf{M}$ is uniquely specified by  $(r_i)$ and $\alpha$.
\end{proof}

It is worthwhile highlighting the special case of Theorem~\ref{thm84} for semiregular infinite friezes (the $\alpha=1$ case). This result is known already (see \cite{Mo2015}*{Theorem 1.15}), and indeed one can easily obtain Theorem~\ref{thm84} from this special case.

\begin{corollary}\label{cor84}
Given any bi-infinite sequence $(r_i)$ in $R$ there is a unique tame semiregular infinite frieze with $m_{i+1,i-1}=r_i$, for $i\in\mathbb{Z}$.
\end{corollary}

In the introduction we defined the map $\widetilde{\Psi}\colon\mathscr{P}\longrightarrow\SL$ by the rule $\widetilde{\Psi}(\gamma)=\widetilde{\Phi}(\gamma,-\gamma)$. This map sends the bi-infinite path $\gamma$ in $\mathscr{E}_R$ with vertices $(a_i,b_i)$ to the tame $\text{SL}_2$-tiling $\mathbf{M}$ over $R$ with entries $m_{i,j}=a_jb_i-b_ja_i$. Observe that $m_{i,i}=0$ and $m_{i+1,i}=1$, for $i\in\mathbb{Z}$, so $\mathbf{M}$ is in fact (the extension of) a tame semiregular infinite frieze. Since $\widetilde{\Phi}(A\gamma,-A\gamma)=\widetilde{\Phi}(\gamma,-\gamma)$, for $A\in\text{SL}_2(R)$, it follows that $\widetilde{\Psi}(A\gamma)=\widetilde{\Psi}(\gamma)$, so we obtain an induced map $\Psi$ from $\text{SL}_2(R)\backslash \mathscr{P}$ to the set of tame semiregular infinite friezes. Theorem~\ref{thm51}, stated at the start of this section, says that this map $\Psi$ is bijective.

\begin{proof}[Proof of Theorem~\ref{thm51}] 
We recall from Theorem~\ref{thm3}  that the map $\Phi\colon \text{SL}_2(R)\backslash(\mathscr{P}\times\mathscr{P})\longrightarrow\SL$ (where $\Phi=\Phi_{\{1\}}$) is a one-to-one correspondence. A straightforward consequence is that $\Psi$ is injective. All that remains is to prove that $\Psi$ is surjective; however, this has been done already, in Lemma~\ref{lem25} for the case $\alpha=1$.
\end{proof}

A version of this result for integer friezes was proven by the first author in \cite{Sh2023}*{Theorem~1.3}.

\section{Tame friezes}\label{section8}

In this section we prove Theorem~\ref{thm4}. First, though, we prove (as promised in the introduction) that a tame frieze $\mathbf{F}$ of width $n$ has a unique extension to a tame $\text{SL}_2$-tiling. We require a preparatory lemma, which is similar to Lemma~\ref{lem95}.

\begin{lemma}\label{lem96}
Let $\mathbf{F}$ be a frieze of width $n$. Then $m_{i+n-1,i}=\beta^{(-1)^i}$, for $i\in\mathbb{Z}$, where $\beta=m_{n-1,0}$.
\end{lemma}
\begin{proof}
Since $\mathbf{F}$ is a frieze we have $m_{i+n-1,i}m_{i+n,i+1}-m_{i+n-1,i+1}m_{i+n,i}=1$, for each $i\in\mathbb{Z}$. But $m_{i+n,i}=0$, so $m_{i+n-1,i}m_{i+n,i+1}=1$, and the result follows by induction.
\end{proof}

Next we can state the extension result.

\begin{proposition}\label{lem93}
Let $\mathbf{F}$ be a tame frieze of width $n$ (with entries $m_{i,j}$) and let $\mathbf{M}\colon \mathbb{Z}\times\mathbb{Z}\longrightarrow R$ be the function determined by the recurrence relations $m_{i+n,j}=-\lambda^{(-1)^i}m_{i,j}$, for $i,j\in\mathbb{Z}$, where $\lambda=m_{1,0}/m_{n-1,0}$. Then $\mathbf{M}$ is a tame $\textnormal{SL}_2$-tiling. Furthermore, $\mathbf{M}$ is the unique tame $\textnormal{SL}_2$-tiling that coincides with $\mathbf{F}$ on $0\leq i-j\leq n$.
\end{proposition}
\begin{proof}
We define $\alpha=m_{1,0}$ and $\beta=m_{n-1,0}$, so $\lambda=\alpha/\beta$. The recurrence relations $m_{i+n,j}=-\lambda^{(-1)^i}m_{i,j}$ specify a well-defined function $\mathbf{M}$ because $m_{i,i}=m_{i+n,i}=0$, for $i\in\mathbb{Z}$. 

Let us establish the diamond rule $m_{i,j}m_{i+1,j+1}-m_{i,j+1}m_{i+1,j}=1$, for $i,j\in\mathbb{Z}$. We know that this holds for $0<i-j<n$ because $\mathbf{F}$ is a tame frieze of width $n$. For $i=j$ we have $m_{i,i}=0$, $m_{i+1,i}=\alpha^{(-1)^{i}}$ by Lemma~\ref{lem95}, and 
\[
m_{i,i+1}=-\lambda^{-(-1)^{i}}m_{i+n,i+1}  =-\lambda^{-(-1)^{i}}\beta^{(-1)^{i+1}}=-\alpha^{-(-1)^i},
\]
where we have applied Lemma~\ref{lem96} to obtain $m_{i+n,i+1}=\beta^{(-1)^{i+1}}$. With these values we see that the diamond rule is indeed satisfied with $i=j$. We have now established the diamond rule for $0\leq i-j< n$, and the recurrence relations $m_{i+n,j}=-\lambda^{(-1)^i}m_{i,j}$ can then be applied to show that the rule is satisfied for all $i,j\in\mathbb{Z}$; we omit the details.

We must check that $\mathbf{M}$ is tame; that is, we must check that $\det A=0$, where
\[
A= \begin{pmatrix}m_{i-1,j-1} & m_{i-1,j} & m_{i-1,j+1}\\ m_{i,j-1} & m_{i,j} & m_{i,j+1}\\  m_{i+1,j-1} & m_{i+1,j} & m_{i+1,j+1}\end{pmatrix}\!,
\]
for each $i,j\in\mathbb{Z}$. This holds when $1<i-j<n-1$ because $\mathbf{F}$ is tame. It also holds whenever $m_{i,j}$ is a unit, by Lemma~\ref{lemK} (and the diamond rule), so in particular it holds for $i=j+1$ and $i=j+n-1$. When $i=j$ each entry on the leading diagonal is 0, and a short calculation shows that $\det A=m_{j+1,j-1}+m_{j-1,j+1}=0$. Therefore $\det A=0$ for $0\leq i-j<n$, and the recurrence relations $m_{i+n,j}=-\lambda^{(-1)^i}m_{i,j}$ can then be applied to show that $\det A=0$ for all $i,j\in\mathbb{Z}$; again we omit the details.

It remains to prove the uniqueness part of the proposition. Suppose then that $\mathbf{M}$ is any tame $\text{SL}_2$-tiling that coincides with $\mathbf{F}$ on $0\leq i-j\leq n$. Then $\mathbf{M}$ is a tame infinite frieze (since $m_{i,i}=0$, for $i\in\mathbb{Z}$). We can then invoke Theorem~\ref{thm84} to see that $\mathbf{M}$ is uniquely specified by the second and third rows of $\mathbf{F}$, as required.
\end{proof}

There is an embedding of the collection $\FR$ of tame friezes of width $n$ into $\SL$ that identifies a tame frieze $\mathbf{F}$ with the unique tame $\text{SL}_2$-tiling $\mathbf{M}$ determined by Proposition~\ref{lem93}. For convenience, we also write $\FR$ for the image under this embedding; it comprises those tame $\text{SL}_2$-tilings $\mathbf{M}$ that satisfy $m_{i,i}=0$ and
\[
m_{i+n,j}=-\lambda^{(-1)^i}m_{i,j},\quad \text{for }i,j\in\mathbb{Z},
\]
where $\lambda\in R^\times$ (and in fact $\lambda=m_{1,0}/m_{n-1,0}$, by the uniqueness part of Proposition~\ref{lem93}).

Let us now set about proving Theorem~\ref{thm4}. We recall the map $\widetilde{\Psi}\colon\mathscr{P}\longrightarrow\SL$ given by the rule $\widetilde{\Psi}(\gamma)=\widetilde{\Phi}(\gamma,-\gamma)$, which sends the bi-infinite path $\gamma$ in $\mathscr{E}_R$ with vertices $v_i=(a_i,b_i)$ to the tame semiregular infinite frieze $\mathbf{M}$ with entries $m_{i,j}=a_jb_i-b_ja_i$. Consider the restriction of $\widetilde{\Psi}$ to $\mathscr{C}_n$; this set comprises bi-infinite paths $\gamma$ with $v_{i+n}=\lambda^{(-1)^i} v_i$, for $i\in\mathbb{Z}$, where $\lambda\in R^\times$. Observe that
\[
m_{i+n,j} = a_{j}b_{i+n}-b_{j}a_{i+n}=\lambda^{(-1)^i}(a_jb_i-b_ja_i)=\lambda^{(-1)^i}m_{i,j}.
\]
Therefore $\widetilde{\Psi}(\gamma)\in \FRS$ (the collection of tame semiregular friezes of width $n$). Consequently, we have a map $\widetilde{\Psi}\colon\mathscr{C}_n\longrightarrow\FRS$.

\begin{lemma}\label{lem51}
The map $\widetilde{\Psi}\colon\mathscr{C}_n\longrightarrow\FRS$ is surjective and  $\widetilde{\Psi}(A\gamma)=\widetilde{\Psi}(\gamma)$, for $A\in\textnormal{SL}_2(R)$.
\end{lemma}
\begin{proof}
First we prove that $\widetilde{\Psi}$ is surjective. Let $\mathbf{M}\in\FRS$. By Lemma~\ref{lem25} we can find a bi-infinite path $\gamma$ in $\mathscr{E}_R$ with vertices $v_i=(a_i,b_i)$ such that $m_{i,j}=a_jb_i-b_ja_i$ (here $\alpha=m_{1,0}=1$ because $\mathbf{M}$ is a semiregular frieze). Since $m_{i+n,i}=0$ we have $v_{i+n}=\lambda_i v_i$, for some $\lambda_i \in R^\times$, by Lemma~\ref{lem66}. Observe that 
\[
\lambda_i\lambda_{i+1}=\lambda_i\lambda_{i+1}(a_ib_{i+1}-b_ia_{i+1})=a_{i+n}b_{i+n+1}-b_{i+n}a_{i+n+1}=1,
\]
from which it follows that $\lambda_{i}=\lambda^{(-1)^i}$, where $\lambda=\lambda_0$. Therefore $\gamma\in \mathscr{C}_n$, so $\widetilde{\Psi}$ is indeed surjective.

The $\text{SL}_2(R)$-invariance property follows immediately from the $\text{SL}_2(R)$-invariance property of $\widetilde{\Phi}$ given in Lemma~\ref{lem78}.
\end{proof}

Next, mirroring the procedure from Section~\ref{section6}, we define $\tau_U\colon \FRS\longrightarrow U\backslash \FRS$ to be the map that takes a tame semiregular frieze $\mathbf{M}$ to its orbit under the action of $U$ on $\FRS$ (given by $m_{i,j}\longmapsto \lambda^{(-1)^i+(-1)^j}m_{i,j}$, for $\lambda\in U$). Then we define $\widetilde{\Psi}_U\colon\mathscr{C}_{n,U}\longrightarrow U\backslash\FRS$ by
\[
\widetilde{\Psi}_U(\pi_U(\gamma))=\tau_U\widetilde{\Psi}(\gamma),
\]
where $\mathscr{C}_{n,U}=\pi_U(\mathscr{C}_n)$. This is a well-defined function because if $\pi_U(\gamma)=\pi_U(\gamma')$ -- where $\gamma$ and $\gamma'$ have vertices $(a_i,b_i)$ and $\lambda^{(-1)^i}(a_i,b_i)$ for some $\lambda\in U$ (see Lemma~\ref{lem61}) -- then $\mathbf{M}=\widetilde{\Psi}(\gamma)$ and $\mathbf{M}'=\widetilde{\Psi}(\gamma')$ satisfy
\[
m_{i,j}'=a_j'b_i'-b_j'a_i'=\lambda^{(-1)^i+(-1)^j}m_{i,j}.
\]
Thus $\tau_U(\mathbf{M}')=\tau_U(\mathbf{M})$. 

Notice that $\widetilde{\Psi}_U(A\gamma)=\widetilde{\Psi}_U(\gamma)$, for $A\in\text{SL}_2(R)$, by Lemma~\ref{lem51} and the equivariance of $\pi_U$ under the action of $\text{SL}_2(R)$. It follows that $\widetilde{\Psi}_U$ induces a map
\[
\Psi_U\colon \text{SL}_2(R)\backslash \mathscr{C}_{n,U}\longrightarrow U\backslash\FRS,
\]
which is surjective since $\widetilde{\Psi}$ is surjective. To complete the proof of Theorem~\ref{thm4} we have only to show that $\Psi_U$ is injective.

Suppose then that $\widetilde{\Psi}_U(\pi_U(\gamma))=\widetilde{\Psi}_U(\pi_U(\gamma'))$, for $\gamma,\gamma'\in \mathscr{C}_n$ with vertices $(a_i,b_i)$ and $(a_i',b_i')$. Then  $\tau_U\widetilde{\Psi}(\gamma)=\tau_U\widetilde{\Psi}(\gamma')$. By replacing $(a_i',b_i')$ with $\lambda^{(-1)^i}(a_i',b_i')$, for some suitable unit $\lambda\in U$ (which preserves $\pi_U(\gamma')$), we can assume that in fact $\widetilde{\Psi}(\gamma)=\widetilde{\Psi}(\gamma')$. Then Lemma~\ref{lem45} tells us that $\gamma'=A\gamma$, for some $A\in\text{SL}_2(R)$, so $\pi_U(\gamma)$ and $\pi_U(\gamma')$ lie in the same $\text{SL}_2(R)$-orbit in $\mathscr{C}_{n,U}$. Therefore $\Psi_U$ is indeed injective. This completes the proof of Theorem~\ref{thm4}.

\section{Tame regular friezes}\label{section9}

In this section we prove Theorem~\ref{thm5}, which says that the map $\Psi$ induces a one-to-one correspondence between 
\[
\textnormal{SL}_2(R)\Big\backslash\mleft\{\parbox{3.1cm}{\centering\textnormal{semiclosed~paths~of length~$n$~in~$\mathscr{E}_R$}}\mright\}\quad \longleftrightarrow\quad \mleft\{\parbox{3.2cm}{\centering\textnormal{tame~regular~friezes over~$R$~of~width~$n$}}\mright\}.
\]
Let us denote by $\mathscr{S}_n$ the subcollection of $\mathscr{C}_n$ of semiclosed paths of length $n$. A semiclosed path of length $n$ considered as a bi-infinite path with vertices $v_i$ satisfies $v_{i+n}=-v_i$, for $i\in\mathbb{Z}$. We denote by $\FRT$ the subcollection of $\FRS$ of tame regular friezes of width $n$. By Lemma~\ref{lem51}, the map $\widetilde{\Psi}\colon\mathscr{C}_n\longrightarrow\FRS$ is surjective. Suppose now that $\gamma\in \mathscr{S}_n$ with vertices $v_i=(a_i,b_i)$. Then $\mathbf{M}=\widetilde{\Psi}(\gamma)$ is a tame semiregular frieze of width $n$ and
\[
m_{i+n-1,i}= a_ib_{i+n-1}-b_ia_{i+n-1}=-a_{i+n}b_{i+n-1}+b_{i+n}a_{i+n-1}=1.
\]
Hence $\mathbf{M}\in \FRT$. 

Conversely, suppose that  $\mathbf{M}\in \FRT$ and $\widetilde{\Psi}(\gamma)=\mathbf{M}$, where $\gamma\in\mathscr{C}_n$. Let $\gamma$ have vertices $v_i=(a_i,b_i)$, where $v_{i+n}=\lambda^{(-1)^i}v_i$, for some $\lambda\in R^\times$. Then, because $m_{n-1,0}=1$, we have 
\[
\lambda =\lambda m_{n-1,0} =  \lambda(a_0b_{n-1}-b_0a_{n-1})=a_{n}b_{n-1}-b_{n}a_{n-1}=-1.
\]
Hence $\gamma\in \mathscr{S}_n$. 

It follows that the restriction of $\Psi$ to $\text{SL}_2(R)\backslash \mathscr{S}_n$ is a surjective map from $\text{SL}_2(R)\backslash \mathscr{S}_n$ onto $\FRT$. Now, we know from Theorem~\ref{thm4} that the map $\Psi\colon \text{SL}_2(R)\backslash \mathscr{C}_n\longrightarrow \FRS$ is a bijection, so it follows that the restriction map $\Psi\colon \text{SL}_2(R)\backslash \mathscr{S}_n\longrightarrow \FRT$ is a bijection also. This completes the proof of Theorem~\ref{thm5}.

\section{Quiddity sequences of tame friezes}\label{section10}

We recall that a quiddity sequence for a tame semiregular frieze $\mathbf{F}$ is a period $a_1,a_2,\dots,a_k$ from the third row of $\mathbf{F}$. Here we prove Theorem~\ref{thm34}, which says that any finite sequence in a finite ring $R$ is the quiddity sequence for some tame semiregular frieze over $R$.

\begin{proof}[Proof of Theorem~\ref{thm34}]
Consider any finite sequence $e_1,e_2,\dots,e_k$ in a finite ring $R$. We extend this to a periodic bi-infinite sequence by defining $e_{k+i}=e_i$, for $i\in\mathbb{Z}$. Let $\gamma$ be any bi-infinite path in $\mathscr{E}_R$ with itinerary $(e_i)$. As usual, we denote the vertices of $\gamma$ by $v_i=(a_i,b_i)$.  From Lemma~\ref{lem85} we have that
\[
\begin{pmatrix}a_i &  a_{i+1}\\ b_i & b_{i+1}\end{pmatrix}=\begin{pmatrix}a_{i-1} &  a_{i}\\ b_{i-1} & b_{i}\end{pmatrix}U_i,\quad\text{where }U_i=\begin{pmatrix}0 & -1 \\ 1 & e_i\end{pmatrix}\!,
\]
for $i\in\mathbb{Z}$. Let $g=U_1U_2\dotsb U_k$ and let $m$ be a positive integer such that $g^m=I$, the identity matrix in (the finite group) $\text{SL}_2(R)$. We define $n=mk$; then $U_1U_2\dotsb U_n=I$ and, more generally, \(U_iU_{i+1}\dotsb U_{i+n-1}=I\), for any $i\in\mathbb{Z}$. It follows that 
\[
\begin{pmatrix}a_{i+n-1} &  a_{i+n}\\ b_{i+n-1} & b_{i+n}\end{pmatrix}=\begin{pmatrix}a_{i-1} &  a_{i}\\ b_{i-1} & b_{i}\end{pmatrix}\!,
\]
so $v_{i+n}=v_i$. As we saw in Section~\ref{section8}, the image of $\gamma$ under $\widetilde{\Psi}$ is a tame semiregular frieze $\mathbf{F}$ of width $n$ with entries $m_{i,j}=a_jb_i-b_ja_i$. Observe that
\[
m_{i+1,i-1}=a_{i-1}b_{i+1}-b_{i-1}a_{i+1}=e_i.
\]
Thus $e_1,e_2,\dots,e_k$ is a quiddity sequence for $\mathbf{F}$, as required.
\end{proof}

In contrast to Theorem~\ref{thm34}, it is not true that any finite sequence $e_1, e_2,\dots,e_k$ in a finite ring $R$ is the quiddity sequence for some tame \emph{regular} frieze over $R$. The simplest example to demonstrate this is the sequence with a single entry $-1$ in the ring $\mathbb{Z}/3\mathbb{Z}$. We explore this no further here and instead refer the reader to \cite{Mo2015}*{Theorem~1.15} for an alternative approach.

\section{Enumerating friezes over finite fields}\label{section11}

Here we prove Theorem~\ref{thm9}, which says that the number of tame friezes of width $n$ over a finite field $R$ of order $q$ is
\[
\frac{(q-1)(q^{n-1}+(-1)^n)}{q+1}.
\]
We use the Farey complex $\mathscr{F}_{R,R^\times}$, which is the complete graph on $q+1$ vertices. In this Farey complex any two equivalent vertices (vertices $u$ and $v$ with $v=\lambda u$, for $\lambda\in R^\times$) are in fact equal, so a path between equivalent vertices is a closed path.

\begin{proof}[Proof of Theorem~\ref{thm9}]
Let $R$ denote the finite field of order $q$. The number of closed paths of length $n$ in $\mathscr{F}_{R,R^\times}$ is $q(q^{n-1}+(-1)^n)$. This is a straightforward observation in graph theory, which can be proven by induction. Now, the group $\text{SL}_2(R)$ acts transitively on $\mathscr{F}_{R,R^\times}$, and the kernel of this action is $\{\pm I\}$. It is a well-known and elementary observation that this group has order $q(q^2-1$), so the quotient set
\[
\textnormal{SL}_2(R)\Big\backslash\mleft\{\parbox{2.7cm}{\centering\textnormal{closed~paths~of length~$n$~in~$\mathscr{F}_{R,R^\times}$}}\mright\}
\]
has cardinality $2(q^{n-1}+(-1)^n)/(q^2-1)$. By Theorem~\ref{thm4}, this is also the cardinality of the set
\[
R^\times\Big\backslash\mleft\{\parbox{3.85cm}{\centering\textnormal{tame~semiregular~friezes over~$R$~of~width~$n$}}\mright\}.
\]
The group $R^\times$ acts transitively on the collection of tame semiregular friezes of width $n$ (by the rule $m_{i,j}\longmapsto \lambda^{(-1)^i+(-1)^j}m_{i,j}$) and the kernel of this action is $\{\pm 1\}$. It follows that there are $(q^{n-1}+(-1)^n)/(q+1)$ tame semiregular friezes over $R$ of width $n$. Finally, as noted in the introduction, there is a $(q-1)$-to-1 map from the full collection of tame friezes of width $n$ to the subcollection of tame semiregular friezes given by $m_{i,j}\longmapsto \alpha^{(-1)^j}m_{i,j}$, where $\alpha=m_{1,0}$. Consequently, there are $(q-1)(q^{n-1}+(-1)^n)/(q+1)$ tame friezes over $R$ of width $n$, as required.
\end{proof}

Next we recover a result of Morier-Genoud \cite{Mo2021}*{Theorem~1} on the number of tame regular friezes of width $n$. Morier-Genoud's approach is to utilise a correspondence between tame regular friezes over a finite field and a certain moduli spaces of curves with marked points; this correspondence was developed by Morier-Genoud and others in \cite{MoOvScTa2014,MoOvTa2012}. Morier-Genoud's result is generalised by Cuntz and Mabilat in \cite{CuMa2023} who enumerate the solutions of matrix equations of the form
\[
\begin{pmatrix}a_n & -1\\ 1 & 0\end{pmatrix}\begin{pmatrix}a_{n-1} & -1\\ 1 & 0\end{pmatrix}\dotsb \begin{pmatrix}a_1 & -1\\ 1 & 0\end{pmatrix}=B,
\]
for $a_1,a_2,\dots,a_n\in R$ and $B\in\text{SL}_2(R)$, where $R$ is a finite field. The solutions of this equation when $B=-I$ (the negative of the identity matrix) correspond to quiddity sequences of tame regular friezes over $R$. There have been further advances on this topic by B\"ohmler and Cuntz \cite{BoCu2024}, Mabilat \cite{Ma2024}, and the authors (with Benzaira) \cite{BeShSoZa2024} who carry out similar enumerations for the ring of integers modulo $N$. Common to all the approaches (including our own) is the strategy of representing friezes by an algebraic or geometric model and deriving recurrence relations for enumerating friezes within that model.

Following Morier-Genoud, we frame the result using the notation of $q$-integers and $q$-binomial coefficients, with
\[
[m]_{q^2} = \frac{q^{2m}-1}{q^2-1}\quad\text{and}\quad \binom{m}{2}_{\!\!q} = \frac{(q^m-1)(q^{m-1}-1)}{(q-1)(q^2-1)}.
\]

\begin{theorem}\label{thm10}
The number $\alpha_n$ of tame regular friezes of width $n$ over the finite field $R$ of order $q$ is
\[
\alpha_n = 
\begin{cases}
[m]_{q^2}, &\text{if $n=2m+1$,}\\
(q-1)\displaystyle\binom{m}{2}_{\!\!q} &\text{if $n=2m$ with $m$ even and $\cha R\neq 2$,}\\[1pt]
(q-1)\displaystyle\binom{m}{2}_{\!\!q}+q^{m-1}, &\text{if $n=2m$ with $m$ odd or $\cha R=2$}.
\end{cases}
\]
\end{theorem}

To prove Theorem~\ref{thm10}, it is convenient to work with the Farey complex $\mathscr{E}_R$.

\begin{lemma}\label{lemZ}
The number $\mu(a,b)$ of paths of length 2 from $(a,b)$ to $(1,0)$ in $\mathscr{E}_R$, where $R$ is the finite field of order $q$, is given by 
\[
\mu(a,b) = 
\begin{cases}
1, &\text{if $b\neq 0$,}\\
q, &\text{if $(a,b)=(-1,0)$,}\\
0, &\text{otherwise.}
\end{cases}
\]
\end{lemma}
\begin{proof}
Any path of length 2 from $(a,b)$ to $(1,0)$ has the form
\[
(a,b) \to (\lambda,-1) \to (1,0),
\]
where $\lambda\in R$. If $b\neq 0$, then there is only one such path, with $\lambda=-(a+1)b^{-1}$. If $(a,b)=(-1,0)$, then there are $q$ such paths, since any value of $\lambda\in R$ yields a path. Otherwise, if $a\neq -1$ and $b=0$, then no value of $\lambda$ yields a path, so there are no paths of length 2 from $(a,b)$ to $(1,0)$.
\end{proof}

\begin{proof}[Proof of Theorem~\ref{thm10}]
For each $n=2,3,\dotsc$ we define four classes of paths $A_n$, $B_n$, $C_n$, and $D_n$ in $\mathscr{E}_R$ as follows:
\begin{description}[leftmargin=20pt,topsep=0pt,itemsep=0pt]
\item $D_n$ is the collection of all paths of length $n$ with initial vertex $(1,0)$ and second vertex $(0,1)$,
\item $C_n$ is the subcollection of $D_n$ of paths with final vertex $(a,0)$, for some $a\in R^\times$,
\item $B_n$ is the subcollection of $C_n$ of paths with final vertex $(1,0)$,
\item $A_n$ is the subcollection of $C_n$ of paths with final vertex $(-1,0)$.
\end{description}
The collection $D_n$ has cardinality $q^{n-1}$. Let $\lambda_n$ be the cardinality of $C_n$. Under the covering map $\pi_{R^\times}\colon \mathscr{E}_R\longrightarrow \mathscr{F}_{R,R^\times}$, each element of $C_n$ maps to a closed path in $\mathscr{F}_{R,R^\times}$. Also, each closed path in $\mathscr{F}_{R,R^\times}$ has precisely $q(q+1)$ preimages in $C_n$ under $\pi_{R^\times}$. We noted earlier that there are $q(q^{n-1}+(-1)^n)$ closed paths of length $n$ in $\mathscr{F}_{R,R^\times}$, so 
\[
\lambda_n=\frac{q^{n-1}+(-1)^n}{q+1}.
\]
Next, recall that a semiclosed path in $\mathscr{E}_R$ is a path with initial vertex $v$ and final vertex $-v$, for some vertex $v$ in $\mathscr{E}_R$. By Theorem~\ref{thm5} the number $\alpha_n$ of tame regular friezes of width $n$ over $R$ is equal to the cardinality of 
\[
\textnormal{SL}_2(R)\Big\backslash\mleft\{\parbox{3.1cm}{\centering\textnormal{semiclosed~paths~of length~$n$~in~$\mathscr{E}_R$}}\mright\}.
\]
Each semiclosed path of length $n$ in $\mathscr{E}_R$ is $\text{SL}_2(R)$-equivalent to precisely one element of $A_n$; hence $\alpha_n$ is the cardinality of $A_n$. Let $\beta_n$ denote the cardinality of $B_n$. 

Consider the map $A_n\longrightarrow D_{n-2}$ that sends a path $\langle v_0,v_1,\dots,v_n\rangle$ in $A_n$ to $\langle v_0,v_1,\dots,v_{n-2}\rangle$. By Lemma~\ref{lemZ}, each element of $D_{n-2} -  C_{n-2}$ has exactly one preimage under this map, each element of $B_{n-2}$ has exactly $q$ preimages, and all other elements have no preimages. Consequently,
\[
\alpha_n = (q^{n-3}-\lambda_{n-2})+q\beta_{n-2},
\]
and similarly one can see that
\[
\beta_n = (q^{n-3}-\lambda_{n-2})+q\alpha_{n-2}.
\]
Suppose now that $n$ is odd. In this case there is a bijection $A_n\longrightarrow B_n$ that sends
\[
(1,0)\to(0,1)\to(a_2,b_2)\to (a_3,b_3)\to (a_4,b_4)\to (a_5,b_5)\to\dotsb \to (-1,0)
\]
to
\[
(1,0)\to(0,1)\to(a_2,-b_2)\to (-a_3,b_3)\to (a_4,-b_4)\to (-a_5,b_5)\to\dotsb \to (1,0).
\]
Hence $\alpha_n=\beta_n$, so $\alpha_n=(q^{n-3}-\lambda_{n-2})+q\alpha_{n-2}$. The formula $\alpha_n=[m]_{q^2}$ can then be proven by induction.

Suppose instead that $n$ is even. By subtracting the recurrence relations for $\alpha_n$ and $\beta_n$ we obtain $\alpha_n-\beta_n=-q(\alpha_{n-2}-\beta_{n-2})$. Since $\alpha_2=q$ and $\beta_2=0$, we see that $\alpha_{2m}-\beta_{2m}=-(-q)^m$. From this we obtain the revised recurrence relation
\[
\alpha_{2m} = (q^{2m-3}-\lambda_{2m-2})+q\alpha_{2m-2}-(-q)^{2m-1},
\]
and the required formulas for $\alpha_{2m}$ can then be established by induction.
\end{proof}

We note that the recent work \cite{CuMa2023} enumerates certain tame friezes over finite fields and $\mathbb{Z}/N\mathbb{Z}$ for some values of $N$; the friezes they consider are termed `quasiregular' in Lemma~\ref{lem75}, to follow.


\section{Lifting $\text{SL}_2$-tilings}\label{section12}

In this section we consider a third and final application of Theorems~\ref{thm3} to~\ref{thm5}, on lifting $\text{SL}_2$-tilings and friezes. In particular, we prove Theorem~\ref{thm100}.

Let $R$ be a ring and $I$ an ideal in $R$; then $R/I$ is also a ring. The quotient map $R\longrightarrow R/I$ given by $a\longmapsto a+I$ induces a map $\theta\colon \text{SL}_2(R)\longrightarrow \text{SL}_2(R/I)$ by applying the quotient map to each entry, and it also induces a map $\Theta$ from the collection of tame $\text{SL}_2$-tilings over $R$ to the collection of tame $\text{SL}_2$-tilings over $R/I$. The first of three main results in this section follows.

\begin{theorem}\label{thm44}
For any ideal $I$ in a ring $R$, the map $\theta$ is surjective if and only if $\Theta$ is surjective.
\end{theorem}

The question of when $\theta$ is surjective in general is complex. It is surjective if $R$ is the ring of integers of a number field; on the other hand, there are examples such as  \cite{Mi1971}*{Example 13.5} for which the map is not surjective.

Let $\rho\colon \mathscr{E}_R\longrightarrow \mathscr{E}_{R/I}$ be the map with rule $(a,b)\longmapsto (a+I,b+I)$. It is straightforward to prove that $\rho$ is a graph homomorphism, in the sense that it maps vertices to vertices and preserves directed edges. It is not in general a covering map because it need not be locally bijective at vertices, and it may not be surjective. However, if $\theta$  is surjective then $\rho$ \emph{is} surjective. This is straightforward to establish, for if $u\to v$ is a directed edge in $\mathscr{E}_{R/I}$, then there is a matrix $A\in \textnormal{SL}_2(R/I)$ with rows $u$ and $v$. We can then find $\widetilde{A}\in \textnormal{SL}_2(R)$ with $\theta(\widetilde{A})=A$, in which case the rows $\tilde{u}$ and $\tilde{v}$ of $\widetilde{A}$ satisfy $\rho(\tilde{u})=u$ and $\rho(\tilde{v})=v$.

In Section~\ref{section5} we introduced principal congruence subgroups for $\text{SL}_2(\mathbb{Z})$; now we consider these groups for more general rings and ideals. The \emph{principal congruence subgroup} of $\text{SL}_2(R)$ for the ideal $I$ is the group
\[
\Gamma_I = \mleft\{\begin{pmatrix}a&b\\ c&d\end{pmatrix}\in \text{SL}_2(R) : a-1,b,c,d-1 \in I\mright\}. 
\]
It is a normal subgroup of $\text{SL}_2(R)$. It preserves the fibres of $\rho$, in the sense that the action of $\Gamma_I$ on $\mathscr{E}_R$ fixes each set $\rho^{-1}(v)$, for $v\in\mathscr{E}_{R/I}$. In fact, $\Gamma_I$ acts transitively on each fibre.

\begin{lemma}\label{lem49}
The group $\Gamma_I$ acts transitively on each fibre of $\rho$.
\end{lemma}
\begin{proof}
Suppose that $\rho(a,b)=\rho(c,d)$; we will prove that there exists an element of $\Gamma_I$ that maps $(a,b)$ to $(c,d)$.

Consider first the case $(a,b)=(1,0)$. Let us choose $x,y\in R$ such that $cx-dy=1$ and define
\[
A=\begin{pmatrix} c & y(1-c) \\ d & 1+x(1-c)\end{pmatrix}\!.
\]
Then $c-1,y(1-c),d,x(1-c)\in I$ and $\det A = c+(1-c)(cx-dy)=1$, so $A\in \Gamma_I$. Also, $A(1,0)=(c,d)$, as required.

For the general case, first choose $B\in \text{SL}_2(R)$ with $B(a,b)=(1,0)$. Then we can find $A\in\Gamma_I$ with $AB(a,b)=B(c,d)$, so $B^{-1}AB(a,b)=(c,d)$. Since $\Gamma_I$ is a normal subgroup of $\text{SL}_2(R)$ we have found the required element of $\Gamma_I$ that maps $(a,b)$ to $(c,d)$.
\end{proof}

We will use Lemma~\ref{lem49} to prove that bi-infinite paths lift from $\mathscr{E}_{R/I}$ to $\mathscr{E}_{R}$ under $\rho$.

\begin{lemma}\label{lem75}
Suppose that $\theta\colon \textnormal{SL}_2(R)\longrightarrow \textnormal{SL}_2(R/I)$ is surjective. Then for any bi-infinite path $\gamma$ in $\mathscr{E}_{R/I}$ there is a bi-infinite path $\tilde{\gamma}$ in $\mathscr{E}_{R}$ with $\rho(\tilde{\gamma})=\gamma$.
\end{lemma}
\begin{proof}
The hypothesis that $\theta$ is surjective guarantees that any directed edge in $\mathscr{E}_{R/I}$ lifts to a directed edge in  $\mathscr{E}_{R}$. Let $\gamma$ be a bi-infinite path in $\mathscr{E}_{R/I}$ with vertices $v_i$, for $i\in\mathbb{Z}$. We can choose a directed edge $\tilde{v}_0\to\tilde{v}_1$ in $\mathscr{E}_R$ with image $v_0\to v_1$ under $\rho$.  Suppose that we have found directed edges $\tilde{v}_{i-1}\to\tilde{v}_{i}$ in $\mathscr{E}_R$ with images $v_{i-1}\to v_i$ under $\rho$, for $i=1,2,\dots,n$. Let $w\to w'$ be some lift of $v_{n}\to v_{n+1}$. Since $\rho(\tilde{v}_n)=\rho(w)=v_n$, we can apply Lemma~\ref{lem49} to find $A\in\Gamma_I$ with $Aw=\tilde{v}_n$. We define $\tilde{v}_{n+1}=Aw'$. Then $\tilde{v}_{n}\to\tilde{v}_{n+1}$ has image $v_n\to v_{n+1}$ under $\rho$.

In this way, using the axiom of dependent choice, we can construct a path $\tilde{v}_0\to \tilde{v}_1\to \tilde{v}_{2}\to\dotsb$ in $\mathscr{E}_R$ with image $v_0\to v_1\to v_2\to\dotsb$ in $\mathscr{E}_{R/I}$. The path can be extended in the opposite direction in the same manner to give the required lift $\tilde{\gamma}$ of $\gamma$.
\end{proof}

The map $\rho\colon \mathscr{E}_R\longrightarrow \mathscr{E}_{R/I}$ induces a map $(\tilde{\gamma},\tilde{\delta})\longmapsto (\rho(\tilde{\gamma}),\rho(\tilde{\delta}))$ from the collection of pairs of bi-infinite paths in $\mathscr{E}_R$ to the collection of pairs of bi-infinite paths in $\mathscr{E}_{R/I}$; we also denote this map by $\rho$. We write $\widetilde{\Phi}_R$ and $\widetilde{\Phi}_{R/I}$ for the map $\widetilde{\Phi}$ (first introduced before Theorem~\ref{thm3}), for the rings $R$ and $R/I$, respectively. One can easily verify that the following diagram commutes.

\begin{tikzcd}[row sep=4.5em,column sep=5.5em]
\mleft\{\parbox{2.8cm}{\centering\textnormal{pairs~of~bi-infinite paths~in~$\mathscr{E}_{R}$}}\mright\}\arrow[r, "\widetilde{\Phi}_R",pos=0.5,shorten= 0ex,shorten >=0ex] \arrow[d,"\rho" left]
& \mleft\{\parbox{2.6cm}{\centering\textnormal{tame~$\text{SL}_2$-tilings over~$R$}}\mright\} \arrow[d, "\Theta"] \\
\mleft\{\parbox{2.8cm}{\centering\textnormal{pairs~of~bi-infinite paths~in~$\mathscr{E}_{R/I}$}}\mright\} \arrow[r,"\widetilde{\Phi}_{R/I}" below]
& \mleft\{\parbox{2.6cm}{\centering\textnormal{tame~$\text{SL}_2$-tilings over~$R/I$}}\mright\}
\end{tikzcd}

We are ready to prove Theorem~\ref{thm44}.

\begin{proof}[Proof of Theorem~\ref{thm44}]
Suppose first that $\theta\colon\textnormal{SL}_2(R)\longrightarrow \textnormal{SL}_2(R/I)$ is not surjective. Then there is a matrix $A\in\text{SL}_2(R/I)$ (with entries $a,b,c,d$) that does not lift to $\text{SL}_2(R)$ under $\theta$. Consider the following tame $\text{SL}_2$-tiling over $R/I$.
\begin{center}
$
  \vcenter{
  \xymatrix @-0.2pc @!0 {
          & &  &  & \vdots &  &  & &\\        
         & \phantom{-}a & \phantom{-}b& -a & -b & \phantom{-}a & \phantom{-}b & -a & \\
         & \phantom{-}c & \phantom{-}d& -c & -d & \phantom{-}c & \phantom{-}d & -c & \\
          & -a & -b& \phantom{-}a & \phantom{-}b & -a & -b & \phantom{-}a & \\
          \dotsb & -c & -d& \phantom{-}c & \phantom{-}d & -c & -d & \phantom{-}c &\dotsb \\
         & \phantom{-}a & \phantom{-}b& -a & -b & \phantom{-}a & \phantom{-}b & -a & \\
         & \phantom{-}c & \phantom{-}d& -c & -d & \phantom{-}c & \phantom{-}d & -c & \\
          & -a & -b& \phantom{-}a & \phantom{-}b & -a & -b & \phantom{-}a & \\
       &  &  &  &   \vdots &  &  & & 
            }
          }
$
\end{center}
Evidently this does not lift to a tame $\text{SL}_2$-tiling over $R$ under $\Theta$.

Suppose now that $\theta$ is surjective. Let $\mathbf{M}$ be a tame $\text{SL}_2$-tiling over $R/I$. Choose bi-infinite paths $\gamma$ and $\delta$ in $\mathscr{E}_{R/I}$  with $\widetilde{\Phi}_{R/I}(\gamma,\delta)=\mathbf{M}$. By Lemma~\ref{lem75} there are bi-infinite paths $\tilde{\gamma}$ and $\tilde{\delta}$ in $\mathscr{E}_{R/I}$ with $\rho(\tilde{\gamma},\tilde{\delta})=(\gamma,\delta)$. Now let $\widetilde{\mathbf{M}}=\widetilde{\Phi}_R(\tilde{\gamma},\tilde{\delta})$. Then
\[
\Theta(\widetilde{\mathbf{M}})=\Theta(\widetilde{\Phi}_R(\tilde{\gamma},\tilde{\delta}))=\widetilde{\Phi}_{R/I}(\rho(\tilde{\gamma},\tilde{\delta}))=\mathbf{M}.
\]
Hence $\Theta$ is surjective, as required.
\end{proof}

For the remainder of the paper we restrict our attention to the case $R=\mathbb{Z}$ and $I=N\mathbb{Z}$, where $N>1$, which we considered already in Section~\ref{section5}. It is well known, and easy to establish, that the map $\theta\colon\text{SL}_2(\mathbb{Z})\longrightarrow \text{SL}_2(\mathbb{Z}/N\mathbb{Z})$ is surjective, so Theorem~\ref{thm44} tells us that any tame $\text{SL}_2$-tiling over $\mathbb{Z}/N\mathbb{Z}$ lifts to a tame $\text{SL}_2$-tiling over $\mathbb{Z}$. In fact, we will prove that any  tame $\text{SL}_2$-tiling over $\mathbb{Z}/N\mathbb{Z}$ lifts to a \emph{positive integer} $\text{SL}_2$-tiling over $\mathbb{Z}$.
 
\begin{theorem}\label{thm99}
Given any tame $\textnormal{SL}_2$-tiling $\mathbf{M}$ over $\mathbb{Z}/N\mathbb{Z}$ there is a tame $\textnormal{SL}_2$-tiling $\widetilde{\mathbf{M}}$ with positive integer entries for which $\Theta(\widetilde{\mathbf{M}})=\mathbf{M}$.
\end{theorem}

In proving this theorem we write $\mathscr{E}_N$ in place of the more cumbersome $\mathscr{E}_{\mathbb{Z}/N\mathbb{Z}}$, just as we have been writing $\mathscr{F}_N$ in place of $\mathscr{F}_{\mathbb{Z}/N\mathbb{Z}}$.

\begin{lemma}\label{lem41}
Let $u\to v$ be a directed edge in $\mathscr{E}_N$, and let $(a,b)\in \mathscr{E}_\mathbb{Z}$ satisfy $\rho(a,b)=u$, $b>0$, and $-1<a/b<1$. Then we can find a directed edge $(a,b)\to (c,d)$ in $\mathscr{E}_\mathbb{Z}$ with $\rho(c,d)=v$, $d>0$, and $-1<c/d<a/b$.
\end{lemma}
\begin{proof}
Since $\theta\colon\text{SL}_2(\mathbb{Z})\longrightarrow \text{SL}_2(\mathbb{Z}/N\mathbb{Z})$ is surjective, we can certainly find a directed edge $\tilde{u}\to \tilde{v}$ in $\mathscr{E}_\mathbb{Z}$ that is mapped to $u\to v$ by $\rho$. From Lemma~\ref{lem49}, the principal congruence subgroup $\Gamma_N$ acts transitively on $\rho^{-1}(u)$, so there exists $A\in\Gamma_N$ with $A\tilde{u}=(a,b)$. Let $(x,y)=A\tilde{v}$. Then $(a,b)\to (x,y)$ is a directed edge in $\mathscr{E}_\mathbb{Z}$ and $\rho(x,y)=v$. Let us define $(c,d)=(x+kNa,y+kNb)$, where $k$ is a positive integer chosen to be sufficiently large that $d>0$ and $c/d>-1$. Observe that $(a,b)\to (c,d)$ is also a directed edge in $\mathscr{E}_\mathbb{Z}$ and $\rho(c,d)=v$. Furthermore,
\[
\frac{a}{b}-\frac{c}{d} = \frac{1}{bd}(ad-bc)=\frac{1}{bd}(ay-bx)=\frac{1}{bd}>0.
\]
Hence $(c,d)$ has all the required properties.
\end{proof}

From this we deduce the following stronger version of Lemma~\ref{lem75}.

\begin{lemma}\label{lem76}
For any bi-infinite path $\gamma$ in $\mathscr{E}_{N}$ there is a bi-infinite path $\tilde{\gamma}$ in $\mathscr{E}_{\mathbb{Z}}$ such that $\rho(\tilde{\gamma})=\gamma$ and such that the vertices $(a_i,b_i)$ of $\tilde{\gamma}$ satisfy $b_i>0$, for $i\in\mathbb{Z}$, and
\[
-1< \dotsb <\frac{a_2}{b_2}<\frac{a_1}{b_1}<\frac{a_0}{b_0}<\frac{a_{-1}}{b_{-1}}<\frac{a_{-2}}{b_{-2}}<\dotsb <1.
\]
\end{lemma}
\begin{proof}
It is straightforward to choose a preimage $(a_0,b_0)$ of the zeroth vertex of $\gamma$ such that $b_0>0$ and $-1<a_0/b_0<1$. We can then apply Lemma~\ref{lem41} and a recursive argument to construct suitable vertices $(a_i,b_i)$, for $i\geq 0$. Then, by changing the sign of $a_0$ and applying a similar argument, we can extend the sequence to $i\leq 0$.
\end{proof}

We can now prove Theorem~\ref{thm99}.

\begin{proof}[Proof of Theorem~\ref{thm99}]
We begin by choosing bi-infinite paths $\gamma$ and $\delta$ in $\mathscr{E}_N$ with $\widetilde{\Phi}_{\mathbb{Z}/N\mathbb{Z}}(\gamma,\delta)=\mathbf{M}$. By Lemma~\ref{lem76} there are paths $\tilde{\gamma}$ and $\tilde{\delta}$ in $\mathscr{E}_{\mathbb{Z}}$ such that $\rho(\tilde{\gamma},\tilde{\delta})=(\gamma,\delta)$ and such that the vertices $(a_i,b_i)$ and $(c_j,d_j)$ of $\tilde{\gamma}$ and $\tilde{\delta}$ satisfy $b_i,d_j>0$, for $i,j\in\mathbb{Z}$, and 
\[
-1<\dotsb <\frac{c_2}{d_2}<\frac{c_1}{d_1}<\frac{c_0}{d_0}<\frac{c_{-1}}{d_{-1}}<\frac{c_{-2}}{d_{-2}}<\dotsb <1
\]
and
\[
N-1< \dotsb <\frac{a_2}{b_2}<\frac{a_1}{b_1}<\frac{a_0}{b_0}<\frac{a_{-1}}{b_{-1}}<\frac{a_{-2}}{b_{-2}}<\dotsb <N+1.
\]
(Replace $(a_i,b_i)$ with $(a_i+Nb_i,b_i)$ for the second set of inequalities.)  Now let $\widetilde{\mathbf{M}}=\widetilde{\Phi}_\mathbb{Z}(\tilde{\gamma},\tilde{\delta})$. Then, as before, we have 
$\Theta(\widetilde{\mathbf{M}})=\mathbf{M}$. What is more, the entries $\tilde{m}_{i,j}$ of $\widetilde{\mathbf{M}}$ satisfy
\[
\tilde{m}_{i,j}=a_id_j-b_ic_j = b_id_j\mleft(\frac{a_i}{b_i}-\frac{c_j}{d_j}\mright)>0,
\]
as required.
\end{proof}

Theorems~\ref{thm44} and~\ref{thm99} tell us that every tame $\text{SL}_2$-tiling over $\mathbb{Z}/N\mathbb{Z}$ lifts to a tame $\text{SL}_2$-tiling over $\mathbb{Z}$ and in fact the lifted $\text{SL}_2$-tiling can be chosen with positive entries. We learned in the introduction that a similar statement cannot be made for friezes. Here we will prove Theorem~\ref{thm100}, which tells us exactly when a tame frieze over $\mathbb{Z}/N\mathbb{Z}$ lifts to a tame frieze over $\mathbb{Z}$.

For this purpose it is more convenient to use the Farey complexes $\mathscr{F}_N$ and $\mathscr{F}_\mathbb{Z}$ rather than $\mathscr{E}_N$ and $\mathscr{E}_\mathbb{Z}$, because  $\mathscr{F}_N$ is a surface complex, as we saw in Section~\ref{section5}, and $\mathscr{E}_N$ is not. We recall the notation $\bar{a}$ from Section~\ref{section5} for the class of integers congruent to $a$ modulo $N$. With this notation we have $\rho(a,b)=(\bar{a},\bar{b})$. Notice that $\rho$ maps $\pm(a,b)$ to $\pm(\bar{a},\bar{b})$, so we obtain a map from $\mathscr{F}_\mathbb{Z}$ to $\mathscr{F}_N$, which we also denote by $\rho$. This preserves edges and triangles. Furthermore, if there is an edge between $a/b$ and $c/d$ in $\mathscr{F}_\mathbb{Z}$ (where we have switched to fractional notation $a/b$ in place of $\pm(a,b)$), then the other vertices of the two triangles incident to this edge are $(a+b)/(c+d)$ and $(a-b)/(c-d)$. Under $\rho$ these are mapped to the vertices of the two triangles incident to the edge between $\rho(a/b)$ and $\rho(c/d)$. (Unless $N=2$, in which case they are mapped to the single common neighbour of $\rho(a/b)$ and $\rho(c/d)$.)

We recall from the introduction that a (finite) closed path in a Farey complex $\mathscr{F}_R$ is said to be strongly contractible if it can be transformed to a point by applying a finite number of the following two elementary homotopies. 
\begin{itemize}
\item[](E1)  Replace a subpath $\langle v,u,v\rangle$ with the subpath $\langle v \rangle$.
\item[](E2)  Replace a subpath $\langle u,v,w\rangle$ with $\langle u,w\rangle$, where $u$, $v$, and $w$ are mutually adjacent.
\end{itemize}
The property of being strongly contractible is preserved under the action of $\text{SL}_2(R)$, in the sense that if $\gamma$ is strongly contractible then so is $A\gamma$, where $A\in\text{SL}_2(R)$. Not all closed paths are strongly contractible, as Figure~\ref{fig94} demonstrates; however, they are in $\mathscr{F}_\mathbb{Z}$.

\begin{lemma}\label{lem52}
All closed paths in $\mathscr{F}_\mathbb{Z}$ are strongly contractible.
\end{lemma}
\begin{proof}
Suppose to the contrary that some closed paths in $\mathscr{F}_\mathbb{Z}$ are not strongly contractible. Then there exists a closed path $\gamma=\langle v_0,v_1,\dots,v_m\rangle$, where $m\geq 2$, to which neither (E1) nor (E2) can be applied. We will demonstrate that this leads to a contradiction.

It is convenient to assume that $v_0=\infty$, which can be achieved by applying an element of $\text{SL}_2(\mathbb{Z})$ to $\gamma$. 

Let us choose coprime pairs of integers $(a_i,b_i)$ with  $v_i=a_i/b_i$, for $i=1,2,\dots,m$ (and $(a_0,b_0)=(1,0)$). The neighbours in $\mathscr{F}_\mathbb{Z}$ of $a_i/b_i$ have the form $(\lambda a_i-a_{i-1})/(\lambda b_i-b_{i-1})$, for $\lambda\in\mathbb{Z}$. Hence $a_{i+1}=\lambda_ia_i-a_{i-1}$ and $b_{i+1}=\lambda_ib_i-b_{i-1}$, for some $\lambda_i\in\mathbb{Z}$, where $1\leq i<m$. It cannot be that $\lambda_i=0$, for if that were so then $v_{i+1}$ would equal $v_{i-1}$ and we could apply (E1) to $\gamma$. Nor can we have $\lambda_i=\pm 1$, for in these cases $v_{i-1}$, $v_i$, and $v_{i+1}$ are mutually adjacent, so we could apply (E2) to $\gamma$. Thus $|\lambda_i|\geq 2$. As a consequence,
\[
|b_{i+1}|=|\lambda_ib_i-b_{i-1}|\geq |b_i| + (|b_i|-|b_{i-1}|).
\]
From this inequality it follows that the sequence $|b_0|, |b_1|, \dotsc$ is increasing. Hence $b_m\neq 0$, so $v_m\neq v_0$, which is the desired contradiction. Therefore all closed paths in 
$\mathscr{F}_\mathbb{Z}$ are strongly contractible after all.
\end{proof}

Suppose that $\tilde{\gamma}_2$ is a closed path in $\mathscr{F}_\mathbb{Z}$ obtained from another closed path $\tilde{\gamma}_1$ by applying one of (E1) or (E2). Let $\gamma_1=\rho(\tilde{\gamma}_1)$ and $\gamma_2=\rho(\tilde{\gamma}_2)$. Since $\rho$ preserves edges and triangles, we can see that $\gamma_2$ is obtained from $\gamma_1$ by applying one of (E1) or (E2). From this observation we deduce the following lemma.

\begin{lemma}\label{lem53}
Let $\tilde{\gamma}$ be a closed path in $\mathscr{F}_\mathbb{Z}$ and let $\gamma=\rho(\tilde{\gamma})$. Then $\gamma$ is strongly contractible in~$\mathscr{F}_N$.
\end{lemma}
\begin{proof}
By Lemma~\ref{lem52}, $\tilde{\gamma}$ is strongly contractible, so there is a finite sequence of elementary homotopies from $\tilde{\gamma}$ to a single point. Since these elementary homotopies are preserved under $\rho$ we see that there is a finite sequence of elementary homotopies from $\gamma$ to a single point. Hence $\gamma$ is strongly contractible in $\mathscr{F}_N$.
\end{proof}

In fact the converse to this lemma holds.

\begin{lemma}\label{lem54}
Suppose that $\gamma$ is a strongly contractible closed path in $\mathscr{F}_N$. Then there exists a closed path $\tilde{\gamma}$ in $\mathscr{F}_\mathbb{Z}$ with $\rho(\tilde{\gamma})=\gamma$.
\end{lemma}
\begin{proof}
We prove this by induction on the length of $\gamma$. Observe first that any closed path of length $0$ (a single point) is strongly contractible and lifts to a closed path of length $0$ in $\mathscr{F}_\mathbb{Z}$. 

Suppose now that each strongly contractible closed path of length $k$ in $\mathscr{F}_N$ lifts to a closed path of length $k$ in $\mathscr{F}_\mathbb{Z}$, for $k=0,1,\dots,m$. Let $\gamma$ be a strongly contractible closed path of length $m+1$ in $\mathscr{F}_N$. Since $\gamma$ is strongly contractible, we can apply an elementary homotopy to $\gamma$ to obtain another strongly contractible closed path $\gamma'$, which has length either $m-1$ or $m$, depending on whether (E1) or (E2) was applied. Let us suppose that (E2) was applied (the other case is handled similarly). Then we can write $\gamma'=\langle v_0,v_1,\dots,v_m\rangle$, where $v_m=v_0$. The path $\gamma$ differs from $\gamma'$ in that there is an additional vertex $v$ inserted between $v_{j-1}$ and $v_j$, where $1\leq j\leq m$. By the induction hypothesis, there is a lift $\tilde{\gamma}'=\langle \tilde{v}_0,\tilde{v}_1,\dots,\tilde{v}_m\rangle$ of $\gamma'$ to $\mathscr{F}_\mathbb{Z}$, where $\tilde{v}_m=\tilde{v}_0$. Let $\tilde{v}$ be the unique vertex in $\mathscr{F}_\mathbb{Z}$ that is adjacent to both $\tilde{v}_{j-1}$ and $\tilde{v}_j$ and satisfies $\rho(\tilde{v})=v$ (in the exceptional case $N=2$ there are two choices for $\tilde{v}$). We define $\tilde{\gamma}$ to be the closed path obtained from $\tilde{\gamma}'$ by inserting $\tilde{v}$ between $\tilde{v}_{j-1}$ and $\tilde{v}_j$. Then $\rho(\tilde{\gamma})=\gamma$, and this completes the argument by induction.
\end{proof}

In Section~\ref{section8} we introduced the surjective map $\widetilde{\Psi}_U$ from the set $\mathscr{C}_{n,U}$ of paths of length $n$ between equivalent vertices in $\mathscr{F}_{R,U}$ to the set $U\backslash\FRS$, where $\FRS$ is the collection of tame semiregular  friezes over $R$ of width $n$, and where $U$ acts on $\FRS$ by the rule $m_{i,j}\longmapsto \lambda^{(-1)^i+(-1)^j}m_{i,j}$, for $\lambda \in U$. Here we restrict our attention to $U=\{\pm 1\}$, in which case the action of $U$ is trivial, so the image of $\widetilde{\Psi}_U$ is $\FRS$. We will describe a semiregular frieze as \emph{quasiregular} if its second-last row comprises all $1$s or all $-1$s (that is, $m_{i+n-1,i}$ is equal to (one of) 1 or $-1$ for $i\in\mathbb{Z}$). 

\begin{lemma}\label{lem86}
Let $\mathbf{F}=\widetilde{\Psi}_U(\gamma)$, for $U=\{\pm 1\}$, where $\gamma\in\mathscr{C}_{n,U}$ and $\mathbf{F}\in\FRS$. Then $\gamma$ is a closed path if and only if $\mathbf{F}$ is quasiregular. 
\end{lemma}
\begin{proof}
Let $\tilde{\gamma}$ be a lift of $\gamma$ to $\mathscr{C}_n$. Then $\tilde{\gamma}=\langle v_0,v_1,\dots,v_n\rangle$, where $v_i=(a_i,b_i)$ and $v_n=\lambda v_0$ for $\lambda\in R^\times$. Using the formula $v_{i+n}=\lambda^{(-1)^i}v_i$ for the extension of $\gamma$ to a bi-infinite path, we see that the second-last row of $\mathbf{F}$ has entries 
\[
m_{i+n-1,i}=a_ib_{i+n-1}-b_ia_{i+n-1}=\lambda^{(-1)^{i-1}}(a_ib_{i-1}-b_ia_{i-1})=-\lambda^{(-1)^{i-1}}.
\]
Therefore $\mathbf{F}$ is quasiregular if and only if $\lambda\in \{\pm 1\}$. The result follows, because $\lambda\in \{\pm 1\}$ if and only if $\gamma$ is a closed path in $\mathscr{F}_R$.
\end{proof}

We denote by $\widetilde{\Psi}_\mathbb{Z}$ and $\widetilde{\Psi}_N$ the two versions of the map $\widetilde{\Psi}_{\{\pm 1\}}$ for the rings $R=\mathbb{Z}$ and $R=\mathbb{Z}/N\mathbb{Z}$, respectively. With this terminology, it is straightforward to check that the following diagram commutes.

\begin{tikzcd}[row sep=4.5em,column sep=5.5em]
\mleft\{\parbox{2.4cm}{\centering\textnormal{closed~paths~of length~$n$~in~$\mathscr{F}_{\mathbb{Z}}$}}\mright\} \arrow[r, "\widetilde{\Psi}_\mathbb{Z}",pos=0.5,shorten= 0ex,shorten >=0ex] \arrow[d,"\rho" left]
& \mleft\{\parbox{3.95cm}{\centering\textnormal{tame~quasiregular~friezes over~$\mathbb{Z}$~of~width~$n$}}\mright\} \arrow[d, "\Theta"] \\
\mleft\{\parbox{2.4cm}{\centering\textnormal{closed~paths~of length~$n$~in~$\mathscr{F}_N$}}\mright\} \arrow[r,"\widetilde{\Psi}_{N}" below]
& \mleft\{\parbox{3.8cm}{\centering\textnormal{tame~quasiregular~friezes over~$\mathbb{Z}/N\mathbb{Z}$~of~width~$n$}}\mright\}
\end{tikzcd}

We are now ready to prove Theorem~\ref{thm100}.

\begin{proof}[Proof of Theorem~\ref{thm100}]
Suppose first that $\mathbf{F}$ is a tame semiregular frieze of width $n$ over $\mathbb{Z}/N\mathbb{Z}$ that lifts to a tame frieze $\widetilde{\mathbf{F}}$ of width $n$ over $\mathbb{Z}$. Then $\widetilde{\mathbf{F}}$ is quasiregular, since it is defined over $\mathbb{Z}$. By Lemma~\ref{lem86}, we can find a closed path $\tilde{\gamma}$ of length $n$ in $\mathscr{F}_\mathbb{Z}$ with $\widetilde{\Psi}_\mathbb{Z}(\tilde{\gamma})=\widetilde{\mathbf{F}}$. Let $\gamma=\rho(\tilde{\gamma})$. Then $\gamma$ is strongly contractible, by Lemma~\ref{lem53}. Also, we have
\[
\widetilde{\Psi}_N(\gamma)=\widetilde{\Psi}_N(\rho(\tilde{\gamma}))=\Theta(\widetilde{\Psi}_\mathbb{Z}(\tilde{\gamma}))=\mathbf{F}.
\]
Thus $\gamma$ is a strongly contractible closed path corresponding to $\mathbf{F}$. Furthermore, any other path corresponding to $\mathbf{F}$ is also strongly contractible and closed because these properties are preserved under the action of $\text{SL}_2(\mathbb{Z}/N\mathbb{Z})$.

For the converse, suppose that $\gamma$ is a strongly contractible closed path corresponding to a tame semiregular frieze $\mathbf{F}$. Then $\widetilde{\Psi}_N(\gamma)=\mathbf{F}$.  By Lemma~\ref{lem54} there is closed path $\tilde{\gamma}$ in $\mathscr{F}_\mathbb{Z}$ with $\rho(\tilde{\gamma})=\gamma$. Let $\widetilde{\mathbf{F}}=\widetilde{\Psi}_\mathbb{Z}(\tilde{\gamma})$. Then
\[
\Theta(\widetilde{\mathbf{F}})=\Theta(\widetilde{\Psi}_\mathbb{Z}(\tilde{\gamma}))=\widetilde{\Psi}_N(\rho(\tilde{\gamma}))=\mathbf{F},
\]
so $\mathbf{F}$ does indeed lift to a tame frieze $\widetilde{\mathbf{F}}$ of width $n$ over $\mathbb{Z}$, as required.
\end{proof}


\end{document}